\documentclass{article}

\usepackage[preprint,nonatbib]{neurips_2023}
\usepackage{fdk-math}

\usepackage[biblatex-style=authoryear-comp]{fdk-cite}
\usepackage{fdk-typesetting}
\usepackage[utf8]{inputenc} 
\usepackage[T1]{fontenc}    
\usepackage{enumitem}    
\usepackage{hyperref}  
\usepackage{url}            
\usepackage{amsfonts}       
\usepackage{nicefrac}       
\usepackage{mdframed}
\usepackage{wrapfig}
\usepackage{xcolor}         
\usepackage{caption}
\usepackage{subcaption}
\usepackage{etoc}
\usepackage{suffix}

\usepackage{vsp-notation}

\usepackage{tikz}
\newcommand{\tikzcircle}[2][red,fill=red]{\tikz[baseline=-0.5ex]\draw[#1,radius=#2] (0,0) circle ;}%
\definecolor{ptblue}{RGB}{0, 68, 136}

\newcommand{\renewmathenvironment}[1]{%
  \renewenvironment{#1}{%
    \begin{base#1}%
  }{%
    \end{base#1}%
  }%
}
\newcommand{\newmathenvironment}[1]{%
  \newenvironment{#1}{%
    \begin{base#1}%
  }{%
    \end{base#1}%
  }%
}

\renewmathenvironment{proof}
\newmathenvironment{lemma}
\newmathenvironment{requirement}
\newmathenvironment{assumption}
\newmathenvironment{proposition}
\newmathenvironment{theorem}
\newmathenvironment{definition}
\newmathenvironment{remark}

\DeclareMathOperator{\Vol}{Vol}

\newcommand{\Rnn}{\R_{\raisebox{1pt}{$\scriptscriptstyle \geq 0$}}}
\newcommand{\Rplus}{\R_{\raisebox{1pt}{$\scriptscriptstyle > 0$}}}

\newcommand{\invalid}{\hyperref[def:invalid]{invalid}}

\newcommand{\sufficientprogress}{\hyperref[eq:sufficient-progress]{sufficient progress \eqref{eq:sufficient-progress}}}
\newcommand{\sufficientprogressnl}{sufficient progress}

\newlist{alglisting}{itemize}{10}
\setlist[alglisting]{label={},leftmargin=0pt}
\newlist{listing}{itemize}{10}
\setlist[listing]{label={},leftmargin=1.1em}
\newlist{apxlisting}{itemize}{10}
\setlist[apxlisting]{label={},leftmargin=2.0em}

\newcommand{\mPstar}{\mP_{\!\!*}}
\newcommand{\mPt}{\mP_{\!t}}

\title{
Searching for Optimal Per-Coordinate Step-sizes\\ with Multidimensional Backtracking
}

\author{%
  \iftrue
  Frederik Kunstner  $\quad$  Victor S. Portella  $\quad$  Mark Schmidt${}^\dagger$  $\quad$  Nick Harvey   \\
  $\{$\texttt{kunstner,victorsp,schmidtm,nickhar}\}\texttt{@cs.ubc.ca} \\
  {University of British Columbia}  $\quad$  $\quad$  {Canada CIFAR AI Chair (Amii)${}^{\dagger}$}  \\
  \else
  Frederik Kunstner \\  %
  {University of British Columbia}\\
  \texttt{kunstner@cs.ubc.ca} \\
  \And
  Victor S. Portella \\  %
  {University of British Columbia}\\
  \texttt{victorsp@cs.ubc.ca} \\
  \And
  Mark Schmidt \\  %
  {University of British Columbia}\\
  {Canada CIFAR AI Chair (Amii)}\\
  \texttt{schmidtm@cs.ubc.ca} \\
  \And
  Nick Harvey \\  %
  {University of British Columbia}\\
  \texttt{nickhar@cs.ubc.ca} \\
  \fi
}

\DeclareUnicodeCharacter{0301}{*************************************}

\AtEveryBibitem{%
  \clearlist{address}
  \clearfield{eprint}
  \clearfield{isbn}
  \clearfield{issn}
  \clearlist{location}
  \clearfield{month}
  \clearfield{series}
 \ifentrytype{book}{}{%
  \clearlist{publisher}
  \clearname{editor}
 }
 \ifentrytype{inproceedings}{
  \clearfield{volume}
  }{
 }
}

\begin{document}

\maketitle

\begin{abstract}
The backtracking line-search is an effective technique 
to automatically tune the step-size in smooth optimization.
It guarantees similar performance to using the theoretically optimal step-size.
Many approaches have been developed to instead tune 
per-coordinate step-sizes, also known as diagonal preconditioners,
but none of the existing methods 
are provably competitive with the optimal per-coordinate step-sizes.
We propose \emph{multidimensional backtracking}, 
an extension of the backtracking line-search 
to find good diagonal preconditioners for smooth convex problems. %
Our key insight is that the gradient with respect to the step-sizes, 
also known as hypergradients, yields separating hyperplanes 
that let us search for good preconditioners using cutting-plane methods.
As black-box cutting-plane approaches like the ellipsoid method are computationally prohibitive, 
we develop an efficient algorithm tailored to our setting. %
Multidimensional backtracking is provably competitive with the best
diagonal preconditioner and requires no manual tuning. 
\end{abstract}

\iffalse
When training machine learning models, tuning the hyperparameters of the
optimizer is often a major pain point. 
For example, finding a reasonable step-size hyperparameter for gradient descent 
typically involves trial-and-error or a costly grid search. 
However, only so much improvement can be achieved by tuning the step-size. 
For ill-conditioned problems, 
per-coordinate step-sizes---or diagonal preconditioning---can 
drastically improve convergence.
This has motivated a wide variety of approaches
specific to machine learning problems
to \emph{adaptively} find per-coordinate step-sizes.
However, the community lacks clear definitions of what 
\emph{adaptive step-sizes are}.

The most well-known definition of adaptivity comes from online learning,
where methods such as AdaGrad adapt to problem specific-constants without user input
while maintaining strong guarantees, even in the adversarial setting. 
However, this resilience to adversaries is a double-edged sword.
To satisfy this definition of adaptivity, AdaGrad monotonically decreases its step-sizes.
As a result, it performs poorly on non-adversarial problems,
and many follow-up methods have focused on working around this decreasing property.
Methods commonly used in deep learning,
such as RMSProp and Adam, are often motivated by analogy to AdaGrad,
but without decreasing step-sizes. 
This change is crucial for their practical performance 
but nullifies their adaptivity guarantees,
indicating that the online-learning definition of adaptivity might not capture what we want it to.

Alternative approaches to tune per-coordinate step-sizes 
during the course of optimization, 
such as adaptive gain and hypergradient methods, 
do not have a formal definition of what they are aiming to achieve,
and are instead motivated from intuition alone.
While showing promising practical performance in some settings, 
hypergradient methods are often unstable and can require as much babysitting 
as the original optimizer they are tuning.
The lack of a well-defined objective makes comparison of those methods purely empirical, 
and the field lacks direction on how to improve on the state-of-the-art.

However, there is an alternative definition of adaptivity in smooth optimization,
where the standard approach to tuning the step-size is to do a backtracking line-search. 
Applied to gradient descent, this line-search guarantees that the step-size
is within a constant factor of the best theoretical step-size.
But this method is often overlooked as it only capture adaptivity of a single step-size.

\textbf{Contribution.}
We propose \emph{multidimensional backtracking}, a method analogous to
a backtracking line-search that automatically finds
per-coordinate step-sizes while running gradient descent.
The main difficulty in extending the line-search to higher dimensions 
is that the signal used to search for a good scalar step-size, 
that the step-size is ``too big'', is insufficient to
efficiently search over per-coordinate step-sizes. 
Our key insight is that the gradient with respect to the step-sizes can be used in
conjunction with a cutting-plane method to make this search feasible in high dimensions,
and we develop a cutting-plane method tailored to the problem 
with minimal overhead. %
We show that our method has a similar guarantee as line-search, in that
its convergence rate is within a \(\smash{O(1/\sqrt{d})}\) factor of
preconditioned gradient descent with the optimal, but unknown, diagonal
preconditioner.
\else

\section{Introduction}
When training machine learning models, tuning the hyperparameters of the
optimizer is often a major challenge. 
For example, finding a reasonable step-size hyperparameter for gradient descent typically
involves trial-and-error or a costly grid search. 
In smooth optimization, a common approach to set the
step-size without user input is a backtracking line-search: 
start with a large step-size, and decrease it when it is too big to make sufficient progress.
For ill-conditioned problems, however, 
there are limits to the improvement achievable by tuning the step-size.
Per-coordinate step-sizes%
---also known as diagonal preconditioners---%
can drastically improve performance.
Many approaches have been developed to automatically tune per-coordinate step-sizes.
Those are often described as ``adaptive'' methods,
but the meaning of this term varies widely,
from describing heuristics that set per-coordinate step-sizes,
to ensuring performance guarantees as if a particular property of the problem were known in advance.
Yet, even on the simplest case of a smooth and strongly convex deterministic problem
where a good fixed diagonal preconditioner exists 
(i.e., one that reduces the condition number),
none of the existing adaptive methods are guaranteed to 
find per-coordinate step-sizes that improve the convergence rate.
We discuss approaches to adaptive methods in the next section.

\textbf{Contribution.}
We propose \emph{multidimensional backtracking}, 
an extension of the standard backtracking line-search to higher dimension, 
to automatically find good per-coordinate step-sizes. 
Our method recovers the convergence rate 
of gradient descent with the \emph{optimal preconditioner} for the problem,
up to a $\smash{\sqrt{2d}}$ factor where $d$ is the number of coordinates.
This is a direct generalization of the line-search guarantee, 
with a penalty depending on dimension due to the extra degrees of freedom, as expected.

\iffalse
The main difficulty in extending the line-search to higher dimensions 
is that when searching for a scalar step-size, 
all we can check is whether the step-size is ``too big''. 
This is insufficient to efficiently search over per-coordinate step-sizes. 
Our key insight is that the gradient 
with respect to the step-sizes 
can be used in conjunction with a cutting-plane method 
to make this search feasible in high dimensions.
\fi

\fi 

\subsection{Adaptive step-sizes and preconditioning methods}

\newcommand{\putfigurehere}{%
  \begin{wrapfigure}{r}{0.38\textwidth}
  \vspace{-1.5em}
  \includegraphics[width=0.31\textwidth]{script_plots/adagrad_has_issues.pdf}
  \vspace{-1em}
  \caption{
  {\bf Adaptive methods from online learning are overly pessimistic and can be slow on simple problems.}
  Experimental details in \cref{apx:experimental-details}.
  }%
  \label{fig:adagrad-has-issues}%
  \vspace{-1.25em}
\end{wrapfigure}%
}

\textbf{Adaptive and parameter-free methods in online learning}
are an example where \emph{adaptive methods} have a well-defined
meaning.
AdaGrad \citep{mcmahan2010adaptive,duchi2011adagrad} 
and Coin Betting \citep{orabona2016coinbetting,orabona2017cocob} 
can adapt to problem-specific constants without user input
and have strong guarantees, even in the \emph{adversarial} setting. 
However, this resilience to adversaries is a double-edged sword;
to satisfy this definition of adaptivity, 
AdaGrad uses monotonically decreasing step-sizes.
While AdaGrad still converges at the desired asymptotic rate
on smooth, Lipschitz functions \citep[][]{ward2019adagradstepsizes,li2019convergence},
its performance can be worse than plain gradient descent.
This motivated investigations of workarounds 
to avoid the monotonically decreasing updates, including 
augmenting the update with an increasing step-size schedule~\citep{agarwal2020disentangling},
a line-search~\citep{vaswani2020adaptive},
or modifying the update to the preconditioner~\citep{defazio2022gradagrad}.
Methods commonly used in deep learning,
such as RMSProp and Adam~\citep{hinton2012rmsprop,kingma2015adam},
are often motivated as \emph{adaptive} by analogy to AdaGrad, 
but without decreasing step-sizes \citep[e.g., ][\S4.3]{defossez2022asimple}. 
This change is crucial for their practical performance,
but nullifies their online-learning adaptivity guarantees.

\textbf{Adaptive gain and hypergradient heuristics.}
Many heuristics that tune
the hyperparameters of the optimization procedure 
use the gradient with respect to the hyperparameters, 
or \emph{hypergradients} \citep{maclaurin2015gradient}.
Methods have been proposed
to tune the step-size \citep{masse2015speed}, 
a preconditioner \citep{moskovitz2019firstorder}, 
any hyperparameter \citep{baydin2018online}, 
or to maintain a model of the objective \citep{bae2022amortized}.
``Stacking'' such optimizers recursively has been shown 
to reduce the dependency on user-specified hyperparameters in practice~\citep{chandra2022gradient}. 
This idea pre-dates the hypergradient nomenclature;
\citet{kesten1958accelerated} presents 
a method to update the step-size based on the sign of successive gradients, 
and \citet{saridis1970learning} presents a control perspective for per-coordinate step-sizes, 
which can be cast as a hypergradient update to a diagonal preconditioner.\footnote{%
The hypergradient with respect to a diagonal preconditioner
$\mP=\Diag(\vp)$ is, by the chain rule, the element-wise product ($\odot$) of subsequent
gradients, $- \nabla_{\vp} f(\vx - \Diag(\vp) \nabla f(\vx)) = \nabla
f(\vx) \odot \nabla f(\vx - \Diag(\vp) \nabla f(\vx))$. } 
This approach has led to \emph{adaptive gain} methods
such as Delta-Delta and variants~\citep{%
barto1981goal,jacobs1988increased,silva1990acceleration,sutton1992adapting,sutton1992gain},
and further developed using
the sign of the hypergradient~\citep{riedmiller1993rprop},
full-matrix updates~\citep{almeida1999adaptation},
a larger history~\citep{plagianakos2001learning},
updates in log-space~\citep{schraudolph1999local,schraudolph2005fast},
heuristics to adjust the outer step-size~\citep{mahmood2012tuningfree},
or multiplicative weight updates~\citep{amid2022stepsize}.
While showing promising practical performance in some settings, 
existing methods are often motivated from intuition rather than a formal definition of adaptivity,
giving no guarantee that the tuned method will converge faster, if at all.
Indeed, hypergradient methods are often unstable,
and may require as much manual tuning as the original optimizer they are intended to tune.

\textbf{Second-order methods.}
A classical approach to preconditioning is to use second-order information, 
as in Newton's method or its regularized variants \citep[e.g.,][]{nesterov2006cubic}.
To avoid the load of computing and inverting the Hessian, 
quasi-Newton methods \citep{dennis1977quasi} such as L-BFGS \citep{liu1989lbfgs}
fit an approximate Hessian using the secant equation.
Variants using diagonal approximations have also been proposed, 
framed as Quasi-Cauchy, diagonal BFGS, or diagonal Barzilai-Borwein methods%
~\citep{zhu1999quasicauchy,andrei2019diagonalbfgs,park2020variable},
while other methods use the diagonal of the Hessian \citep{lecun2012efficient,yao2021adahessian}.
Some second-order and quasi-Newton methods converge super linearly
(although not the diagonal or limited memory variants used in practice),
but those guarantees only hold locally when close to the minimum. 
To work when far from a solution, those methods require
``globalization'' modifications, such as regularization or a line-search. 
Unfortunately, analyses of second-order methods 
do not capture the global benefit of preconditioning and instead
lead to worse rates than gradient descent, as in the results of 
\citet[Cor. 3.3]{byrd2016stochastic}, %
\citet[Thm. 3.1]{bollapragada2018progressive},
\citet[Thm. 1]{meng2020fast},
\citet[Apx.]{yao2021adahessian},
\citet[Thm. 5.2]{berahas2022quasi},
or 
\citet[Thm. 4.9]{jahani2022doubly}.

\textbf{Line-searches.}
Adaptivity in smooth optimization is most closely related to line-searches.
The standard guarantee for gradient descent on a \(L\)-smooth function 
requires a step-size of \(1/L\), but $L$ is typically unknown. 
The backtracking line-search based on the Armijo condition \citep{armijo1966minimization}
approximately recovers this convergence guarantee by starting with a large step-size, and backtracking; 
halving the step-size whenever it does not yield sufficient improvement.
However, line-searches are often overlooked in the discussion of adaptive methods, 
as they do not provide a way to set more than a scalar step-size.
While line-searches can be shown to work in the stochastic overparameterized setting 
and have been applied to train neural networks \citep{vaswani2019painless}, 
improvements beyond backtracking have been limited. 
Additional conditions \citep{wolfe1969convergence}, 
non-monotone relaxations \citep{grippo1986nonmonotone}, 
or solving the line-search to higher precision \citep{more1994line} 
can improve the performance in practice,
but even an exact line-search cannot improve the convergence rate beyond
what is achievable with a fixed step-size \citep{deKlerk2017}. 

\subsection{Summary of main result: adaptivity to the optimal preconditioner}
Our approach is inspired by the work discussed above,
but addresses the following key limitation: 
none of the existing methods attain better global convergence rates than a backtracking line-search. 
Moreover, this holds even on smooth convex problems for which a good preconditioner exists.

We generalize the backtracking line-search
to handle per-coordinate step-sizes 
and find a good preconditioner.
As in quasi-Newton methods, we build a preconditioner based on first-order information.
However, instead of trying to approximate the Hessian using past gradients, 
our method searches for a preconditioner that minimizes the objective function at the next step.
Our convergence result depends on the best rate achievable by an \emph{optimal diagonal preconditioner}, 
similarly to how methods in online learning are competitive against the best preconditioner in hindsight.
However, our notion of optimality is tailored to smooth strongly-convex problems 
and does not require decreasing step-sizes as in AdaGrad.
Our update to the preconditioner can be interpreted as a hypergradient method, 
but instead of a heuristic update, 
we develop a cutting-plane method that uses hypergradients to guarantee a good diagonal preconditioner.
Our main theoretical contribution is summarized below.
\begin{theorem}[Informal]
\label{thm:thm1}
On a smooth, strongly-convex function $f$ in $d$ dimensions,
steps accepted by multidimensional backtracking guarantee the following progress
\aligns{
  f(\vx_{t+1}) - f(\vx_*) \leq \paren{1-\frac{1}{\sqrt{2d}}\frac{1}{\kappa_*}} \big(f(\vx_t) - f(\vx_*)\big),
}
where $\kappa_*$ is the condition number achieved by
the optimal preconditioner defined in \Cref{sec:optimal}.
The number of backtracking steps is at most linear in $d$ and logarithmic in problem-specific constants.
\end{theorem}

\begin{figure}
\centering
\includegraphics[width=\textwidth]{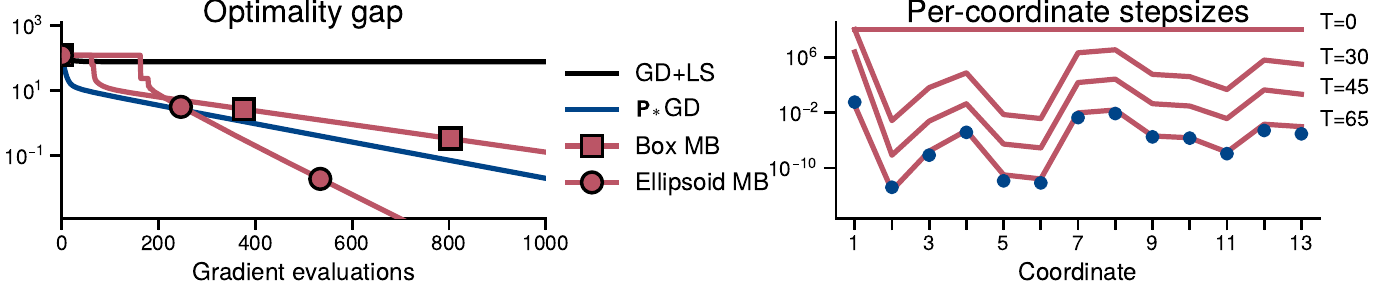}%
\caption{\textbf{Multidimensional backtracking can find the optimal diagonal preconditioner.}
Example on a linear regression where the optimal preconditioner can be computed.
{\bf Left:} Performance
of Gradient Descent (GD), 
optimally preconditioned GD ($\mPstar$GD) with a line-search (+LS),
and Multidimensional Backtracking (MB)
with the strategies in \cref{sec:cutting-plane}.
The ellipsoid variant can outperform the \emph{globally} optimal preconditioner 
by selecting preconditioners that leads to more \emph{local} progress.
{\bf Right:}
Optimal per-coordinate step-sizes (\tikzcircle[ptblue!70 , fill=ptblue, draw = none]{1.5pt}) and the ones found by MB (box) across iterations.
}
\vspace{-.5em}
\label{fig:it-works}
\end{figure}

Multidimensional backtracking finds per-coordinate step-sizes
that lead to a provable improvement over gradient descent on badly
conditioned problems that can be improved by diagonal preconditioning,
i.e., if the condition number of $f$ is at least
$\smash{\sqrt{2d} \cdot \kappa_*}$.  
Moreover, this guarantee is worst-case,
and multidimensional backtracking can outperform the globally optimal
preconditioner by finding a better \emph{local} preconditioner, as
illustrated on an ill-conditioned linear regression problem in
\Cref{fig:it-works}.

To find a competitive diagonal preconditioner, 
we view backtracking line-search as a cutting-plane method
and generalize it to higher dimensions
in \Cref{sec:convergence}.
In \Cref{sec:direction} we show how to use hypergradients
to find separating hyperplanes in the space of preconditioners, 
and in \Cref{sec:cutting-plane} we develop an efficient cutting-plane methods tailored to the problem.
In \Cref{sec:experiments}, 
we illustrate the method through preliminary experiments 
and show it has consistent performance across problems.

{\bf Notation.}
We use standard font weight $d$, $n$, $\alpha$ for scalars,
bold $\vx$, $\vy$ for vectors, and capital bold $\mP$, $\mA$ for matrices.
We use \(\vp[i]\) for the \(i\)-th entry of \(\vp\),
$\odot$ for element-wise multiplication, 
and $\vp^2$ for $\vp \odot \vp$.
$\odot$ for element-wise multiplication, 
and $\vp^2$ for $\vp \odot \vp$.
We use $\mP = \Diag(\vp)$ to denote 
the diagonal matrix with diagonal \(\vp\), 
and $\vp = \diag(\mP)$ to denote the vector of diagonal entries of $\mP$. 
We say \(\mA\) is larger than \(\mB\), 
\(\mA \succeq \mB\), if \(\mA - \mB\) is positive semidefinite. 
If $\mA = \Diag(\va)$, $\mB = \Diag(\vb)$,
the ordering \(\mA \succeq \mB\) is equivalent to \(\va[i] \geq \vb[i]\) for all $i$, 
which we write $\va \geq \vb$.
We use $\mI$ for the identity matrix 
and $\vone$ for the all-ones vector.

\section{Optimal preconditioning and sufficient progress}
\label{sec:optimal}
\label{sec:multidimensional-backtracking}
\label{sec:alg-design}

Consider a twice-differentiable
function $f \colon \R^d\to\R$ 
that is $L$-smooth and $\mu$-strongly convex,\footnote{%
While we use strong-convexity and twice-differentiability of $f$ to define the optimal preconditioner, 
those assumptions can be relaxed to only rely on the PL inequality
\citep{Polyak63a, Lojasiewicz63a} (see \cref{apx:without-sc}).} 
i.e.,
\aligns{%
  \textstyle
  \mu \frac{1}{2}\norm{\vx - \vy}^2 \leq f(\vy) - f(\vx) - \lin{\nabla f(\vx), \vy - \vx }
  \leq L \frac{1}{2}\norm{\vy - \vx}^2,
  \quad \text{ for all } \vx, \vy,
}%
or $\mu \mI \preceq \nabla^2 f(\vx) \preceq L \mI$ for all $\vx$.
We measure the quality of a preconditioner $\mP$ by how tightly it approximates \((\nabla^2 f(\vx))^{-1}\). %
We define an \emph{optimal diagonal preconditioner} for $f$ as
\alignn{%
  \textstyle
  \mPstar \in \argmin_{\mP\succeq0, \text{diagonal}}\, \kappa 
  \qquad \text{ such that } \qquad
  \frac{1}{\kappa}\mP^{-1}  \preceq \nabla^2 f(\vx)\preceq \mP^{-1}
  \text{ for all } \vx,
  \label{eq:optimal-preconditioner-def}
} 
and denote by $\kappa_*$ the optimal \(\kappa\) above.
\cref{eq:optimal-preconditioner-def}
is equivalent to minimizing 
$\smash{\kappa(\mP^{1/2}\nabla^2f(\vx)\mP^{1/2})}$,
a known measure of the convergence rate 
of preconditioned methods \citep[][\S1.3.2]{bertsekas1997nonlinear},
and reduces to the definition of optimal preconditioning for linear systems 
\citep{jambulapati2020fast,qu2022optimal} when $f$ is quadratic. 
Alternatively, the optimal preconditioner can be viewed as the matrix $\mPstar$ 
such that $f$ is $1$-smooth and maximally
strongly-convex in the norm $\norm*{\vx}^2{}_{\!\!\!\mPstar^{-1}} =
\lin{\vx,\mPstar^{-1}\vx}$, %
\alignn{%
  \textstyle
  \smash{
  \frac{1}{\kappa_*} \frac{1}{2}\norm{\vx - \vy}^2_{\mPstar^{-1}} \leq
  f(\vy) - f(\vx) - \lin{\nabla f(\vx), \vy - \vx } \leq
  \frac{1}{2}\norm{\vy - \vx}^2_{\mPstar^{-1}}, \quad \text{ for all }
  \vx, \vy. 
  }
  \label{eq:relative-smooth-strong-convexity}
} 
Similar definitions of smoothness and strong-convexity relative to a matrix 
are common in coordinate descent methods 
\citep[e.g.,][]{qu2016sdna,safaryan2021smoothnessmatrices},
where the matrices are assumed to be known a priori.
If we knew $\mPstar$, preconditioned gradient descent 
using $\mPstar$
would converge at the rate 
\aligns{%
  \textstyle
  f\paren{x - \mPstar \nabla f(\vx)} - f(\vx_*) 
  \leq \paren{1 - \frac{1}{\kappa_*}} (f(\vx) - f(\vx_*)),
}
where \(\vx_*\) minimizes $f$.
We do not know $\mPstar$ and will be searching for a good approximation.
For the standard backtracking line-search on $L$-smooth functions,
the goal is to find a step-size that works as well as $\nicefrac{1}{L}$
without knowledge of $L$.
To do so, we can start with a large step-size $\alpha \gg \nicefrac{1}{L}$ 
and check the \emph{Armijo condition}: 
the step-size $\alpha$ makes progress as if $f$ were $1/\alpha$-smooth, that is,
\alignn{
  \textstyle
  f(\vx - \alpha \nabla f(\vx))
  \leq 
  f(\vx) - \alpha \frac{1}{2}\norm*{\nabla f(\vx)}^2.
  \label{eq:armijo}
}
If the condition is satisfied, we take the step $\vx - \alpha \nabla f(\vx)$.
By the descent lemma, \citep[][A.24]{bertsekas1997nonlinear},
the condition is satisfied if $\alpha \leq \nicefrac{1}{L}$.
So if the condition fails, we know $\alpha$ is too large and can decrease $\alpha$.
For diagonal preconditioners, 
the Armijo condition checks whether the preconditioner 
makes sufficient progress in the norm induced by \(\mP^{-1}\),
as if \(f\) were 1-smooth in \cref{eq:relative-smooth-strong-convexity}, that is,%
\alignn{
  \textstyle
  f(\vx - \mP \nabla f(\vx)) \leq f(\vx) - \frac{1}{2}\norm*{\nabla f(\vx)}^2_{\mP}.
  \label{eq:sufficient-progress}
}%
As with a scalar step-size, sufficient progress holds for any matrix $\mP$
that satisfies $\smash{\nabla^2 f(\vx) \preceq \mP^{-1}}$.

\section{Multidimensional Backtracking}
\label{sec:convergence}

The typical presentation of the backtracking line-search 
maintains a step-size and decreases it when the Armijo condition fails
\citep[e.g.,][Alg 3.1]{nocedal1999numerical}.
We instead take the following non-standard view, 
which generalizes more naturally to high dimension;
as maintaining a set \emph{containing} the optimal step-size,
and using bisection to narrow down the size of the set.
Starting with an interval $\setS = [0, \alpha_{\max}]$ containing $\nicefrac{1}{L}$, 
we pick a candidate step-size $\alpha$ 
by ``backtracking'' by $\gamma < 1$ from the largest step-size in $\setS$, 
taking $\alpha = \gamma \alpha_{\max}$ to balance two properties;
\begingroup\setlength{\parskip}{3pt}
\begin{enumerate}[leftmargin=1.25em,itemsep=3pt,parsep=0pt,topsep=0pt,partopsep=0pt]
  \item \textbf{Large progress:}  
  If the candidate step-size satisfies the Armijo condition and the step is accepted, 
  the value of $f$ decreases proportionally to $\alpha$ as in
  \eqref{eq:armijo}.
  To maximize the progress, $\gamma$ should be large.
  \item \textbf{Volume shrinkage:} 
  If the candidate step-size fails the Armijo condition, 
  we learn that $\alpha > \nicefrac{1}{L}$
  and can cut the interval to $\setS' = [0, \gamma\alpha_{\max}]$.
  To ensure the interval shrinks fast, $\gamma$ should be small.
\end{enumerate}%
\vspace{-.25em}
\endgroup%
Taking $\gamma = \nicefrac{1}{2}$ balances both properties; 
$\alpha$ is at least $\nicefrac{1}{2}$ as large as any step-size in $\setS$, 
and we can halve the interval if the Armijo condition fails.
We do not use $\alpha_{\max}$ as a candidate since, although the largest in $\setS$, 
it would give no information to update the interval in case it failed
the Armijo condition.

For multidimensional backtracking, 
we can check whether a candidate preconditioner yields sufficient progress with \cref{eq:sufficient-progress}
instead of the Armijo condition,
and replace the intervals by sets of diagonal preconditioners.
The high-level pseudocode is given in \cref{fig:pseudocode},
where each iteration either leads to an improvement in function value 
or shrinks the sets of potential step-sizes/preconditioners.

To complete the algorithm, 
we need to define the steps marked as $(\dagger)$
to select preconditioners that lead to large progress when the step is accepted, 
while significantly reducing the search space when the preconditioner does not yield \sufficientprogressnl{}.
For computational efficiency, we want methods that take $O(d)$ time and memory like plain gradient descent.

\begin{figure}
\newcommand{\makelessvisible}[1]{{\color{gray}#1}}
\begin{minipage}{.49\textwidth}
\small
\textbf{Backtracking line-search}\\[-.5em]
\noindent\makebox[\textwidth]{\rule{\textwidth}{0.4pt}}
\begin{alglisting}
  \item {\bf Input:} 
  starting point $\vx_0$,
  backtracking coefficient $\gamma$,
  set $\setS_0 = [0, \alpha^{\max}_0]$ containing the optimal step-size $\nicefrac{1}{L}$. %
  \vspace{.5em}
  \item Iterate for $t$ in $0, 1, ..., T$
  \begin{listing}
    \item Pick step-size $\alpha_t = \gamma \alpha^{\max}_t$ %
    \item  If $(\vx_t, \alpha_t)$ satisfy the Armijo condition \eqref{eq:armijo}
    \begin{listing}
      \item Accept $\vx_{t+1} = \vx_t - \alpha_t \nabla f(\vx_t)$
      \item Keep max step-size $\alpha_{t+1}^{\max}=\alpha_t^{\max}$
    \end{listing}
    \item Otherwise, 
    \begin{listing}
      \item Don't move, $\vx_{t+1}=\vx_t$
      \item Cut max step-size $\alpha_{t+1}^{\max} = \gamma \alpha_t^{\max} $
    \end{listing}
  \end{listing}
  \item {\bf Output: } $\vx_T$
\end{alglisting}
\end{minipage}%
\hfill%
\begin{minipage}{.49\textwidth}
\small
\textbf{Multidimensional Backtracking}\\[-.5em]
\noindent\makebox[\textwidth]{\rule{\textwidth}{0.4pt}}
\begin{alglisting}
  \item {\bf Input:} 
  starting point $\vx_0$,
  backtracking coefficient $\gamma$,
  set $\setS_0$ of preconditioners containing the optimal $\mPstar$.
  \vspace{.5em}
  \item Iterate for $t$ in $0, 1, ..., T$
  \begin{listing}
    \item %
    Pick step-sizes $\mPt = \CANDIDATE(\setS_t, \gamma , \vx_t)$ \hfill $(\dagger)$
    \item  If $(\vx_t, \mPt)$ satisfy the Armijo condition \eqref{eq:sufficient-progress}
    \begin{listing}
      \item Accept $\vx_{t+1} = \vx_t - \mPt \nabla f(\vx_t)$
      \item Keep set $\setS_{t+1} = \setS_t$
    \end{listing}
    \item Otherwise, 
    \begin{listing}
      \item Don't move, $\vx_{t+1} = \vx_t$
      \item Cut set {$\setS_{t+1} = \CUT(\setS_{t}, \vx_t, \mPt)$} \hfill $(\dagger)$
    \end{listing}
  \end{listing}
  \item {\bf Output: } $\vx_T$
\end{alglisting}
\end{minipage}%

\caption{\textbf{Pseudocode for the backtracking line-search and multidimensional backtracking.}
We view backtracking as maintaining a set of step-sizes,
testing one at each iteration
that either make progress on $f$ or reduce the size of the set.
Steps marked by $(\dagger)$,
are the subject of Sections \ref{sec:convergence}--\ref{sec:cutting-plane}.
}
\vspace{-.5em}
\label{fig:pseudocode}
\end{figure}

\subsection{Guaranteed progress competitive with the optimal preconditioner}
We start by formalizing the progress guarantee. If \(\mPt\) satisfies the Armijo condition~\eqref{eq:sufficient-progress} at \(\vx_t\), the function value decreases by at
least $\norm*{\nabla f(\vx_t)}_{\mPt}^2$. If we can guarantee that
${\norm*{\nabla f(\vx_t)}_{\mPt}^2 \geq \gamma \norm*{\nabla
f(\vx_t)}_{\mPstar}^2}$ for some $\gamma > 0$, we can recover the
convergence rate of gradient descent preconditioned with \(\mPstar\) up to a
factor of $\gamma$. However, we do not know $\mPstar$, but know a
set $\setS_t$ that contains preconditioners we have not yet ruled out,
including $\mPstar$. 
To guarantee that $\mPt$ is competitive with $\mPstar$, 
we can enforce that $\mPt$ is competitive 
with \emph{all} the preconditioners in $\setS_t$, as captured by the following definition.

\begin{definition}[$\gamma$-competitive candidate preconditioners]
  \label{def:competitive}
  A matrix $\mPt \in \setS_t$ is $\gamma$-competitive in \(\setS_t\),
  for a gradient \(\nabla f(\vx_t)\), 
  if $\norm*{\nabla f(\vx_t)}_{\mPt}^2 \geq \gamma
  \norm*{\nabla f(\vx_t)}_{\mQ}^2$ for any $\mQ\in \setS_t$.
\end{definition}

If \(\mPt\) is \(\gamma\)-competitive, 
then it is competitive with $\mPstar$ as
$\smash{\max_{\mQ\in\setS_t} \norm*{\nabla f(\vx_t)}_{\mQ}^2 \geq \norm*{\nabla f(\vx_t)}_{\mPstar}^2}$.
However, this is a strong requirement. %
To illustrate what competitive ratios are attainable, 
we show in \cref{apx:competitive} that even the optimal
preconditioner $\mPstar$ might only be $\nicefrac{1}{d}$-competitive, %
as other preconditioners can lead to more \emph{local} progress depending on $\nabla f(\vx_t)$, 
whereas $\mPstar$ is a fixed \emph{global} optimal preconditioner. 
This also suggests that selecting a preconditioner that guarantees more local progress
may lead to better performance, 
which we take advantage of to ensure a \(\gamma = \nicefrac{1}{\sqrt{2d}}\) competitive ratio.

To see how to ensure a competitive ratio, 
consider the case where $\setS$ contains diagonal preconditioners 
whose diagonals come from the 
box \(\setB(\vb) \coloneqq \setst{\vp\in \smash{\Rnn^{d}}}{ \vp \leq \vb}\).
To select a candidate preconditioner that is $\gamma$-competitive in \(\setS\),
we can backtrack from the largest %
vector in $\setB(\vb)$ by some constant $\gamma < 1$, and take \(\mP =
\gamma \Diag(\vb)\). 
While a large $\gamma$ leads to more progress when the step is accepted, 
we will see that we need a small $\gamma$ 
to ensure the volume shrinks when the step is rejected.

We can obtain the convergence rate of \Cref{thm:thm1}
depending on $\gamma$ and the optimal preconditioned condition number $\kappa_*$
if we ensure \(\mPstar \in \setS_t\) and that \(\mPt\) is \(\gamma\)-competitive for all \(t\).
\begin{proposition}
\label{prop:convergence}
Let \(\mPstar, \kappa_*\) be an optimal preconditioner and condition number for \(f\) \eqref{eq:optimal-preconditioner-def}.
If the set \(\setS_t\) from the algorithm in \Cref{fig:pseudocode} contains
\(\mPstar\), and \(\mPt \in \setS_t\) is $\gamma$-competitive
(\Cref{def:competitive}), then
\aligns{%
  f(\vx_{t+1}) - f(\vx_*)
  \leq
  \paren{1 - \frac{\gamma}{\kappa_*}}
  \paren{f(\vx_t) - f(\vx_*)}
}
whenever the candidate step leads to sufficient progress
and is accepted.
\end{proposition}
\vspace{-1em}
\begin{proof}
The proof relies on three inequalities. 
(1) The iterate $\vx_{t+1}$ yields sufficient progress (Eq. \ref{eq:sufficient-progress}),
(2) any accepted preconditioner $\mPt$ is $\gamma$-competitive in $\setS_t$ 
and thus with $\mPstar$, 
and (3) $f$ is $\nicefrac{1}{\kappa_*}$-strongly convex in $\norm{\cdot}{}_{\mPstar^{-1}}$, 
which implies $\kappa_* \smash{\frac{1}{2}} \norm*{\nabla f(\vx_t)}^2_{\mPstar} \geq f(\vx_t) - f(\vx_*)$.
Combining those yields
\begingroup
\setlength{\abovedisplayskip}{2pt}
\setlength{\belowdisplayskip}{2pt}
\aligns{
  f(\vx_{t+1}) 
  \stackrel{\scriptscriptstyle(1)}{\leq} f(\vx_t) - \frac{1}{2}\norm{\nabla f(\vx_t)}_{\mPt}^2
  \stackrel{\scriptscriptstyle(2)}{\leq} f(\vx_t) - \gamma \frac{1}{2}\norm{\nabla f(\vx_t)}_{\mPstar}^2
  \stackrel{\scriptscriptstyle(3)}{\leq} f(\vx_t) - \frac{\gamma}{\kappa_*} \paren{f(\vx_t) - f(\vx_*)}.
  \vspace{-.2em}
}
\endgroup
Subtracting $f(\vx_*)$ on both sides yields the contraction guarantee.
\end{proof}

\section{Separating hyperplanes in higher dimensions}
\label{sec:direction}

In one dimension, if the step-size $\alpha$ does not satisfy the
sufficient progress condition~\eqref{eq:armijo}, we know $\alpha >
\nicefrac{1}{L}$ and can rule out any $\alpha' \geq \alpha$. 
We are looking for a generalization to higher dimensions: if the queried
preconditioner fails the \sufficientprogressnl{} condition, we should be
able to discard all larger preconditioners. The notion of \emph{valid}
preconditioners formalizes this idea.
\begin{definition}[Valid preconditioner]
\label{def:invalid}
A preconditioner $\mP$ is \emph{valid} 
if $\mP^{\nicefrac{1}{2}}\nabla^2 f(\vx)\mP^{\nicefrac{1}{2}} \preceq \mI$ for all $\vx$,
which guarantees that $\mP$ satisfies the \sufficientprogress{} condition, 
and \emph{invalid} otherwise. %
\end{definition}

\begin{figure}[t]
  \begin{subfigure}[t]{0.31\textwidth}
    \centering
    \includegraphics[width=.55\textwidth]{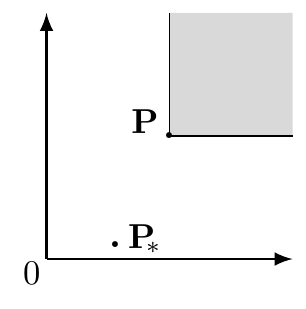}%
    \caption{
      Failing the Armijo condition cuts the interval in half
      in one dimension, 
      but only removes $\nicefrac{1}{2^d}$ of the volume in $d$ dimensions.
    }
    \label{fig:example-lack-info}
  \end{subfigure}\hfill\begin{subfigure}[t]{0.31\textwidth}
    \centering
    \includegraphics[width=.55\textwidth]{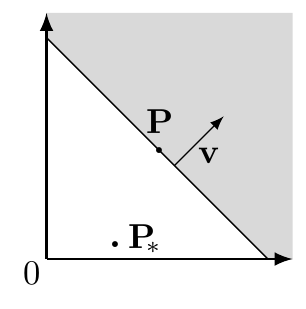}%
    \caption{
      Half-space $\setH_>(\vu)$
      obtained by using the hypergradient
      when failing the Armijo condition at $\mP$
      in \Cref{lem:starting-hyperplane-main}.
    }
    \label{fig:hp}
  \end{subfigure}\hfill\begin{subfigure}[t]{0.31\textwidth}
    \centering
    \includegraphics[width=.55\textwidth]{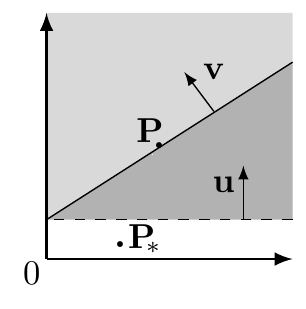}%
    \caption{
      Stronger half-space $\setH_>(\vv)$
      described by \cref{lem:stronger-hp},
      removing $\mP'\,{\succ}\,\mP$
      for any $\mP$ ruled out by $\setH_>(\vu)$
      in \Cref{lem:starting-hyperplane-main}.
    }
    \label{fig:stronger-hp}
  \end{subfigure}
\caption{
{\bf Lack of information from the ordering
and separating hyperplanes.}
}
\end{figure}

Validity is a global property: a preconditioner $\mP$ might lead to
sufficient progress locally but still be invalid. 
Using the partial order, 
if $\mP$ is invalid then any preconditioner $\mP' \succeq \mP$ is
also invalid. 
However, this property alone only discards an exceedingly small portion
of the feasible region in high dimensions. Consider the example
illustrated in \Cref{fig:example-lack-info}: if the diagonals are in a
box \(\setB(\vb)\), the fraction of volume discarded in this way if
\((\nicefrac{1}{2})\Diag(\vb)\)
is invalid is only
\(\smash{\nicefrac{1}{2^d}}\).

To efficiently search for valid preconditioners, we show that if \(f\)
is convex, then the gradient of the \sufficientprogressnl{} condition
gives a \emph{separating hyperplane} for valid preconditioners. That is,
it gives a vector $\smash{ \vu \in \R^d }$ such that if \(\vp \in
\Rnn^d\) satisfies \(\lin{\vu, \vp} > 1\), then \(\Diag(\vp)\) is
invalid, as illustrated in \Cref{fig:hp}. We use the following notation
to denote normalized half-spaces:
\aligns{
    \setH_{>}(\vu) \coloneqq \setst{\vp \in \Rnn^d}{ \lin{\vu,\vp} > 1}
  && 
  \text{and}
  && 
    \setH_{\leq}(\vu) \coloneqq \{\vp \in \Rnn^d : \lin{\vu,\vp} \leq 1\}.
}
\begin{proposition}[Separating hyperplane in preconditioner space]
\label{lem:starting-hyperplane-main}
  Suppose $\mQ=\Diag(\vq) \succ 0$
  does not lead to \sufficientprogress{} at \(\vx\), 
  and let $h(\vq)$ be the gap in the sufficient progress condition,
  \aligns{
    h(\vq)
    \coloneqq f(\vx - \mQ \nabla f(\vx)) - f(\vx) + \tfrac{1}{2}\norm{\nabla f(\vx)}_{\mQ}^2 > 0.
  } Then $\Diag(\vp)$ for any \(\vp\) in the following half-space
  satisfies $h(\vp) > 0$ and is also \invalid{}, 
  \alignn{
    \setst{\vp \in \R^d}{ \lin{\nabla h(\vq), \vp} > \lin{\nabla h(\vq), \vq} - h(\vq)},
    \label{eq:part-1}
  } 
  This half-space is equal to \(\setH_{>}(\vv)\) with $\vv$ given by 
  $\vv = \nicefrac{\nabla h(\vq)}{\paren{\lin{\nabla h(\vq), \vq} - h(\vq)}}$, or 
  \alignn{%
    \vv \coloneqq \frac{
    (\frac{1}{2}\vg - \vg^+) \odot \vg
  }{
    f(\vx) 
    - \lin{\vg^+, \mQ \vg} 
    - f(\vx^+) 
  },
  &&\text{with}&&
  \left\{\begin{aligned}
  \vx^+ &\coloneqq \vx - \mQ \nabla f(\vx),\\
  (\vg, \vg^+) &\coloneqq (\nabla f(\vx), \nabla f(\vx^+)).
  \end{aligned}\right.
  \label{eq:strong-sep-hyp}
 }
\end{proposition}
\vspace{-1.0em}
\begin{proof}[Proof idea]
If $f$ is convex, then $h$ also is. Convexity guarantees that  $h(\vp)
\geq h(\vq) + \lin{\nabla h(\vq), \vp - \vq}$ for any \(\vp\). A
sufficient condition for $h(\vp) > 0$, which means $\vp$ is invalid, is
whether $h(\vq) + \lin{\nabla h(\vq), \vp - \vq} > 0$ holds.
Reorganizing yields \cref{eq:part-1}, and \cref{eq:strong-sep-hyp}
expresses the half-space in normalized form, $\setH_>(\vv)$, expanding
$h$ in terms of $f$, its gradients, and $\mQ$.
\end{proof}
\vspace{-0.5em}

The half-space in \cref{lem:starting-hyperplane-main}
is however insufficient to find good enough cutting-planes, 
as it uses convexity to invalidate preconditioners
but ignores the ordering that if $\mP$ is \invalid{}, any $\mP' \succeq \mP$ is also \invalid{}. 
If such preconditioners are not already ruled out by convexity, 
we can find a stronger half-space by removing them, 
as illustrated in \Cref{fig:stronger-hp}. 
We defer proofs to \cref{apx:hyperplanes}.

\begin{proposition}[Stronger hyperplanes]
\label{lem:stronger-hp}
If $\setH_>(\vv)$ is a half-space given by
\cref{lem:starting-hyperplane-main}, then $\setH_>(\vu)$ where $\vu
\coloneqq \max\{\vv, 0\}$ element-wise is a stronger half-space in the
sense that $\setH_>(\vv) \subseteq \setH_>(\vu)$, and $\setH_>(\vu)$
contains only invalid preconditioners.
\end{proposition}

\section{Cutting-plane methods}
\label{sec:cutting-plane}

The multidimensional backtracking method is in fact
a cutting-plane method that uses separating hyperplanes (from  \Cref{lem:stronger-hp}) to search for valid preconditioners. 
The canonical example is the ellipsoid method \citep{YudinN76a, Shor77a}, but its computational cost is $\Omega(d^2)$ in $\mathbb{R}^d$.
We now describe cutting-plane methods with three desirable properties: the preconditioners have good competitive ratios,
the feasible set shrinks significantly when backtracking,
and the computational cost is $O(d)$.
There are many details, but 
the overall idea is similar to the ellipsoid method.

{\bf A simple warm-up: boxes.} Consider the case when \(\setS_0\) consists
of diagonal matrices  
with diagonals in the box \(\setB(\vb_0) = \setst{\vp \in \smash{\Rnn^d}}{\vp
\leq \vb_0}\). We pick a candidate preconditioner by backtracking from
the largest point in $\setB(\vb_0)$ by some constant $\gamma < 1$,
taking \(\mP \coloneqq \gamma \Diag(\vb_0)\). If \(\mP\) satisfies the
Armijo condition~\eqref{eq:sufficient-progress}, we take a gradient
step. If it does not, we compute the vector $\vu_0$ as
in~\cref{lem:stronger-hp}, and obtain a half-space $\setH_{>}(\vu_0)$
that contains only invalid preconditioners. We then know we only need to
search inside \(\setS_0 \cap \setH_{\leq}(\vu_0)\). However, maintaining
the set $\setS_0 \cap \setH_{\leq}(\vu_0) \cap \cdots \cap
\setH_{\leq}(\vu_t)$ would be too complex to fit in $O(d)$ time or
memory. To reduce complexity, we define \(\setS_{t+1}\) as the box
$\setB(\vb_{t+1})$ of minimum volume containing $\setB(\vb_t) \cap
\setH_{\leq}(\vu_t)$, as illustrated in \cref{fig:box-good}. Due to this
restriction, we might not be able to find a smaller set; the original
box $\setB(\vb_{t})$ may already be the minimum volume box containing
$\setB(\vb_t) \cap \setH_{\leq}(\vu_t)$ if $\vu_t$ does not cut deep
enough, as illustrated in \cref{fig:box-bad}. However, with enough
backtracking ($\gamma < \nicefrac{1}{d}$), we can show that the new box
is smaller. This yields the following subroutines to fill in the gaps of
\cref{fig:pseudocode} (detailed in \cref{apx:set-box}) \alignn{
  \textstyle
  \CANDIDATE(\setS_t, \gamma, \vx_t) \coloneqq \gamma \Diag(\vb_t),
  &&
  \CUT(\setS_t, \mPt) \coloneqq \setst{\Diag(\vp)}{\vp_t \in \setB(\vb_{t+1})},
  \label{eq:wtv}
}
where $\setS_t = \setB(\vb_t)$ and
\(\vb_{t+1} \coloneqq \min\braces{\vb_t, \nicefrac{1}{\vu_t}}\)
element-wise,
which give the following guarantees.
\begin{theorem}
\label{thm:box}
Consider the multidimensional backtracking from \Cref{fig:pseudocode} 
initialized with a set \(\setS_0 = \setst{\Diag(\vp)}{\vp \in \setB(\vb_0)}\) containing $\mPstar$,
with the subroutines in~\cref{eq:wtv} with $\gamma = \nicefrac{1}{2d}$.
Then: (a) $\mPstar\in\setS_t$,
(b) the candidate preconditioner 
\(\mPt\) is \(\nicefrac{1}{2d}\)-competitive in $\setS_t$ for any $t$, 
and%
\aligns{
  \textstyle\text{(c)}\quad
  \Vol(\setB(\vb_{t+1})) \leq \frac{1}{d+1} \Vol(\setB(\vb_t))
  \quad\text{when $\mPt$ fails \cref{eq:sufficient-progress}.}
}
In particular, \(\CUT\) is not called more than \(d \log_{d+1}(L \norm{\vb_0}_{\infty})\) times.
\end{theorem}
\vspace{-0.75em}
\begin{proof}[Proof idea]
To guarantee that the box shrinks,
we have to guarantee that the half-space $\setH_{\leq}(\vu_t)$ cuts deep enough.
We know that the half-space has to exclude the query point $\mPt$, i.e. $\lin{\vp_t, \vu_t} \geq 1$, 
by \cref{lem:starting-hyperplane-main}
and that $\vu_t \geq 0$ by \cref{lem:stronger-hp}. 
Querying \(\mPt\) sufficiently close to the origin,
by taking $\gamma = \nicefrac{1}{2d}$, 
is then enough to guarantee the decrease.
To bound the total number of cuts, 
we note that the sets $\setB(\vb_t)$ have a minimum volume $\Vol_{\min}$, 
as they have to contain the valid preconditioners.
The number of cuts is at most 
\(\log_c\paren{\Vol(\setB(b_0))/\Vol_{\min}}\) 
for \(c = d+1\).
We then bound $\Vol(\setB(b_0)) \leq \norm*{\vb_0}_{\infty}^d$
and $\Vol_{\min} \geq \nicefrac{1}{L^d}$ 
as $(\nicefrac{1}{L})\mI$ is a valid preconditioner.
\end{proof}
\vspace{-.5em}

\begin{figure}
\centering
  \begin{subfigure}[b]{0.58\textwidth}
    \centering
    \includegraphics[width=.15\paperwidth]{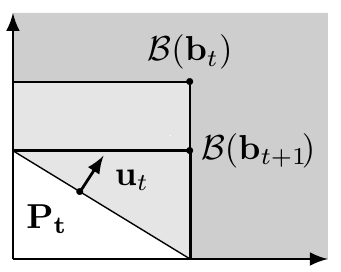}%
    \vspace{-.3em}
    \caption{
      {\bf Minimum-volume box containing the intersection.}
      We maintain sets of low-complexity 
      by computing the minimum-volume box $\setB(\vb_{t+1})$ 
      containing the intersection of the initial box $\setB(\vb_t)$
      and the half-space $\setH_{\leq}(\vu_t)$ 
      obtained from \Cref{lem:stronger-hp}
      when the preconditioner $\mPt$ 
      fails to yield sufficient decrease.
    }
    \label{fig:box-good}
  \end{subfigure}\hfill\begin{subfigure}[b]{0.39\textwidth}
    \centering
    \includegraphics[width=.15\paperwidth]{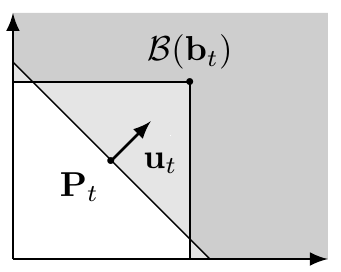}%
    \vspace{-.3em}
    \caption{
      {\bf Need sufficient backtracking.}
      If the candidate preconditioner $\mPt$ 
      selected inside the initial box $\setB(\vb_t)$
      is not close enough to 0,
      there might not be a box smaller than $\setB(\vb_t)$
      that contains the intersection.
    }
    \label{fig:box-bad}
  \end{subfigure}
\caption{\textbf{Minimum-volume enclosing boxes}}
\end{figure}

\subsection{Multidimensional Backtracking with Centered Axis-aligned Ellipsoids}

We now improve the competitive ratio from \(O(\nicefrac{1}{d})\) to $O(\nicefrac{1}{\sqrt{d}})$ by switching from boxes to ellipsoids.
Whereas general ellipsoids 
would require $\Omega(d{}^2)$ complexity (as they involve a $d {\times} d$ matrix),
we consider \emph{centered, axis-aligned ellipsoids},
defined by a diagonal matrix $\mA = \Diag(\va)$,
of the form
$\smash{\setE(\va) \coloneqq \{\vp \in \Rnn^d : \norm*{\vp}_{\mA} \leq 1\}}$,
where $\smash{\norm*{\vp}_{\mA}^2 = \lin{\vp, \mA \vp}}$.
As preconditioners are non-negative, 
we consider only the positive orthant of the ellipsoid.
For simplicity, we refer to those sets as \emph{ellipsoids}.

{\bf Candidate preconditioner.}
In the box example, 
we selected the candidate preconditioner
by backtracking from the largest preconditioner in the box.
With an ellipsoid, 
there is no \emph{largest} preconditioner. 
We need to choose where to backtrack from.
To ensure the candidate preconditioner $\mP$
is competitive (\Cref{def:competitive}), 
we backtrack from the preconditioner 
that maximizes the progress $\norm*{\nabla f(\vx)}_\mP^2$,
\alignn{
  \label{eq:max_over_ellipsoid}
  \argmax_{\vp \in \setE(\va)} 
  \norm{\nabla f(\vx)}_{\mP}^2 
  = 
  \frac{\mA^{-1} \nabla f(\vx)^2}{\norm{\nabla f(\vx)^2}_{\mA^{-1}}},
  && 
  \textstyle\paren{\text{where } \nabla f(\vx)^2 \coloneqq \nabla f(\vx) \odot \nabla f(\vx)}.
}
This lets us pick the preconditioner that makes the most progress
\emph{for the current gradient}, 
and will let us improve the competitive ratio 
by allowing a backtracking coefficient of $\nicefrac{1}{\sqrt{d}}$
instead of $\nicefrac{1}{d}$.

{\bf Cutting.} 
To complete the algorithm,
we need to find a new set $\setE(\vb_{t+1})$ with smaller volume
which contains the intersection of the previous set $\setE(\vb_t)$
and the half-space $\setH_{\leq}(\vu_t)$.
Unlike the box approach, 
the minimum volume ellipsoid has no closed form solution.
However, if we backtrack sufficiently, by a factor of $\gamma < \nicefrac{1}{\sqrt{d}}$,
we can find an ellipsoid guaranteed to decrease the volume.
\begin{lemma}
  \label{lem:lemlem}
  Consider the ellipsoid $\setE(\va)$ 
  defined by $\mA = \Diag(\va)$ for $\va \in \Rplus^d$.
  Let $\vp \in \setE(\va)$ 
  be a point sufficiently deep inside the ellipsoid, 
  such that $\norm{\vp}_{\mA} \leq \nicefrac{1}{\sqrt{2d}}$,
  and $\setH_{>}(\vu)$ be a half-space 
  obtained from \cref{lem:stronger-hp} at $\vp$.
  The intersection $\setE(\va) \cap \setH(\vu)_{\leq}$
  is contained in the new ellipsoid 
  \alignn{
    \setE(\va^+(\va, \vu)),
    &&\text{ where }&&
    \va^+(\va, \vu) = \lambda \va + (1 - \lambda) \vu^2,
    &&
    \textstyle
    \lambda = \frac{\ell}{d} \frac{d-1}{\ell-1},
    &&
    \ell = \norm{\vu }_{\mA^{-1}}^2,
   \label{eq:new-ellipsoid}
  }
  which has a smaller volume, 
  \(\Vol(\setE(\va^+(\va, \vu)) \leq c \Vol(\setE(\va))\),
  where $c = \nicefrac{\sqrt[4]{e}}{\sqrt{2}} \approx 0.91$.
\end{lemma}
\vspace{-.5em}
\begin{proof}[Proof idea.]
The new ellipsoid in~\eqref{eq:new-ellipsoid}
is a convex combination between $\setE(\va)$
and the minimum volume ellipsoid containing the set 
\(\smash{\setst{\vp \in \R^d}{
\lin{\vu, |\vp|}\leq 1}}\) where $|\vp|$ is the element-wise absolute value of $\vp$. 
The choice of $\lambda$ in~\eqref{eq:new-ellipsoid}
is not optimal, but suffices to guarantee progress
as long as $\norm{\vp}_{\mA}$ is small.
A similar approach was used by \citet{GoemansHIM09a}
to approximate submodular functions, although they consider
the polar problem of finding a maximum-volume enclosed ellipsoid. 
The full proof and discussion on the connections to the polar problem 
are deferred to \Cref{apx:set-ellipsoid}.
\end{proof}
\vspace{-1.0em}
To improve the cuts, 
we can refine the estimate of $\lambda$ in \cref{lem:lemlem}
by minimizing the volume numerically. 
We include this modification, 
detailed in \cref{apx:refinement},
in our experiments in \cref{sec:experiments}.

{\bf Overall guarantees.}
We can now define the two subroutines for the ellipsoid method, 
and obtain the main result that we stated informally in \Cref{thm:thm1}, 
by combining the guarantees of the ellipsoid approach 
with the convergence result of \cref{prop:convergence}.
\begin{theorem}
  \label{thm:main_thm_ellipsoid}
  Consider the multidimensional backtracking from \Cref{fig:pseudocode}
  initialized with the set \(\setS_0 = \setst{\Diag(\vp)}{\vp \in
  \setE(\va_0)}\) containing $\mPstar$,
  given by some scaling $\alpha_0 > 0$ of the uniform vector, $\va_0 = \alpha_0 \vone$.
  For $\setS_t$, let $\mA_t = \Diag(\va_t)$.
  Define the subroutines 
  \aligns{
    \mPt = \CANDIDATE(\setS_t, \gamma, \vx_t) 
    \coloneqq
    \gamma \tfrac{\mA_t^{-1} \nabla f(\vx_t)^2}{\norm{\nabla f(\vx_t)^2}_{\mA_t^{-1}}},
    \,\,
    \CUT(\setS_t, \mPt) 
    \coloneqq 
    \setst{\Diag(\vp)}{\vp \in \setE(\va^+(\va_t, \vu_t))},
    \label{eq:ellipsoid_subroutines}
  }
  where \(\vu_t\) is the vector 
  given by \Cref{lem:stronger-hp}
  when
  $\mPt$ fails the Armijo condition at \(\vx_t\),
  and $\va^+$ is computed as in~\eqref{eq:new-ellipsoid}.
  If \(\gamma = \nicefrac{1}{\sqrt{2d}}\), then:
  (a) $\mPstar\in\setS_t$ for all $t$, 
  (b) the candidate preconditioners
  \(\mPt\) are \(\nicefrac{1}{\sqrt{2d}}\)-competitive in $\setS_t$,
  and (c) \(\CUT\) is called no more than \(12 d \log(L/\alpha_0)\) times.
\end{theorem}

\section{Experiments}
\label{sec:experiments}

\begin{figure}
  \centering
  \includegraphics{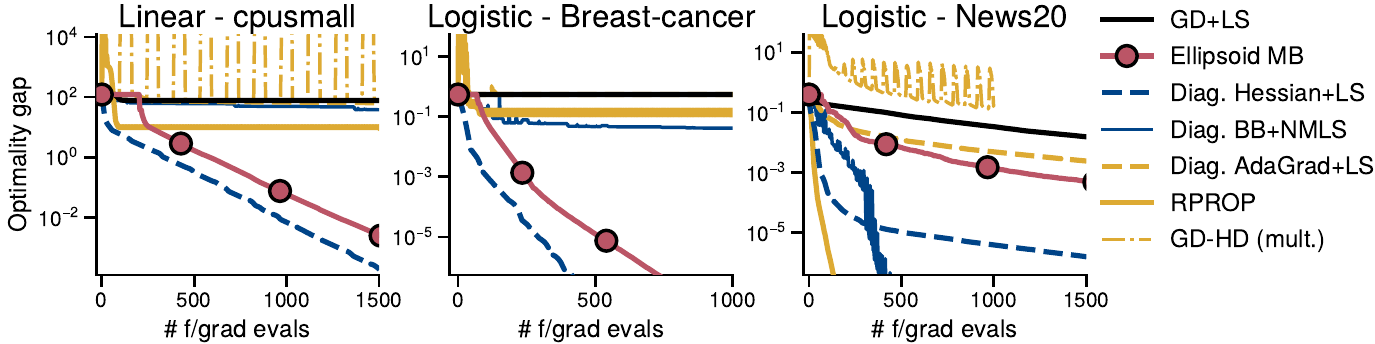}
  \caption{ 
  {\bf Multidimensional backtracking finds a good preconditioner, when there is one.
  } 
  Experiments on regularized linear and logistic regression on  
  small but ill-conditioned datasets, 
  cpusmall and  breast-cancer (left, middle),
  and the large dataset News20 (right, \(d \approx 10^6\)).
  Methods used: Gradient Descent (GD)
  Multidimensional Backtracking (MB) with ellipsoids, 
  diagonal Hessian, 
  diagonal BB, 
  and diagonal AdaGrad---%
  all of which use a line-search (+LS)---
  RPROP, 
  and GD with hypergradient-tuned step-size (GD-HD) using the multiplicative update.
  Details in \cref{apx:experimental-details}.
}
\label{fig:exps}
\end{figure}

To illustrate that  multidimensional backtracking
finds good preconditioners 
and improves over gradient descent on ill-conditioned problems
even when accounting for the cost of backtracking,
we run experiments on small but very ill-conditioned 
and large ($d\approx 10^6$) problems.

As examples of adaptive gain and hypergradient methods,
we include RPROP \citep{riedmiller1993rprop} 
and GD with a hypergradient-tuned step-size (GD-HD, \citealt{baydin2018online} with the multiplicative update).
As examples of approximate second-order methods,
we take diagonal BB %
\citep{park2020variable} 
and preconditioned GD using the diagonal of the Hessian. 
We use default parameters, 
except for the hypergradient method GD-HD,
where we use $10^{-10}$ as the initial step-size 
instead of $10^{-3}$ to avoid immediate divergence. 
We include AdaGrad (diagonal), but augment it with a line-search 
as suggested by \citet{vaswani2020adaptive},
to make it competitive in the deterministic setting.

{\bf Line-searches and forward steps.}
For all methods that use a line-search, we include a \emph{forward} step, 
a common heuristic in line-search procedures
to allow for larger step-sizes when possible, although it can increase the number of backtracking steps.
When a step-size or preconditioner is accepted, 
we increase the size of the set,
allowing for larger (scalar or per-coordinate) step-sizes by a factor of $1.1$.
We measure performance per function and gradient evaluations to capture the cost of backtracking.

{\bf On a small but extremely ill-conditioned problems},
our method is the only one that gets remotely close to being competitive with 
preconditioning with the diagonal Hessian---while only using first-order information.
The diagonal Hessian is very close to the optimal preconditioner for those problems.
On the cpusmall dataset, it reduces the condition number 
from $\kappa \approx 5\cdot 10^{13}$ to $\approx 300$,
while $\kappa_* \approx 150$.
All other methods struggle to make progress
and stall before a reasonable solution is achieved,
indicating they are not competitive with the optimal preconditioner.

{\bf On large regularized logistic regression on News20 ($d \approx 10^6$)},
gradient descent performs relatively better, 
suggesting the problem is less ill-conditioned to begin with 
(the regularized data matrix has condition number $\kappa \approx 10^4$).
Despite the bound of $O(d)$ backtracking steps, 
our methods finds a reasonable preconditioner within 100 gradient evaluations.
Despite the high dimensionality, it improves over gradient descent 
when measured in number of oracle calls.

Using plain gradient updates on the hyperparameters in GD-HD 
leads to unstable behavior, but diagonal BB and even RPROP,
perform remarkably well on some problems%
---even outperforming preconditioning with the diagonal Hessian, which uses second-order information. 
However, they fail on other ill-conditioned problems, 
even when a good diagonal preconditioner exists.
This pattern holds across other problems, as shown in \cref{apx:exp-additional}.
Multidimensional backtracking demonstrates robust performance across problems, 
a clear advantage of having worst-case guarantees.

\section{Conclusion}
We designed \emph{multidimensional backtracking}, 
an efficient algorithm to automatically find diagonal preconditioners 
that are competitive with the optimal diagonal preconditioner. 
Our work provides a definition of adaptive step-sizes
that is complementary to the online learning definition.
While online learning focuses on the adversarial or highly stochastic setting, 
we define and show how to %
find optimal per-coordinate step-sizes
in the deterministic smooth convex setting. 
We show it is possible to build provably robust methods 
to tune a preconditioner using hypergradients.
While our specific implementation uses cutting-planes,
the general approach may lead to alternative algorithms, 
that possibly tune other hyperparameters,
with similar guarantees.

The main limitation of our approach 
is its reliance on the convex deterministic setting.
The results might transfer to the stochastic overparametrized regime 
using the approach of \citet{vaswani2019painless},
but the non-convex case seems challenging.
It is not clear how to get reliable information 
from a cutting-plane perspective using hypergradients without convexity.
As the first method 
to provably find competitive preconditioners,
there are likely modifications that lead to practical improvements
while preserving the theoretical guarantees.
Possible ideas to improve practical performances include
better ways to perform forward steps, 
using hypergradient information from accepted steps (which are currently ignored),
or considering alternative structures to diagonal preconditioners.

\clearpage
\begin{ack}
We thank Aaron Mishkin for helpful discussions in the early stages of this work,
and Curtis Fox and Si Yi (Cathy) Meng for providing comments on an early version of the manuscript.
This research was partially supported by the Canada CIFAR AI Chair Program, 
the Natural Sciences and Engineering Research Council of Canada (NSERC) Discovery Grants RGPIN-2022-03669, 
\end{ack}

\section*{References}
\AtNextBibliography{\normalsize}
\printbibliography[heading=none]

\clearpage
\appendix

\vbox{
  {\hrule height 4pt \vskip 0.25in \vskip -\parskip}
  \centering
  {\LARGE\bf Supplementary Material\par}
  {\vskip 0.29in \vskip -\parskip \hrule height 1pt \vskip 0.09in}
}

\newcommand\invisiblepart[1]{%
  \refstepcounter{part}%
  \addcontentsline{toc}{part}{\protect\numberline{\thepart}#1}%
}

\invisiblepart{Appendix}
\setcounter{tocdepth}{2}
\localtableofcontents

\vspace{2em}

Code available at 
\hfill
\href{https://github.com/fKunstner/multidimensional-backtracking}{\tt https://github.com/fKunstner/multidimensional-backtracking}

\vfill

\section{Full pseudocode of the algorithms}
\label{apx:pseudocode}

We first give with a generic version using the subroutines 
$\INITSET$, $\CANDIDATE$ and $\CUT$, 
to be specialized for the 
backtracking line-search (\cref{alg:subroutines-linesearch}), 
multidimensional backtracking using boxes (\cref{alg:subroutines-box}),
and ellipsoids (\cref{alg:subroutines-ellipsoid}).
The generic pseudocode is written in terms of preconditioners, 
but also applies to the step-size version, 
which we can consider as looking for a preconditioner 
constrained to isotropic diagonal preconditioners, that is, preconditioners in the set $\{\alpha \mI : \alpha \in \Rnn\}$.

Although we write the pseudocode maintaining at each iteration an
abstract set of preconditioners $\setS$, the only information the
algorithm needs to maintain on each iteration for the implementation in
the different cases is 
\begin{itemize}[leftmargin=1em]
  \item {\bf For the line-search:}\\
  the current maximum step-size $\alpha_{\max}$ 
  defining the interval of valid step-sizes,
  $[0,\alpha_{\max}]$ 
  such that the set of preconditioners 
  is $\setS = \{\alpha \mI : \alpha \in [0,\alpha_{\max}]\}$;
  \item {\bf For multidimensional backtracking with boxes:}\\
  the vector $\vb$ defining the maximum corner of the box 
  $\setB(\vb) = \{\vp \in \Rnn^d : \vp \leq \vb\}$
  used to define the candidate diagonals preconditioners 
  in the set $\setS = \{\Diag(\vp) : \vp \in \setB(\vb)\}$;
  \item {\bf For multidimensional backtracking with ellipsoids:}\\
  the vector $\va$ defining the axis-aligned ellipsoid
  $\setE(\va) = \{\vp \in \Rnn^d : \lin{\vp, \Diag(\va) \vp}\leq 1\}$
  used to define the candidate diagonal preconditioners 
  in the set $\setS = \{\Diag(\vp) : \vp \in \setE(\va)\}$. 
\end{itemize}

The pseudocode in \cref{alg:full-pseudocode}
updates $(\vx_t, \setS_t)$ to $(\vx_{t+1}, \setS_{t+1})$ at each iteration,
and ensures that 
either the function value decreases, $f(\vx_{t+1}) < f(\vx_t)$,
or the volume decreases, $\Vol(\setS_{t+1}) < \Vol(\setS_t)$.

We give an alternative pseudocode %
in \cref{alg:full-pseudocode-while-loop}, 
which defines iterations as updates to the iterates $\vx_t$ that decrease the function value,
and uses a \texttt{while}-loop to backtrack. 
Since it more closely resemble standard ways backtracking line-search is described, 
some reader may find it easier to understand. 
We stress, however, that this is still the same algorithm as \cref{alg:full-pseudocode} but written differently.

The pseudocode
in Figures \ref{alg:full-pseudocode}--\ref{alg:subroutines-ellipsoid},
are expressed in a modular form 
to highlight how the algorithm works and its similarity to a line-search.
In \cref{apx:implementable}, 
we give a more directly implementable pseudocode 
of multidimensional backtracking in both box and ellipsoid variants solely relying on vector notation.

\clearpage

~

\vfill

\begin{figure}[h]
  \caption{
    {\bf Generic pseudocode for the line-search or multidimensional backtracking.}
    uses the subroutines $\INITSET$, $\CANDIDATE$, $\CUT$ defined in the later sections.
  }
  \label{alg:full-pseudocode}
  \fbox{\begin{minipage}{\textwidth}
  \begin{alglisting}
    \item \textbf{Backtracking Preconditioner Search with Sets}
    \item \hrule
    \item {\bf Input:} 
    \begin{apxlisting}
      \item A starting point \(\vx_0 \in \R^d\);
      \item A backtracking coefficient \(\gamma \in [0,1]\);
      \item A scalar $c_0 > 0$ larger than the optimal preconditioner, i.e., such that $\mPstar \preceq c_0 \mI$.
    \end{apxlisting}
    \item
    \item $\setS_0 = \INITSET(c_0)$
    \item
    \item Iterate for $t$ in $0, 1, ..., T-1$
    \begin{apxlisting}
      \item $\mPt = \CANDIDATE(\setS_t, \gamma, \nabla f(\vx_t))$
      \item If 
      $f(\vx_t - \mPt \nabla f(\vx_t)) \leq f(\vx_t) - \frac{1}{2}\norm*{\nabla f(\vx_t)}^2_{\mPt}$
      \codecomment{Armijo condition \cref{eq:sufficient-progress}}
      \begin{apxlisting}
        \item $(\vx_{t+1}, \setS_{t+1}) = (\vx_t - \mPt \nabla f(\vx_t), \setS_t)$
      \end{apxlisting}
      \item Otherwise, 
      \begin{apxlisting}
        \item $(\vx_{t+1}, \setS_{t+1}) = (\vx_t, \CUT(\setS_t, \vx_t, \mPt))$
      \end{apxlisting}
    \end{apxlisting}
    \item
    \item {\bf Output: } $\vx_T$
  \end{alglisting}
  \end{minipage}}
\end{figure}

\vfill

\begin{figure}[h]
  \caption{
    {\bf Alternative pseudocode for the line-search or multidimensional backtracking.}
    Uses a while-loop for backtracking and only updates the iterates $\vx_t$ 
    when they lead to progress.
  }
  \label{alg:full-pseudocode-while-loop}
  \fbox{\begin{minipage}{\textwidth}
  \begin{alglisting}
    \item \textbf{Backtracking Preconditioner Search with Sets -- \texttt{while}-loop variant}
    \item \hrule
    \item {\bf Input:} 
    \begin{apxlisting}
      \item A starting point \(\vx_0 \in \R^d\);
      \item A backtracking coefficient \(\gamma \in [0,1]\);
      \item A scalar $c_0 > 0$ larger than the best preconditioner, that is,
      $\mPstar \preceq c_0 \mI$.
    \end{apxlisting}
    \item
    \item Initialize the set $\setS = \INITSET(c_0)$
    \item Iterate for $t$ in $0, 1, ..., T-1$
    \begin{apxlisting}
      \item $\mPt \leftarrow \CANDIDATE(\setS_b, \gamma, \nabla f(\vx_t))$
      \item While 
      $f(\vx_t - \mPt \nabla f(\vx_t)) \leq f(\vx_t) - \frac{1}{2}\norm*{\nabla f(\vx_t)}^2_{\mPt}$
      \codecomment{Armijo condition \cref{eq:sufficient-progress}}
      \begin{apxlisting}
        \item $\setS \leftarrow \CUT(\setS, \vx_t, \mPt)$
        \item $\mPt \leftarrow \CANDIDATE(\setS, \gamma, \nabla f(\vx_t))$
      \end{apxlisting}
      \item $\vx_{t+1} = \vx_t - \mPt \nabla f(\vx_t)$
    \end{apxlisting}
    \item
    \item {\bf Output: } $\vx_T$
  \end{alglisting}
  \end{minipage}}
\end{figure}  

\vfill

\clearpage
\subsection{Subroutines for standard backtracking line-search}

Implementation of the subroutines for the standard backtracking line-search.
Although written in terms of sets, 
the algorithm only needs to maintain the maximum step-size 
in the interval $[0, \alpha_{\max}]$ at each iteration.
The corresponding preconditioners 
are the matrices 
$\setS = \{\alpha \mI : \alpha \in [0,\alpha_{\max}]\}$.

\vfill

\begin{figure}[h]
  \caption{{\bf Specialization of the subroutines for the backtracking line-search}}
  \label{alg:subroutines-linesearch}
  \fbox{\begin{minipage}{\textwidth}
    \begin{alglisting}
      \item \(\INITSET(c_0)\)
      \item \hrule
      \item {\bf Input:}
      \begin{apxlisting}
        \item A scaling $c_0$ larger than the optimum step-size, that is, such that $\nicefrac{1}{L} \leq c_0$.
      \end{apxlisting}
      \item
      \item {\bf Output:} Set of preconditioners $\setS = \{\alpha \mI : \alpha \in [0, c_0]\}$.
    \end{alglisting}
    \end{minipage}
  }
  ~\\~\\
  \fbox{\begin{minipage}{\textwidth}
  \begin{alglisting}
    \item \(\CANDIDATE(\setS, \gamma, \vx)\)
    \item \hrule
    \item {\bf Input:}
    \begin{apxlisting}
      \item Set of scalar preconditioners $\setS = \{\alpha \mI : \alpha \in [0, \alpha_{\max}]\}$;
      \item Backtracking coefficient \(\gamma \in [0,1]\);
      \item Current iterate \(\vx \in \R^d\). \codecomment{Not used for the step-size version.}
    \end{apxlisting}
    \item
    \item {\bf Output: } Preconditioner $\gamma \alpha_{\max} \mI$
    \end{alglisting}
   \end{minipage}
 }
 ~\\~\\
 \fbox{\begin{minipage}{\textwidth}
  \begin{alglisting}
    \item \(\CUT(\setS, \vx, \mP)\)
    \item \hrule
    \item {\bf Input:}
    \begin{apxlisting}
      \item Set of scalar preconditioners $\setS = \{\alpha \mI : \alpha \in [0, \alpha_{\max}]\}$;
      \item Current iterate \(\vx \in \R^d\); \codecomment{Not used for the step-size version.}
      \item Preconditioner $\mP = \alpha_{\text{bad}} \mI$ that failed the Armijo condition at \(\vx\).
    \end{apxlisting}
    \item
    \item {\bf Output: } Set of scalar preconditioners with reduced interval, $\setS = \{\alpha \mI : \alpha \in [0, \alpha_{\text{bad}}]\}$
  \end{alglisting}
  \end{minipage}
}
\end{figure}

\vfill
\subsection{Separating hyperplanes used by multidimensional backtracking}

Both versions of multidimensional backtracking 
need a direction to update the set of preconditioners
in the $\CUT$ subroutine.
We define the subroutine $\SEPHP$ in \cref{alg:hp}.
The description of the separating hyperplane 
and their properties 
can be found in \cref{sec:direction} and \cref{apx:hyperplanes}.

\vfill

\begin{figure}[h]
  \caption{
  {\bf Separating hyperplane used by both variants of multidimensional backtracking.}
  }
  \label{alg:hp}
  \fbox{\begin{minipage}{\textwidth}
    \begin{alglisting}
      \item $\SEPHP(\vx, \mP)$ for diagonal preconditioners
      \item \hrule
      \item {\bf Input:} 
      \begin{apxlisting}
      \item Current iterate $\vx \in \R^d$;
      \item Diagonal preconditioner $\mP = \Diag(\vp)$ that failed the Armijo condition at \(\vx\).
      \end{apxlisting}
      \item 
      \item $\begin{aligned}
        \vx^+ &= \vx - \mP \nabla f(\vx)\\
         \vg^{\hphantom{+}} &= \nabla f(\vx)\\
         \vg^+ &= \nabla f(\vx^+)\\
        \vv &= 
        \frac{
          \paren{\tfrac{1}{2}\vg - \vg^+} \odot \vg
        }{
          f(\vx) - f(\vx^+) - \lin{\vg, \mP \vg^+}
        }
        \quad\quad\quad\quad\quad\quad\quad\,
        \text{\codecomment{Separating hyperplane from \Cref{lem:starting-hyperplane-main}}}
        \end{aligned}$
      \item
      \item {\bf Output:} $\vu = \max\{\vv, 0\}$ element-wise
      \codecomment{Stronger hyperplane from \Cref{lem:stronger-hp}}
    \end{alglisting}
    \end{minipage}
  }
\end{figure}

\clearpage
\subsection{Multidimensional backtracking using boxes}

The implementation of multidimensional backtracking with boxes only
needs to maintain a vector $\vb$, representing the maximum step-size for
each coordinate that has not been ruled out, in the box $\setB(\vb)$.
The associated sets of preconditioners $\setS$ are \aligns{ \setB(\vb) =
\{\vp \in \Rnn^d : \vp \leq \vb\}, && \setS = \{\Diag(\vp) : \vp \in
\setB(\vb)\}. } The description of boxes and the theoretical guarantees
when using them in multidimensional backtracking can be found in
\cref{sec:cutting-plane} and \cref{apx:subsec-set-box}. The subroutines
used by the algorithm with boxes are:
\begingroup\setlength{\parskip}{8pt}
\begin{itemize}[itemsep=8pt,parsep=0pt,topsep=0pt,partopsep=0pt]
\item 
$\INITSET$: initializes $\vb$ to $c_0 \vone$ 
so that the diagonal preconditioner $c_0 \mI$ is in $\setS_0$.
\item 
$\CANDIDATE$: 
backtracks from the largest diagonal in $\setB(\vb)$,
returning $\gamma \Diag(\vb)$.
\item $\SEPHP$: 
computes the vector $\vu$ 
defining the half-space 
of invalid preconditioners $\setH_{>}(\vu)$ 
obtained when the preconditioner $\mP$ fails the Armijo condition at $\vx$ as described in \cref{lem:starting-hyperplane-main} and \cref{lem:stronger-hp}.
\item 
$\CUT$: 
returns the minimum volume box $\setB(\vb^+)$ 
containing the intersection $\setB(\vb) \cap \setH_{\leq}(\vu)$.
\end{itemize}
\endgroup

\vfill

\begin{figure}[h]
\caption{{\bf Specialization of the subroutines for multidimensional backtracking with boxes}}
\label{alg:subroutines-box}
  \fbox{\begin{minipage}{\textwidth}
    \begin{alglisting}
      \item \(\INITSET(c_0)\)
      \item \hrule
      \item {\bf Input:}
      \begin{apxlisting}
        \item 
        A scalar $c_0$ 
        such that $c_0\mI$ is larger than the optimal diagonal preconditioner, i.e.,
        $\mPstar \preceq c_0\mI$.
      \end{apxlisting}
      \item
      \item {\bf Output:} 
      Set $\setS = \{\Diag(\vp) : \vp \in \setB(\vb)\}$
      with $\setB(\vb) = \{\vp \in \Rnn^d : \vp \leq \vb\}$ 
      where $\vb \coloneqq c_0\vone$
    \end{alglisting}
    \end{minipage}
  }
  ~\\~\\
  \fbox{\begin{minipage}{\textwidth}
  \begin{alglisting}
    \item \(\CANDIDATE(\setS, \gamma, \vx)\)
    \item \hrule
    \item {\bf Input:}
    \begin{apxlisting}
      \item Set of preconditioners $\setS = \{\Diag(\vp) : \vp \in \setB(\vb)\}$
      with $\setB(\vb) = \{\vp \in \Rnn^d : \vp \leq \vb\}$; 
      \item Backtracking coefficient \(\gamma \in [0,1]\);
      \item Current iterate \(\vx \in \R^d\). \codecomment{Not used for the box version.}
    \end{apxlisting}
    \item
    \item {\bf Output: } Preconditioner $\gamma \Diag(\vb)$
  \end{alglisting}
  \end{minipage}
}
~\\~\\
\fbox{\begin{minipage}{\textwidth}
  \begin{alglisting}
    \item \(\CUT(\setS, \vx, \mP)\)
    \item \hrule
    \item {\bf Input:}
    \begin{apxlisting}
      \item Set of preconditioners $\setS = \{\Diag(\vp) : \vp \in \setB(\vb)\}$
      with $\setB(\vb) = \{\vp \in \Rnn^d : \vp \leq \vb\}$; 
      \item Backtracking coefficient \(\gamma \in [0,1]\);
      \item Preconditioner $\mP_{\text{bad}}$ that failed the Armijo
      condition at \(\vx\).
    \end{apxlisting}
    \item
    \item $\vu = \SEPHP(\vx_t, \mP_{\text{bad}})$
    \item $\vb^+ = \max\{\vb, 1/\vu\}$
    \codecomment{Minimum volume box $\setB(\vb^+)$ containing $\setB(\vb) \cap \setH_{\leq}(\vu)$}
    \item
    \item {\bf Output: } Set of diagonal preconditioners $\setS = \{\Diag(\vp) : \vp \in \setB(\vb^+)\}$.
  \end{alglisting}
  \end{minipage}
}
\end{figure}

\vfill

\clearpage
\subsection{Multidimensional backtracking using ellipsoids}
\label{apx:pseudocode-ellipsoid}

The implementation only needs to maintain a vector $\va$ representing
the diagonal of the matrix defining the (centered, axis-alligned)
ellipsoid $\setE(\va)$ and the associated set of preconditioners $\setS$
given by
\aligns{ \setE(\va) = \{\vp \in \Rnn^d : \lin{\vp, \Diag(\va) \vp} \leq
  1\}, && \setS = \{\Diag(\vp) : \vp \in \setE(\va)\}. } The description
  of the ellipsoids and their properties can be found in
  \cref{sec:cutting-plane} and
  \cref{apx:subsec-set-ellipsoid}. The subroutines used by the algorithm
  with boxes are: \begingroup\setlength{\parskip}{8pt}
\begin{itemize}[itemsep=8pt,parsep=0pt,topsep=0pt,partopsep=0pt]
\item 
$\INITSET$: initializes $\va$ to $(\nicefrac{1}{d c_0^2}) \vone$ 
so that  $c_0 \vone \in \setE(\va)$, implying the diagonal preconditioner \(c_0\mI\) is in $\setS$.
\item 
$\CANDIDATE$: 
backtracks from the diagonal preconditioner
in $\setS$ that maximizes the gradient norm.
Let $\setE(\va)$ be the set of candidate diagonals and
define $\mA = \Diag(\va)$. 
The subroutine returns $\gamma \mP_{\max}$, where 
\aligns{
  \mP_{\max} 
  &\coloneqq 
  \argmax_{\mP \in \setS} \norm{\nabla f(\vx)}^2_{\mP}.
  \intertext{Writing this in terms of the diagonal vector
$\vp_{\max} \coloneqq \diag(\mP_{\max})$ yields }
 \vp_{\max} 
 &=
    \argmax_{\vp \in \setE(\va)} \norm{\nabla f(\vx)}^2_{\Diag(\vp)},
  \\
  &= 
     \argmax_{\vp} \lin{\nabla f(\vx)^2, \vp} : \norm{\vp}_{\mA} \leq 1
  =
     \frac{\mA^{-1} \nabla f(\vx)^2}{\norm{\nabla f(\vx)}_{\mA^{-1}}},
}
where $\nabla f(\vx)^2 = \nabla f(\vx) \odot \nabla f(\vx)$.
\item $\SEPHP$: 
computes the vector $\vu$ 
defining the half-space 
of invalid preconditioners $\setH_{>}(\vu)$ 
obtained when the preconditioner $\mP$ fails the Armijo condition at $\vx$ as described in \cref{lem:starting-hyperplane-main} and \cref{lem:stronger-hp}.
\item 
$\CUT$: 
returns an ellipsoid $\setE(\va^+)$ 
containing the intersection of $\setE(\va) \cap \setH_{\leq}(\vu)$ 
with guaranteed volume decrease from \(\setE(\va)\).
As there is no closed-form solution for the minimum volume ellipsoid, 
we set $\va^+$ as 
a convex combination between the original ellipsoid $\setE(\va)$ 
and the minimum volume axis-aligned ellipsoid containing  
$\setH_{\leq}(\vu)$, given by $\setE(\vu^2)$, that is,
\begin{equation}
  \label{eq:ellip_exp_apx}
  \va^+ \coloneqq \lambda \va + (1-\lambda) \vu^2,
  \quad \text{ where } \quad
  \lambda \coloneqq \frac{\ell}{d}\frac{d-1}{\ell-1}
  \quad \text{ and }\quad
  \ell \coloneqq \norm{\vu}^2_{\mA^{-1}},
\end{equation}
where $\mA \coloneqq \diag(\va)$. Although the above choice of
\(\lambda\) has guaranteed volume decrease, we can find a better value of
$\lambda$ by solving the minimum volume ellipsoid as a function of \(\lambda\) numerically. Namely, approximating 
\aligns{
  \lambda^* \coloneqq \argmin_{0 < \lambda < 1} - \log(\det(\lambda \Diag(\va) + (1-\lambda) \Diag(\vu^2))).`'
}
In our experiments, 
we start with $\lambda$ as in~\eqref{eq:ellip_exp_apx} and, starting from it,
we solve the above minimization problem numerically using
L-BFGS-B \citep{zhu1997lbfgsb} in SciPy \citep{2020SciPy}. This preserves the theoretical guarantee while improving empirical performance.

\end{itemize}
\endgroup

\begin{figure}[h]
\caption{{\bf Specialization of the subroutines for multidimensional backtracking with ellipsoids}}
\label{alg:subroutines-ellipsoid}
  \fbox{\begin{minipage}{\textwidth}
    \begin{alglisting}
      \item \(\INITSET(c_0)\)
      \item \hrule
      \item {\bf Input:}
      \begin{apxlisting}
        \item 
        A scalar $c_0 > 0$ 
        such that $c_0\mI$ is larger than the optimal diagonal preconditioner, i.e.,
        $\mPstar \preceq c_0\mI$.
      \end{apxlisting}
      \item
      \item {\bf Output:} 
      $\setS = \{\Diag(\vp) : \vp \in \setE(\va)\}$
      with $\setE(\va) = \{\vp \in \Rnn^d : \lin{\vp, \Diag(\va) \vp} \leq 1\}$
      for $\va = \frac{1}{d c_0^2}\vone$
    \end{alglisting}
    \end{minipage}
  }
  ~\\~\\
  \fbox{\begin{minipage}{\textwidth}
  \begin{alglisting}
    \item \(\CANDIDATE(\setS, \gamma, \vx)\)
    \item \hrule
    \item {\bf Input:}
    \begin{apxlisting}
      \item A set $\setS = \{\Diag(\vp) : \vp \in \setE(\va)\}$
      where $\setE(\va) = \{\vp \in \Rnn^d : \lin{\vp, \Diag(\va) \vp} \leq 1\}$,
      and $\va \in \Rplus^d$;
      \item Backtracking coefficient \(\gamma \in [0,1]\);
      \item Current iterate \(\vx \in \R^d\).
    \end{apxlisting}
    \item
    \item $\displaystyle \vd = \nabla f(\vx) \circ \nabla f(\vx)$
    \item $\displaystyle
    \vp_{\max} = \frac{\mA^{-1} \vd}{\norm{\vd}_{\mA^{-1}}}$
    \codecomment{Where $\mA = \Diag(\va)$}
    \item 
    \item {\bf Output: } Preconditioner $\gamma \Diag(\vp_{\max})$
  \end{alglisting}
 \end{minipage}
}
~\\~\\
\fbox{\begin{minipage}{\textwidth}
  \begin{alglisting}
    \item \(\CUT(\setS, \vx, \mP)\)
    \item \hrule
    \item {\bf Input:}
    \begin{apxlisting}
      \item A set $\setS = \{\Diag(\vp) : \vp \in \setE(\va)\}$,
      where $\setE(\va) = \{\vp \in \Rnn^d : \lin{\vp, \Diag(\va) \vp} \leq 1\}$,
      and $\va \in \Rplus^d$;
      \item Current iterate \(\vx \in \R^d\);
      \item Preconditioner $\mP_{\text{bad}}$ that failed the Armijo condition at \(\vx\).
    \end{apxlisting}
    \vspace{1em}
    \item $\vu = \SEPHP(\vx_t, \mP_{\text{bad}})$
    \item $\displaystyle \ell = \norm{\vu}^2_{\mA^{-1}}$
    \item $\displaystyle
      \lambda = \tfrac{\ell}{d}\tfrac{d-1}{\ell-1}$
      \hfill 
      $
      \paren{
        \text{or numerically solve } 
        \lambda = \arg\min_{0 < c < 1} - \log(\det(c \Diag(\va) + (1-c) \Diag(\vu^2)))
      }
    $
    \item  $\va^+ = \lambda \va + (1-\lambda) \vu^2$
    \codecomment{Approx. min. volume ellipsoid $\setE(\va^+)$ containing $\setE(\va) \cap \setH_{\leq}(\vu)$}
    \item 
    \item {\bf Output: } The set $\setS = \{\Diag(\vp) : \vp \in \setE(\va^+)\}$
  \end{alglisting}
  \end{minipage}
}
\end{figure}

\clearpage
\subsection{Implementable pseudocode}
\label{apx:implementable}

The pseudocode in Figures
\ref{alg:full-pseudocode}--\ref{alg:subroutines-ellipsoid} are expressed
in a modular form to highlight how the algorithm works and its
similarity to a line-search. In this section, we give a more directly
implementable pseudocode of multidimensional backtracking, in both the
box and ellipsoid variants, using mostly vector notation. Scalar
operations on vectors such as $\vu / \va$, $\sqrt{\vu}$, $\vu^2$ are
understood to be taken element-wise.

\vfill

\begin{figure}[h]
\fbox{\begin{minipage}{\textwidth}
  \begin{alglisting}
    \item {\bf Multidimensional backtracking using boxes} \hfill Direct implementation 
    \item \hrule
    \item {\bf Input:}
    \begin{apxlisting}
      \item Function to optimize $f \colon \R^d \to \R$ ;
      \item Starting point $\vx_0 \in \R^d$;
      \item A scalar for the scale of initial set of preconditioners $c_0 > 0$;
      \item Backtracking coefficient $\gamma < \nicefrac{1}{d}$.
    \end{apxlisting}
    \vspace{1em}
    \item $\vb = c_0 \vone$ \fixedwidthcodecomment{Initialize box}{.35\textwidth}
    \item 
    \item Iterate for $t$ in $0, 1, ...$
    \begin{apxlisting}
      \item
      \item $\vp_t= \gamma \vb$
      \fixedwidthcodecomment{Get candidate preconditioner}{.35\textwidth}
      \item 
      \item $\vg_t = \nabla f(\vx_t)$
      \fixedwidthcodecomment{Get candidate point}{.35\textwidth}
      \item $\vx_{t}^+ = \vx_t - \vp_t \circ \vg_t$
      \item
      \item While $f(\vx_t^+) > f(\vx_t) - \frac{1}{2}\lin{\vg_t^2, \vp_t}$
      \fixedwidthcodecomment{Armijo condition fails}{.35\textwidth}
      \begin{apxlisting}
        \item
        \item $\vg_t^+ = \nabla f(\vx_{t}^+)$
        \fixedwidthcodecomment{Get next gradient to compute}{.35\textwidth}
        \item $\vd_t = \frac{1}{2}(\vg_t - \vg_t^+) \circ \vg_t$
        \fixedwidthcodecomment{the separating hyperplane direction,}{.35\textwidth}
        \item $c_t = f(\vx_t) - f(\vx_{t}^+) - \lin{\vg_t \circ \vp_t, \vg_t^+}$
        \fixedwidthcodecomment{the normalization constant,}{.35\textwidth}
        \item $\vu_t = \max\{\vd_t /c_t, 0\}$ (element-wise)
        \fixedwidthcodecomment{and truncate it}{.35\textwidth}
        \item 
        \item $\vb = \vone / \max\{\vone/\vb, \vu_t\}$ (element-wise)
        \fixedwidthcodecomment{Find new minimum volume box.}{.35\textwidth}
        \item
        \fixedwidthcodecomment{($\infty$-free $\min\{\vb, \vone/\vu\}$)}{.35\textwidth}
        \item 
        \item $\vp_t = \gamma\vb$
        \fixedwidthcodecomment{Pick next candidate preconditioner}{.35\textwidth}
        \item
        \item $\vx_{t}^+ = \vx_t - \vp_t \circ \vg_t$
        \fixedwidthcodecomment{and next candidate point}{.35\textwidth}
      \end{apxlisting}
      \item 
      \item $\vx_{t+1} = \vx_t^+$
      \fixedwidthcodecomment{Accept new point}{.35\textwidth}
    \end{apxlisting}
    \item
    \item {\bf Output: } $\vx_t$
  \end{alglisting}
  \end{minipage}
}
\end{figure}

\vfill

\clearpage

\begin{figure}[p]
\fbox{\begin{minipage}{\textwidth}
  \begin{alglisting}
    \item {\bf Multidimensional backtracking using ellipsoids} \hfill Direct implementation 
    \item \hrule
    \item {\bf Input:}
    \begin{apxlisting}
      \item Function to optimize $f \colon \R^d \to \R$ ;
      \item Starting point $\vx_0 \in \R^d$;
      \item A scalar for the scale of initial set of preconditioners $c_0 > 0$;
      \item Backtracking coefficient $\gamma < \nicefrac{1}{\sqrt{d}}$
    \end{apxlisting}
    \vspace{1em}
    \item $\va = \vone / (d c_0^2)$ \fixedwidthcodecomment{Initialize ellipsoid}{.35\textwidth}
    \item 
    \item Iterate for $t$ in $0, 1, ...$
    \begin{apxlisting}
      \item
      \item $\vg_t = \nabla f(\vx_t)$
      \item $\vp_t = \vg_t^2 / \va$ (element-wise)
      \fixedwidthcodecomment{Get candidate preconditioner}{.35\textwidth}
      \item $\vp_t = \gamma \vp_t / \norm{\vg_t^2 / \sqrt{\va}}$ (element-wise)
      \fixedwidthcodecomment{normalize it}{.35\textwidth}
      \item 
      \item $\vx_{t}^+ = \vx_t - \vp_t \circ \vg_t$
      \fixedwidthcodecomment{Get candidate point}{.35\textwidth}
      \item
      \item While $f(\vx_t^+) > f(\vx_t) - \frac{1}{2}\lin{\vg_t^2, \vp_t}$
      \fixedwidthcodecomment{Armijo condition fails}{.35\textwidth}
      \begin{apxlisting}
        \item
        \item $\vg_t^+ = \nabla f(\vx_{t}^+)$
        \fixedwidthcodecomment{Get next gradient to compute}{.35\textwidth}
        \item $\vd_t = \frac{1}{2}(\vg_t - \vg_t^+) \circ \vg_t$
        \fixedwidthcodecomment{the separating hyperplane direction,}{.35\textwidth}
        \item $c_t = f(\vx_t) - f(\vx_{t}^+) - \lin{\vg_t \circ \vp_t, \vg_t^+}$
        \fixedwidthcodecomment{the normalization constant,}{.35\textwidth}
        \item $\vu_t = \max\{\vd_t /c_t, 0\}$ (element-wise)
        \fixedwidthcodecomment{and truncate it}{.35\textwidth}
        \item 
        \item take $\lambda = \frac{\ell (d-1)}{d (\ell - 1)}$ where $\ell = \lin{\vu^2, 1/\va}$
        \fixedwidthcodecomment{Approx. min. vol. new ellipsoid}{.35\textwidth}
        \item or 
        \fixedwidthcodecomment{}{.35\textwidth}
        \item find $\lambda$ by numerically minimizing $\phi(\lambda)$ where 
        \fixedwidthcodecomment{Find better approximation of min.}{.35\textwidth}
        \item $\quad \phi(\lambda) = - \sum_{i=1}^d \log(\lambda \va[i] + (1-\lambda) \vu_t[i]^2)$
        \fixedwidthcodecomment{of volume of new ellipsoid}{.35\textwidth}
        \item 
        \item $\va = \lambda \va + (1-\lambda) \vu^2$
        \fixedwidthcodecomment{New ellipsoid}{.35\textwidth}
        \item 
        \item $\vp_t = \vg_t^2 / \va$ (element-wise)
        \fixedwidthcodecomment{Get new candidate preconditioner,}{.35\textwidth}
        \item $\vp_t = \gamma \vp_t / \norm{\vg_t^2 / \sqrt{\va}}$ (element-wise)
        \fixedwidthcodecomment{normalized, }{.35\textwidth}
        \item $\vx_{t}^+ = \vx_t - \vp_t \circ \vg_t$
        \fixedwidthcodecomment{and new candidate point}{.35\textwidth}
      \end{apxlisting}
      \item 
      \item $\vx_{t+1} = \vx_t^+$
      \fixedwidthcodecomment{Accept new point}{.35\textwidth}
    \end{apxlisting}
    \item
    \item {\bf Output: } $\vx_t$
  \end{alglisting}
  \end{minipage}
}
\end{figure}

\clearpage

\begin{figure}[h]
\centering
\includegraphics[width=.6\textwidth]{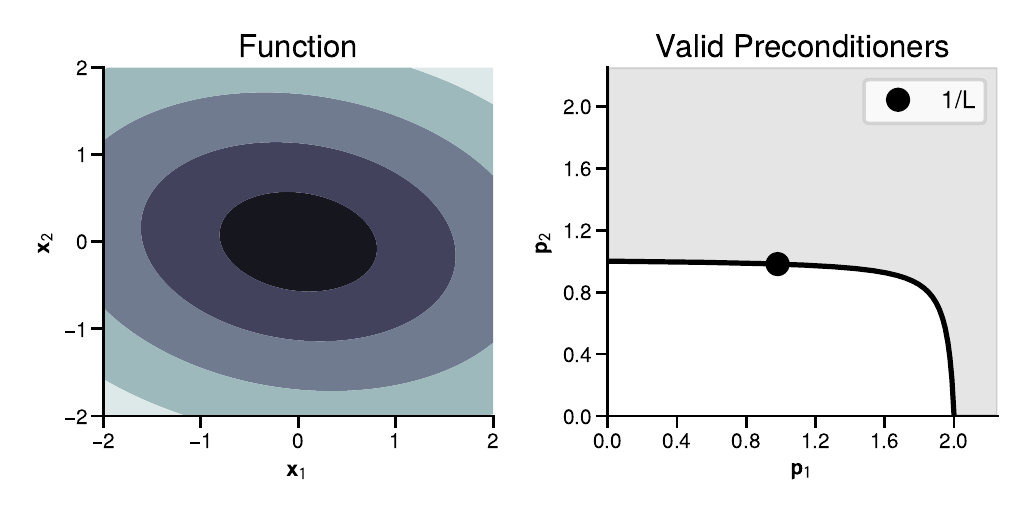}
\vspace{-1em}
\captionsetup{singlelinecheck=off}
\caption[Set of valid diagonal preconditioners for a quadratic]{%
Set of valid diagonal preconditioners (step-sizes $\vp_1$ and $\vp_2$)
for the quadratic in \cref{eq:pareto-1}.
Preconditioned gradient descent can use a larger step-size in the first coordinate.
}
\label{fig:pareto-1}
\end{figure}

\section{Optimal preconditioners, valid preconditioners and competitive ratios}
\label{apx:convergence-opt-p}
\label{apx:without-sc}
\label{apx:smoothness}
\label{apx:competitive}

In \Cref{sec:optimal}, we defined the optimal preconditioner $\mPstar$
as the preconditioner that is the best overall approximation to the inverse
Hessian. Formally, we define the optimal diagonal preconditioner \(\mPstar\)  as
\aligns{
  \mPstar &\coloneqq \argmin_{\mP\succ 0, \text{diagonal}} \kappa \quad \text{such that} \quad \frac{1}{\kappa}\mP^{-1}\preceq \nabla^2 f(\vx) \preceq \mP^{-1} \text{ for all } \vx.
  \tag{\ref{eq:optimal-preconditioner-def}}
}
One way to interpret this definition is that $\mPstar^{-1}$ is the tightest diagonal approximation
to $\nabla^2 f(\vx)$.

We remark that we do not need \(f\) to be (strongly-)convex to define the
theoretically optimal step-size of $\nicefrac{1}{L}$ for gradient descent. 
Thus, one may wonder why we need strong-convexity (although we
relax this to requiring \(f\) to be PL in \cref{apx:pl}) to define what an optimal
preconditioner is in~\eqref{eq:optimal-preconditioner-def}. 

The main difference between the scalar step-size and per-coordinate step-sizes 
settings is whether the ``largest'' step-size or preconditioner is well-defined.
In the scalar setting, the largest step-size that is guaranteed to lead to progress everywhere 
(i.e., a step-size that satisfies the Armijo condition \eqref{eq:armijo} for all \(\vx\)) 
is well-defined and equal to \(\alpha_* \coloneqq \nicefrac{1}{L}\) for \(L\)-smooth function \(f\). 
Equivalently, 
\aligns{
  \alpha_{*} =
  \sup\setst*{ \alpha > 0}{\nabla^2 f(\vx) \preceq \frac{1}{\alpha} \mI}
  = \sup_{\vx \in \R^d} \lambda_{\max}(\nabla^2 f(\vx)),
}
where \(\lambda_{\max}(\nabla^2 f(\vx))\) is the largest eigenvalue of
\(\nabla^2 f(\vx)\). But in the case of preconditioners, the ordering on
positive definite matrices is not complete, so there is no single
``largest'' preconditioner \(\mP\) that satisfies $\nabla^2 f(\vx) \preceq \mP^{-1}$. 
We can still describe ``good'' preconditioners, that are guaranteed to satisfy the Armijo condition
(\cref{eq:sufficient-progress}) everywhere; this is the notion
of valid preconditioners defined in \cref{def:invalid}, 
which in set notation is $\setV \coloneqq \{\mP \succ 0: \nabla^2 f(\vx) \preceq \mP^{-1}\}$. 
With this definition, we can consider the set of valid
preconditioners $\mP$ for which there are no bigger valid
preconditioners, that is, $\setP \coloneqq \{\mP \in \setV : \not\exists
\, \mP' \in \setV~\text{s.t.}~\mP \prec \mP'\}$. However, \(\setP\) contains
incomparable preconditioners, that is, distinct matrices \(\mA, \mB \in
\setP\) that neither \(\mA \succeq \mB\) nor \(\mA \preceq \mB\) hold.

Let us look at an example with a quadratic function
(illustrated in \cref{fig:pareto-1})
\alignn{
  f(\vx) = \frac{1}{2}\lin{\vx, \mA \vx}
  && \text{ with Hessian } &&
  \mA = \bmat{
     .5 & .1\\
     .1 & 1.0
  }.
  \label{eq:pareto-1}
}
There are many preconditioners that are valid,\footnote{
  Up to invertibility issues which we address in the next subsection.
}
for example using the per-coordinate step-sizes
\aligns{
  \mP_L \approx \bmat{
      .91 & 0\\
      0 & .91
  },
  &&
  \mP_1 =\bmat{
      2.0 & 0\\
      0 & 0.0
  }, 
  &&
  \mP_2 =\bmat{
      0.0 & 0\\
      0 & 1.0
  },
  &&
  \mPstar \approx \bmat{
      1.75 & 0\\
      0 & 0.87
  }. 
}
The preconditioner $\mP_L$
corresponds to the $\nicefrac{1}{L}$ step-size, 
$\mP_1$ and $\mP_2$ take the largest possible step-size 
in each coordinate,
and $\mPstar$  is the optimal preconditioner 
according to \cref{eq:optimal-preconditioner-def}.
Those preconditioners are not comparable to each other, 
as neither $\mP_L \prec \mPstar$
nor $\mPstar \prec \mP_L$ hold.
Instead of looking at the matrices themselves, 
we use in~\eqref{eq:optimal-preconditioner-def} the condition number\footnote{Our definition is slightly different, but both notions are equivalent for positive definite \(\mP\).} of \(\smash{\mP^{1/2} \nabla^2 f(\vx) \mP^{1/2}}\) as a measure of quality of \(\mP\). 
This allows for a well-defined optimal preconditioner
as this condition number can be maximized. 

\subsection{Defining optimal preconditioners without twice-differentiability or strong-convexity}
\label{apx:pl}
Although we used twice-differentiability of $f$ 
to define the optimal preconditioner, this is not necessary. 
If $f$ is not twice-differentiable but still strongly-convex, 
the definition in \Cref{eq:optimal-preconditioner-def}
can be replaced by \cref{eq:relative-smooth-strong-convexity},
as finding the $\mP$-norm under which the function is most strongly-convex.
\aligns{
  \mPstar &= \argmin_{\mP \succ 0, \text{ diagonal}} \kappa 
  \\
  &\text{such that }
  \left\{
  \begin{array}{l}
  \frac{1}{\kappa} \frac{1}{2}\norm{\vx - \vy}^2_{\mP^{-1}} \leq
  f(\vy) - f(\vx) - \lin{\nabla f(\vx), \vy - \vx },
  \\
  f(\vy) - f(\vx) - \lin{\nabla f(\vx), \vy - \vx } \leq
  \frac{1}{2}\norm{\vy - \vx}^2_{\mP^{-1}}, 
  \end{array}
  \right.
  \quad \text{ for all } \vx, \vy. 
}
To avoid strong-convexity, 
we can instead use the PL inequality.
A function $f$ is $\mu$-PL if
\alignn{
  \frac{1}{\mu} \frac{1}{2}\norm{\nabla f(\vx)}^2 \geq f(\vx) - f(\vx_*).
  \label{eq:pl}
}
This property is implied by $\mu$-strong convexity.
We refer to the work of \citet{karimi2016pl} for the properties of PL functions 
and its relation to other assumptions.
To adapt \cref{eq:pl} to our results, 
we can measure the PL constant $\mu$ 
in the norm induced by $\mP$, 
and say that $f$ is $\mu$-PL in $\norm{\cdot}_\mP$ if
\alignn{
  \frac{1}{\mu} \frac{1}{2}\norm{\nabla f(\vx)}^2_{\mP} \geq f(\vx) - f(\vx_*).
  \label{eq:pl-P}
} We use this inequality in the convergence proof in
\cref{prop:convergence} since it is a consequence of strong-convexity.
As this property is the only property of strong-convexity needed for our
results, we can adapt our results to be competitive with the optimal
preconditioner defined using the PL inequality, using the definition
\alignn{
  \begin{aligned}
  \mPstar^{\text{\textsc{pl}}} &:= \argmin_{\mP \succ 0, \text{ diagonal}} \kappa 
  \\
  &\text{such that }
  \left\{
  \begin{array}{ll}
    \frac{1}{\kappa} \norm{\nabla f(\vx)}_{\mP}^2 \geq f(\vx) - f(\vx_*)
    & \text{ for all } \vx,
    \\
    f(\vy) - f(\vx) - \lin{\nabla f(\vx), \vy - \vx } \leq
    \frac{1}{2}\norm{\vy - \vx}^2_{\mP^{-1}}, 
    & \text{ for all } \vx, \vy. 
  \end{array}
  \right.
  \end{aligned}
  \label{eq:pl-definition}
}
If $f$ is $\mu$-PL and $L$-smooth, 
\cref{eq:pl-definition} has a feasible solution 
at $\mP = \nicefrac{1}{L}\mI$ 
number $\kappa = \nicefrac{L}{\mu}$.
The constraint based on the $\mu$-PL condition 
in \cref{eq:pl-definition}
is weaker than the definition using strong-convexity,
as strong-convexity implies the PL inequality. 
The optimal preconditioner defined
using the PL inequality \eqref{eq:pl-definition}
might thus achieve a lower condition number 
than the one using strong-convexity \eqref{eq:optimal-preconditioner-def}.
For example,
the quadratic $f(\vx) = (\nicefrac{1}{2})\lin{\vx, \mA \vx}$
with a positive semi-definite $\mA$
is not strongly convex
if the smallest eigenvalue of $\mA$ is 0.
The optimal preconditioner in \cref{eq:optimal-preconditioner-def} is ill-defined
(or has condition number $\kappa_* = \infty$).
In contrast, the optimal preconditioner defined using the PL inequality 
in \cref{eq:pl-definition}
has a finite condition number, 
as $\mP = \nicefrac{1}{L}\mI$ is a feasible solution
with condition number $\kappa = \nicefrac{L}{\lambda_{\min}^+(\mA)}$
where $\smash{\lambda_{\min}^+(\mA)}$
is the smallest non-zero eigenvalue of $\mA$.
As our proofs only use the properties guaranteed by \cref{eq:pl-definition},
our results also apply to PL functions.

\subsection{Valid and optimal preconditioners with singular matrices}
\label{apx:division-by-0}
\label{apx:interible_precs}
In the main text, we defined valid preconditioners (\cref{def:invalid}) only for 
positive definite matrices for ease of presentation. 
The notion of valid 
preconditioners 
can be extended to general positive semidefinite matrices. 
In the diagonal case, the convention \(1/0 = +\infty\) is a useful mental model
but can cause inconsistencies (such as \(\infty \cdot 0\)).
To extend the notion of valid preconditioners to general positive semidefinite matrices,
we can use the definition
\begin{definition}
  \label{def:valid_singular}
  A preconditioner \(\mP \succeq 0\) is \emph{valid} if \(\mP^{1/2} \nabla^2 f(\vx) \mP^{1/2} \preceq I\) for all \(\vx \in \R^d\).
\end{definition}
The above is well-defined for all positive semidefinite matrices.
An alternative to arrive at a definition 
closer to \cref{def:invalid}
is to consider  
the projection matrix \(\Pi_{\mP}\) onto the image of \(\mP\),
given by \(\Pi_{\mP} = \mP^{1/2}(\mP^{1/2})^{\dagger}\)
where \(\mP^{\dagger}\) is the Moore-Penrose pseudo-inverse of \(\mP\).
Using that, one can show that \(\mP\) is \emph{valid} (according to \cref{def:valid_singular}) if and only if 
\begin{equation*}
  \Pi_{\mP} \nabla^2 f(\vx) \Pi_{\mP} \preceq \mP^{\dagger}\qquad \text{for all}~\vx \in \R^d.
\end{equation*}
An example of a valid preconditioner that is covered by \cref{def:valid_singular} but not \ref{def:invalid} 
is the all-zeroes matrix. 
\cref{def:valid_singular} can seamlessly replace \ref{def:invalid}, and all the results follow similarly. 
Moreover, notice that the optimization problem
defining the optimal preconditioner~\eqref{eq:optimal-preconditioner-def} 
may not attain its minima on positive definite matrices when \(f\) is not strongly convex. 
In this case, we can define an optimal preconditioner 
as a limit point of a sequence that
attains in the limit the value in~\eqref{eq:optimal-preconditioner-def}
by replacing the minimum with an infimum. 
In this case, an optimal preconditioner may be singular, 
but the results in the main body also follow seamlessly using this definition. 
We decided to restrict our attention to non-singular preconditioners in the main paper for ease of exposition,
since when \(f\) is strongly-convex, an optimal preconditioner is always non-singular.

\subsection{Best competitive ratio achievable by the optimal preconditioner}

In \cref{sec:convergence}, 
we mentioned that the optimal preconditioner $\mPstar$ 
could be only $\nicefrac{1}{d}$-competitive. 
In fact, the competitive ratio of $\mPstar$ can be arbitrarily bad.
The reason for this is that the competitive ratio $\gamma$ does not
compare against $\mPstar$, but rather against any $\mP$ in the set
$\setS$ of potentially valid preconditioners. Moreover, this definition
only takes into account the norm \(\norm{\nabla f(\vx)}_{\mP}\) at a
\emph{fixed} \(\vx\), while the optimal preconditioner needs to have large norm
\emph{for all} \(\vx\).

For example, consider the scalar step-size case. If our current interval
of candidate step-sizes to try is $\setS = [0,1]$ but the optimal
step-size $\alpha_*$ is small, let us say $\alpha_* = \nicefrac{1}{10}$,
then \(\alpha_*\) is only $\nicefrac{1}{10}$-competitive in $\setS$. The
motivation for this definition of competitive ratio is that we cannot
check whether $\alpha$ is large compared to $\alpha_*$ (as we do not
know $\alpha_*$) but we can more easily ensure that a candidate
step-size $\alpha$ is $\gamma$-competitive in $\setS$ (for example
$\alpha = \nicefrac{1}{2}$ is $\nicefrac{1}{2}$-competitive in $[0,1]$).

In the previous example, the bad competitive ratio of $\alpha_*$ in $\setS$  
was mostly due to the fact that $\setS$ was large and that, for some
\(\vx\), step sizes larger than \(\alpha_*\) could satisfy the Armijo
condition~\eqref{eq:armijo}. Even if $\alpha_*$ is globally
optimal, we could make more progress by using a larger step-size if they were to be accepted, 
and we have not yet ruled out those step-sizes.
However, as $\setS$ shrinks, 
it may eventually converge to the interval $[0,1]$,
in which case the optimal step-size $\alpha_*$ would be $1$-competitive.

{\bf In high dimensions however,} the optimal preconditioner can have a
competitive ratio of $\nicefrac{1}{d}$ even when comparing only against
valid preconditioners.\footnote{%
How small the set $\setS_t$ can get is bounded by construction. 
The cutting plane procedure in \cref{sec:direction,sec:cutting-plane} 
only remove invalid preconditioners. 
The valid preconditioners contained in the initial set $\setS_0$ 
will always be in $\setS_t$, 
along with possibly more preconditioners 
that have not been deemed invalid over the course of optimization. 
} This is because the competitive ratio is
defined using the $\mP$-norm of the gradient, and we need to take the
direction of the gradient into account. For example, consider the
quadratic function (illustrated in \cref{fig:worst-case})
\begin{figure}[t]
\centering
\includegraphics[width=.6\textwidth]{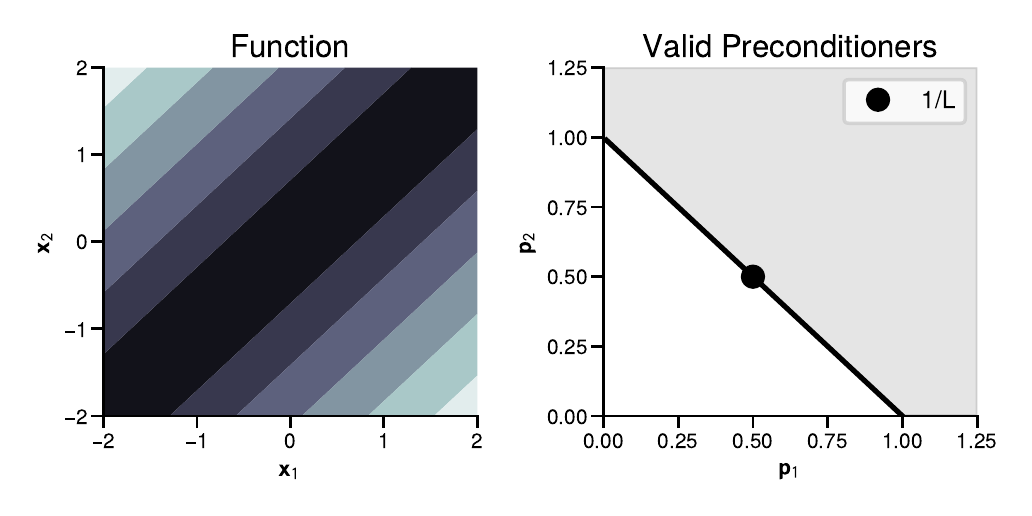}
\vspace{-1em}
\caption[Set of valid diagonal preconditioners for a quadratic]{%
Set of valid diagonal preconditioners (step-sizes $\vp_1$ and $\vp_2$)
for the quadratic in \cref{eq:worst-case-preconditioner}.
The set of valid preconditioners (\cref{def:invalid})
is the white region in the right figure.
}
\label{fig:worst-case}
\end{figure}
\alignn{
  f(\vx) = \frac{1}{2}\lin{\vx, \mA \vx}
  &&\text{ where }&&
  \mA = \bmat{
    1 & -1\\
    -1 & 1
  },
  \label{eq:worst-case-preconditioner}
}
with eigenvalues $\{2,0\}$ 
as $\mA = [-1,1]^\transp [-1,1]$.
The following three preconditioners are all valid:
\aligns{
  \mP_1 = \bmat{
    1 & 0 \\
    0 & 0
  },
  &&
  \mP_2 =
  \bmat{
    0 & 0 \\
    0 & 1
  },~\text{and}
  &&
  \mPstar=
  \bmat{
    \nicefrac{1}{2} & 0 \\
    0 & \nicefrac{1}{2}
  }.
}
The preconditioner $\mP_1$ 
takes the largest possible step-size in the first coordinate 
and ignores the second, while $\mP_2$ does the opposite. 
They are not good global preconditioners, 
as each ignores one coordinate. 
Yet, they can make much more progress (i.e., the objective value may decrease more) than the optimal preconditioner $\mPstar$
if the gradient is very skewed towards one coordinate. 
This implies that
$\mPstar$ may be only $\nicefrac{1}{2}$-competitive
in  $\{\mP_1, \mP_2\}$ for some \(\vx\) since 
\aligns{
  &&\text{if } \nabla f(\vx) = \bmat{1\\0},~\text{then}
  &&
  \norm{\nabla f(\vx)}_{\mP_1}^2 = 1,
  &&
  \norm{\nabla f(\vx)}_{\mP_2}^2 = 0,
  &&
  \norm{\nabla f(\vx)}_{\mPstar}^2 = \nicefrac{1}{2},
  \\
  \text{and}~&&\text{if } \nabla f(\vx) = \bmat{0\\1},
  ~\text{then}
  &&
  \norm{\nabla f(\vx)}_{\mP_1}^2 = 0,
  &&
  \norm{\nabla f(\vx)}_{\mP_2}^2 = 1,
  &&
  \norm{\nabla f(\vx)}_{\mPstar}^2 = \nicefrac{1}{2}.
}
The preconditioner $\mPstar$ is still a better choice globally (i.e, for
all \(\vx\)) since it ensures optimal worst-case linear rate in
preconditioned gradient descent. 
But there are better preconditioners that depend on the current gradient. 
We exploit this in the ellipsoid variant of multidimensional backtracking to improve our
competitive ratio. 
We backtrack from the preconditioner that maximizes the local progress guarantees 
to ensure a  $\nicefrac{1}{\sqrt{d}}$ competitive ratio,
while ensuring volume shrinkage of the set of candidate preconditioners 
when we call \(\CUT\), if the preconditioner fails the Armijo condition.

\clearpage
\section{Separating hyperplanes}
\label{apx:hyperplanes}
\label{apx:standard-hyperplanes}
\label{apx:strong-hyperplanes}
\label{apx:no-convexity-is-hard}
\label{apx:alternative-hyperplane-in-inverse}

In this section, we prove Propositions \ref{lem:starting-hyperplane-main} and \ref{lem:stronger-hp} on existence and strengthening
of separating hyperplanes for valid preconditioners.

{\bf General idea.} 
Let us start with a summary of the 
separating hyperplanes used to search for good preconditioners as discussed in \Cref{sec:convergence,sec:direction}.
The goal of the separating hyperplanes 
is to give us ways to shrink the initial set of potential preconditioners $\setS$ to narrow in on valid preconditioners
using the cutting-plane methods in \cref{sec:cutting-plane}. 
At each iteration we are looking for preconditioners \(\mP\)
that satisfy the Armijo condition at \(\vx\) given by
\aligns{
  f(\vx - \mP \nabla f(\vx))
  \leq 
  f(\vx) - \frac{1}{2}\norm{\nabla f(\vx)}^2_\mP.
}
If \(\mP\) fails the Armijo condition, we conclude that \(\mP\) is invalid. 
To obtain more information, we look at the condition as a function of the (diagonal of the) preconditioner,
and define the gap function at $\vx$,
\aligns{
  h(\vp) \coloneqq f(\vx - \Diag(\vp)\nabla f(\vx)) - f(\vx) + \frac{1}{2}\norm{\nabla f(\vx)}_{\Diag(\vp)}^2, \qquad \forall \vp \in \Rnn^d.
}
Then, $h(\vp) \leq 0$ if $\mP = \Diag(\vp)$ satisfies the Armijo
condition at $\vx$, and $h(\vp) > 0$ otherwise. 
Any preconditioner \(\Diag(\vq)\) such that $h(\vq) > 0$ 
is guaranteed to be invalid.
We can use the gradient of $h$ at \(\vp\) 
and convexity to find a half-space 
such that one side contains only 
preconditioners with $h(\vp) > 0$. 
In this section, we show how to construct such half-space,
and strengthen them using the partial order on matrices,
which is needed to ensure volume shrinkage of our cutting plane methods.

\subsection{Stronger hyperplanes}
In the main body we presented the
strengthening of separating hyperplanes via truncation (\cref{lem:stronger-hp})
after the result of existence of separating hyperplanes (\cref{lem:starting-hyperplane-main}).
Here, we prove a more general lemma on strengthening half-spaces of
invalid preconditioners first, 
as it is useful in simplifying the proof of \cref{lem:starting-hyperplane-main}. 
\cref{lem:stronger-hp} follows directly from the following lemma.

\begin{lemma}
  \label{lemma:general-stronger-hyp}
  Let $\setH_{\vv, \alpha}$
  be the intersection 
  of the non-negative orthant $\Rnn^d$
  and the half-space defined by the vector $\vv \in \R^d$ and coefficient $\alpha > 0$,
  \begin{equation*}
    \setH_{\vv,\alpha} \coloneqq \setst{\vp \in \Rnn^d}{\lin{\vv,\vp} > \alpha}.
  \end{equation*}
  Define \(\vu \coloneqq \max\{\vv, 0 \} \) and let $\setH_{\vu, \alpha}$ 
  be defined similarly as above, that is,
  \begin{equation*}
    \setH_{\vu, \alpha} \coloneqq \setst{\vp \in \Rnn^d}{\lin{\vu,\vp} > \alpha}.
  \end{equation*}
  If $\setH_{\vv,\alpha}$ 
  only contains diagonals of invalid preconditioners, 
  that is, \(\Diag(\vp)\) is invalid for any \(\vp \in \setH_{\vv}\),
  Then \(\setH_{\vv, \alpha} \subseteq \setH_{\vu, \alpha}\) and \(\setH_{\vu, \alpha}\)
  only contains diagonals of invalid preconditioners.
\end{lemma}
\begin{proof}
  {\bf Inclusion $\setH_{\vv,\alpha}\subseteq \setH_{\vu,\alpha}$.}
  We have that $\lin{\vp, \vv} > \alpha$ implies $\lin{\vp,\vu} > \alpha$
  for any $\vp \in \Rnn^d$ since
  \begin{equation*}
    \lin{\vv, \vp} 
    = \sum_{i \colon \vv[i] \geq 0} \vv[i] \vp[i]
   + \sum_{i \colon \vv[i] < 0} \vv[i] \vp[i]
   \leq \sum_{i \colon \vv[i] \geq 0} \vv[i] \vp[i]
   = \sum_{i \colon \vv[i] \geq 0} \vu[i] \vp[i]
   = \lin{\vu, \vp}.
  \end{equation*}
  {\bf $\setH_{\vu, \alpha}$ only contains invalid diagonals.} Let
  \(\vp_{\vu} \in \setH_{\vu, \alpha}\). We can show that
  $\Diag(\vp_{\vu})$ is invalid by finding \( \vp_{\vv} \in \setH_{\vv, \alpha}
  \) such that $\Diag(\vp_{\vv}) \preceq \Diag(\vp_{\vu})$. Since
  \(\Diag(\vp_{\vv})\) is invalid by assumption, this would imply that
  $\Diag(\vp_{\vu})$ is also invalid. To find $\vp_{\vv}$, we can
  truncate the entries of $\vp_{\vu}$ as 
  \begin{equation*}
    \vp_{\vv}[i] \coloneqq \begin{cases}
      \vp_{\vu}[i] &\text{if}~\vv[i] \geq 0\\
      0 &\text{otherwise},
    \end{cases}
    \qquad \forall i \in \{1, \dotsc, d\}.
  \end{equation*}
  Then \(\vp_{\vv} \in \setH_{\vv, \alpha}\) since
  $\alpha
   < \lin{\vu, \vp_{\vu}}
   = \lin{\vu, \vp_{\vv}}
   = \lin{\vv, \vp_{\vv}}$.
  \footnote{
    One may worry that our original
    definition of valid preconditioners 
    has a division by \(0\)
    if any entry of the preconditioner is \(0\)
    as a preconditioner is valid if $\nabla^2 f(\vx) \preceq \mP^{-1}$
    (\cref{def:invalid}). 
    It is enough to use the convention that \(1/0 = +\infty\),
    although this might lead to inconsistencies. 
    In \cref{apx:division-by-0} we 
    discuss a more general definition without the use of infinities.
  } 
  and \(\Diag(\vp_{\vu}) \succeq \Diag(\vp_{\vv})\), as desired.
\end{proof}

\subsection{Separating hyperplanes for invalid preconditioners}

We are now in position to prove \cref{lem:starting-hyperplane-main}.

\begin{proof}[Proof of \Cref{lem:starting-hyperplane-main}]

Throughout the proof, we shall denote by \(\mP\) the matrix
\(\Diag(\vp)\). If $f$ is convex, then $h$ also is since the map $\vp
\in \Rnn^d \mapsto f(\vx - \mP \nabla f(\vx))$ is the composition of an
affine transformation and a convex function, and $\norm*{\nabla
f(\vx)}_{\mP}^2 = \lin{\nabla f(\vx), \Diag(\vp) \nabla f(\vx) }$ is
linear in $\vp$. Convexity of \(h\) yields the inequality
\begin{equation*}
  h(\vp) \geq h(\vq) + \lin{\nabla h(\vq), \vp - \vq},
  \qquad \forall \vp \in \Rnn^d.
\end{equation*}
 This implies that if \(\vp\) is such that $h(\vq) +
\lin{\nabla h(\vq), \vp - \vq} > 0$, then $h(\vp) > 0$, which implies
that \(\Diag(\vp)\) is an~\invalid{} preconditioner. 
Rearranging we conclude that \(\Diag(\vp)\) is invalid for all
\(\vp\) in the set in~\eqref{eq:part-1}, i.e., in
\begin{equation}
  \label{eq:part-1-apx}
 \setst{\vp \in \Rnn^d}{ \lin{\nabla h(\vq), \vp} > \lin{\nabla h(\vq), \vq} - h(\vq)}
\end{equation}
We express the above half-space as 
\aligns{
  \setH_{>}(\vv) = \{\vp : \lin{\vp,\vv} > 1\} \text{ for }
  \vv \coloneqq \frac{\nabla h(\vq)}{\paren{\lin{\nabla h(\vq), \vq} - h(\vq)}}.
}
Yet, for \(\setH_{>}(\vv)\) to be equivalent to the set in~\eqref{eq:part-1-apx} 
or even to be well-defined, we need to ensure $\lin{\nabla h(\vq), \vq} - h(\vq) > 0$. 
To see that this holds, note first that by convexity of
\(h\) and that fact that \(h(0) = 0\) we have
\begin{equation*}
  h(0) \geq h(\vq) + \lin{\nabla h(\vq), 0 - \vq} 
  \implies 
  \lin{\nabla h(\vq), \vq - 0} - h(\vq)  \geq - h(0) = 0
\end{equation*}
To show that the last inequality is strict, assume that 
\(\lin{\nabla h(\vq), \vq - 0} - h(\vq) = 0\) 
for the sake of contradiction. 
By \cref{lemma:general-stronger-hyp}, the half-space 
\(\smash{\setH \coloneqq \setst{\vp \in \Rnn^d}{\lin{{[\nabla h(\vx)]}_+, \vp} > 0}} \)
contains only diagonals of invalid preconditioners, 
where \(\smash{{[\nabla h(\vx)]}_+ \coloneqq \max\{\nabla h(\vx), 0\}}\) entry wise. 
However, \(\smash{(\nicefrac{1}{L})\vone \in \setH}\) as \([\nabla h(\vx)]_+ \geq 0 \)
and should be invalid,
which is a contradiction since \(f\) is \(L\)-smooth 
and $\nicefrac{1}{L}\mI$ is valid. 
Therefore, \(\lin{\nabla h(\vq), \vq - 0} - h(\vq) > 0\).

Finally, we can write \(\vv\) in terms of \(f\) and \(\mQ\). To do so,
first define \(\vx^+ \coloneqq \vx - \mQ \nabla f(\vx)\), and the
gradients of \(f\) at different points by \(\vg \coloneqq \nabla
f(\vx)\) and \(\vg^+ \coloneqq \nabla f(\vx^+)\). Then, by the
chain-rule,
\begin{align*}
  \nabla h(\vq) = -\nabla f(\vx - \mQ \nabla f(\vx)) \odot \nabla f(\vx) + \frac{1}{2} \nabla f(\vx) \odot \nabla f(\vx) = 
  - \vg^+ \odot \vg + \frac{1}{2}\vg \odot \vg,
\end{align*}
which implies
\begin{align*}
  \lin{\nabla h(\vq), \vq} - h(\vq)
  &= - \lin{\vg^+, \mQ \vg} + \frac{1}{2}\lin{\vg, \mQ \vg}
  - f(\vx^+) + f(\vx) + \frac{1}{2} \lin{\vg, \mQ \vg}
  \\
  &= f(\vx)  - \lin{\vg^+, \mQ \vg} - f(\vx^+).
\end{align*}
Plugging these equations in the definition of \(\vv\) yields
\begin{equation*}
  \vv = \frac{\nabla h(\vq)}{\lin{\nabla h(\vq),\vq} - h(\vq)} 
  =  \frac{
    (\frac{1}{2}\vg - \vg^+) \odot \vg
  }{
    f(\vx) 
    - \lin{\vg^+, \mQ \vg} 
    - f(\vx^+) 
  }.
  \qedhere
\end{equation*}
\end{proof}

\paragraph{Remark on assumptions of
\Cref{lem:starting-hyperplane-main}.} One may have noticed that we never
use the assumption that \(\mQ\) fails the Armijo condition (i.e., that
\(h(\vq) > 0\)) in the proof of the proposition. In fact, the
proposition holds for any \(\smash{\vq \in \Rnn^d}\). However, and crucially for
our application, we have that \(\vq\) is in the half-space
\(\setH_{>}(\vu)\) of invalid diagonals from
\Cref{lem:starting-hyperplane-main}. In multidimensional backtracking,
\(\vq\) is the diagonal of a preconditioner \(\Diag(\vq)\) that failed
the Armijo condition \(h(\vq) > 0\). Since \(\vq\) is close to the
origin in multidimensional backtracking, we can ensure the half-space
\(\setH_{>}(\vu)\) contains a significant portion of our current set of
candidate preconditioners, leading to significant shrinkage of the set of candidate preconditioners whenever \(\CUT\) is invoked.

\clearpage
\section{Cutting-plane methods}
\label{apx:set-box}
\label{apx:set-simplex}
\label{apx:set-ellipsoid}
\label{apx:refinement}

\subsection{Boxes}
\label{apx:subsec-set-box}

Given a box \(\setB(\vb)\) for some \(\vb \in \Rnn^d\) and a vector
\(\vu \in \Rnn^d\), our cutting plane method needs to find a box
\(\setB(\vb^+)\) that contains \( \setB(\vb) \cap \setH_{>}(\vu)\)
which, hopefully, has smaller volume than \(\setB(\vb)\). 

The next lemma gives a formula for the \emph{minimum volume} box for any \(\vu\), which is
used in the main text to define \(\CUT\) in \cref{eq:wtv}.
Moreover, we show that if
the half-space \(\setH_{>}(\vu)\) is close enough to the origin (since otherwise we might have \(\vb^+ = \vb\)), then we have a significant volume decrease.

\begin{lemma}
  \label{lemma:box_vol_shrink_apx}
  Let \(\vb \in \Rnn^d\) and \(\vq \in \setB(\vb)\). Let \(\vu \in
  \Rnn^d\). Then the box \(\setB(\vb^+)\) with minimum volume that contains \(\setB(\vb) \cap \setH_{\leq}(\vu)\) is given by
  (using the convention that \(1/\vu[i] = + \infty\) if \(\vu[i] = 0\))
  \begin{equation}
    \label{eq:mvb_formula_apx}
    \vb^+[i] \coloneqq \min\{\vb[i], \nicefrac{1}{\vu[i]}\}, \qquad
    \forall i \in \{1, \dotsc, d\},
  \end{equation}
  Moreover, if \((\nicefrac{1}{2d}) \cdot \vb\) is excluded by the half-space, that is,
  \(\vb \in \setH_{>}(\vu)\), 
  then \(\Vol(\setB(\vb^+)) \leq (\nicefrac{1}{(d + 1)})\Vol(\setB(\vb^+))\).
\end{lemma}
\begin{proof}
  {\bf Formula for \(\vb^+\).} 
  Finding the minimum volume box containing \(\setB(\vb) \cap \setH_{\leq}(\vu)\),
  \begin{align*}
    \vb^+ = \argmin_{\vc \in \R^d} \Vol(\setB(\vc))
    \quad \text{s.t.}~ 
    \setB(\vb) \cap \setH_{\leq}(\vu) \subseteq \setB(\vc),
  \end{align*}
  is equivalent to finding the solution to the following optimization problem:
  \begin{align*}
    \vb^+ = \argmin_{\vc \in \R^d} \prod_i \vc[i] %
    \quad \text{s.t.}~ \max_{\vp \in \setB(\vc) \cap \setH_{\leq}(\vu)} \vp[i] \leq \vc[i]~\text{for each}~i \in \{1, \dotsc, d\}.
  \end{align*}
  As the constraints separate over the coordinates, the minimization can
  be done for each coordinate separately. As the function is increasing
  in \(\vc[i]\), the minimum is achieved by making all the constraints
  tight, which giver the formula for \(\vb^+\) in the statement of the
  lemma.

  {\bf Volume decrease.} Let us prove the second part of the statement.
  Thus, assume for the remainder of the proof that \((\nicefrac{1}{2d})
  \cdot \vb \in \setH_{>}(\vu)\). We first show that $\Vol(\setB(\vb^+))
  \leq (\nicefrac{1}{(d + 1)})\Vol(\setB(\vb^+))$ if we assume that the
  update from \(\setB(\vb)\) to \(\setB(\vb^+)\) shrinks the box in only
  one coordinate, i.e., 
  \begin{equation}
    \label{eq:box_worst_case}
    \setI \coloneqq \setst{i \in [d]}{\vb[i] > \nicefrac{1}{\vu[i]}} = 
    \setst{i \in [d]}{\vb^+[i] \neq \vb[i]} ~\text{has exactly one element}.
  \end{equation}
  Assume the above holds and \(\setI = \{j\}\). 
  Then, as \((\nicefrac{1}{2d}) \cdot \vb \in \setH_{>}(\vu)\)
  implies \(\lin{\vu, (\nicefrac{1}{2d})\vb} > 1\), 
  \begin{equation*}
    1 < \lin{\vu, (\nicefrac{1}{2d})\vb} \leq  \frac{1}{2d}(\vu[j]\vb[j] + d - 1)
    \implies (d + 1) \frac{1}{\vu[j]} \leq \vb[j].
  \end{equation*}
  This together with the fact that \(\vb^+[i] = \vb[i]\) for all \(i
  \neq j\) and \(\vb^+[j] = \nicefrac{1}{\vu[j]}\) yields
  \begin{equation*}
    \Vol(\setB(\vb^+))
    = \prod_{i = 1}^d \vb^+[i]
    =  \frac{1}{\vu[j]} \cdot \prod_{i \neq j} \vb[i]
    \leq \frac{1}{d+1} \prod_{i = 1}^d \vb[i]
    = \frac{1}{d+1} \Vol(\setB(\vb)).
  \end{equation*}
  To complete the proof, we only need to show we may
  assume~\eqref{eq:box_worst_case} holds. 
  Assume the opposite, that is, 
  that there are two distinct coordinates that shrink from \(\vb^+\) to \(\vb\). 
  We will show that the volume shrinks more, meaning 
  the above bound also applies.
  Formally, assume there are \(j,k \in \setI\) that are distinct. 
  For this part, it will be useful to denote by \(\vb^+(\vu)\)
  the point defined in \cref{eq:mvb_formula_apx} for a given vector \(\vu\). 
  We will show we can construct \(\smash{\vu' \in \Rnn^d}\) 
  such that \(\Vol(\vb^+(\vu)) \leq \Vol(\vb^+(\vu'))\) while
  maintaining the property \((\nicefrac{1}{2d}) \vb \in \setH_{>}(\vu')\) and such that \(\vb^+(\vu')[i] \neq \vb[i]\) for all \(i \in \setI \setminus \{j\}\), 
  which makes~\eqref{eq:box_worst_case} follow by induction. 
  Indeed, define \(\vu' \in \Rnn^d\) by
  \begin{equation}
    \label{eq:def_uprime}
    \vu'[i] \coloneqq \vu[i]~\text{for}~i \not\in \{j,k\}, \quad \vu'[j] \coloneqq \frac{1}{\vb[j]}, \quad \text{and}~\vu'[k] \coloneqq \vu[k] + \frac{\vb[j]}{\vb[k]}\paren{\vu[j] - \frac{1}{\vb[j]}}.
  \end{equation}
  First, note that \((\nicefrac{1}{2d})\vb \in \setH_{>}(\vu')\) since 
  \begin{align*}
    \lin{\vu' - \vu, \vb}
    &=  \vb[j](\vu'[j] - \vu[j]) + (\vu'[k] - \vu[k]) \vb[k]
    \\
    &=  \vb[j]\Biggl(\frac{1}{\vb[j]} - \vu[j]\Biggr) + \bigg(\frac{\vb[j]}{\vb[k]}\Big( \vu[j] - \frac{1}{\vb[j]}\Big)\bigg) \vb[k] = 0
  \end{align*}
  and, thus, \(1 < \lin{\vu, (\nicefrac{1}{2d})\vb} = \lin{\vu', (\nicefrac{1}{2d}) \vb}\). Let us now show that \(\Vol(\setB(\vb^+(\vu))) \leq \Vol(\setB(\vb^+(\vu')))\). Since \(\vb^+(\vu)[i] = \vb^+(\vu')[i] \) for \(i \not\in \{j, k\}\), we have
  \begin{align*}
    \frac{\Vol(\setB(\vb^+(\vu)))}{\Vol(\setB(\vb^+(\vu')))}
    &= \frac{\vb^+(\vu)[j]}{\vb^+(\vu')[j]} \cdot \frac{\vb^+(\vu)[k]}{\vb^+(\vu')[k]}
    \\
    & = \frac{\min(\vb[j], \nicefrac{1}{\vu[j]})}{\min(\vb[j], \nicefrac{1}{\vu'[j]})} \cdot
    \frac{\min(\vb[k], \nicefrac{1}{\vu[k]})}{\min(\vb[k], \nicefrac{1}{\vu'[k]})}
    \\
    &= \frac{\nicefrac{1}{\vu[j]}}{ \vb[j]} \cdot
    \frac{\nicefrac{1}{\vu[k]}}{\nicefrac{1}{\vu'[k]}}
    &\text{(since \(j,k \in \setI\) and by~\eqref{eq:def_uprime})}
    \\
    &= \frac{1}{\vb[j] \vu[j]} \cdot \frac{1}{\vu[k]}
    \biggl(\vu[k] + \frac{\vb[j]}{\vb[k]}\paren{\vu[j] - \frac{1}{\vb[j]}}\biggr)
    \\
    &=
    \frac{1}{\vb[j] \vu[j]} \cdot \frac{1}{\vu[k] \vb[k]}(\vb[k] \vu[k] + \vb[j] \vu[j] - 1).
  \end{align*}
  To get that 
  \(\Vol(\vb^+(\vu)) \leq \Vol(\vb^+(\vu'))\), 
  we can show that last line is bounded by $<1$.
  Using the substitution \(\alpha \coloneqq \vb[j] \vu[j] \) and \(\beta \coloneqq \vb[k] \vu[k]\),
  we want to show that 
  \begin{equation*}
    \frac{\alpha + \beta -1}{\alpha \beta} < 1
     \iff \alpha \beta - \alpha - \beta + 1 > 0
     \iff (\alpha -1)(\beta -1) > 0.
  \end{equation*}
  This holds if \(\alpha > 1\) and \(\beta > 1\), 
  is implied by \(j,k \in \setI\) since \(\alpha = \vb[j] \vu[j] > 1\) and $\beta = \vb[k] \vu[k] > 1$.
  A simple induction shows we may assume~\eqref{eq:box_worst_case} holds. To see that \((\alpha + \beta - 1)/\alpha \beta < 1\), note that
\end{proof}

Equipped with the above lemma, we are in position to prove \cref{thm:box}.

\begin{proof}[Proof of~\cref{thm:box}]
  Property (a), 
  holds by induction because, for any \(\vu_t\) used in a
  call to \(\CUT\), we have \(\mP^* \in \setH_{\leq}(\vu_t)\) since
  \(\mP^*\) is valid and since by \cref{lem:starting-hyperplane-main} the half-space
  \(\setH_{\leq}(\vu_t)\) contains only diagonals of invalid preconditioners. 
  For (b), fix \(t \in \{1, \dotsc, T\}\) and recall
  that in this case we have \(\setS_t = \setst{\Diag(\vp)}{ \vp
  \in\setB(\vb_t)}\) and \(\mPt = (\nicefrac{1}{2d}) \cdot
  \Diag(\vb_t)\). The competitive ratio of \(\nicefrac{1}{2d}\)  follows since \(\Diag(\vb_t)\) is the
  preconditioner that maximizes \(\norm{\nabla f(\vx_t)}_{\mP}\) for
  \(\mP \in \setB(\vb_t)\). Finally, for (c) by
  \cref{lemma:box_vol_shrink_apx} we have that every call to \(\CUT\)
  makes the volume of the set decrease by \(1/c \coloneqq 1/(d+1)\).
  Moreover, one can easily verify that \(\vb_t[i] \geq
  \min\{\nicefrac{1}{L}, \vb_0[i]\}\) for all \(i \in \{1, \dotsc, d\}\)
  since \(\setB((\nicefrac{1}{L})\ones)\) contains only diagonals of
  valid preconditioners. Therefore, for \(\vb_{\min}[i] \coloneqq
  \min\{\nicefrac{1}{L}, \vb_0[i]\}\), the volume of \(\setB(\vb_t)\)
  cannot be smaller than \(\setB(\vb_{\min})\) for all iteration \(t\).
  Therefore, the number of times \(\CUT\) is invoked is no more than
  \begin{equation*}
    \log_c\paren{\frac{\Vol(\setB(\vb_0))}{\Vol(\setB(\vb_{\min}))}}
     = \log_c\paren{\prod_{i = 1}^d \frac{\vb_0[i]}{\vb_{\min}[i]}}
     \leq \log_c((\norm{\vb_0}_{\infty} L)^d)
     = d \log_c(\norm{\vb_0}_{\infty} L).
  \end{equation*}
  as desired.
\end{proof}

\subsection{Axis-aligned ellipsoids}
\label{apx:subsec-set-ellipsoid}

We now analyze the cutting-plane method using axis-aligned ellipsoids.
Interestingly, the results that we prove in this sections are connected
to some of the results from~\citet{GoemansHIM09a} via polarity theory.
We defer a discussion on this connection to the end of this section.

Different from the main body, it will be helpful for the analysis of the
method and proofs of the results to not restrict ellipsoids to the
non-negative orthant, as was done in the main text for ease of
exposition. For any symmetric positive definite matrix \(\mA \in \R^{d
\times d}\), define the \emph{ellipsoid} given by \(\mA\) by
\begin{equation*}
  \setE(\mA) \coloneqq \{ \vx \in \R^d \colon \lin{\vx, \mA \vx} \leq 1 \}.
\end{equation*}
When \(\mA\) is diagonal, we say that \(\setE(\mA)\) is
\emph{axis-aligned}. Moreover, we may slightly overload our notation by defining \(\setE(\va) \coloneqq \setE(\Diag(\va))\).

\begin{figure}[t]
  \centering
  \label{fig:ellipsoid_1}
  \includegraphics[width=0.45\textwidth]{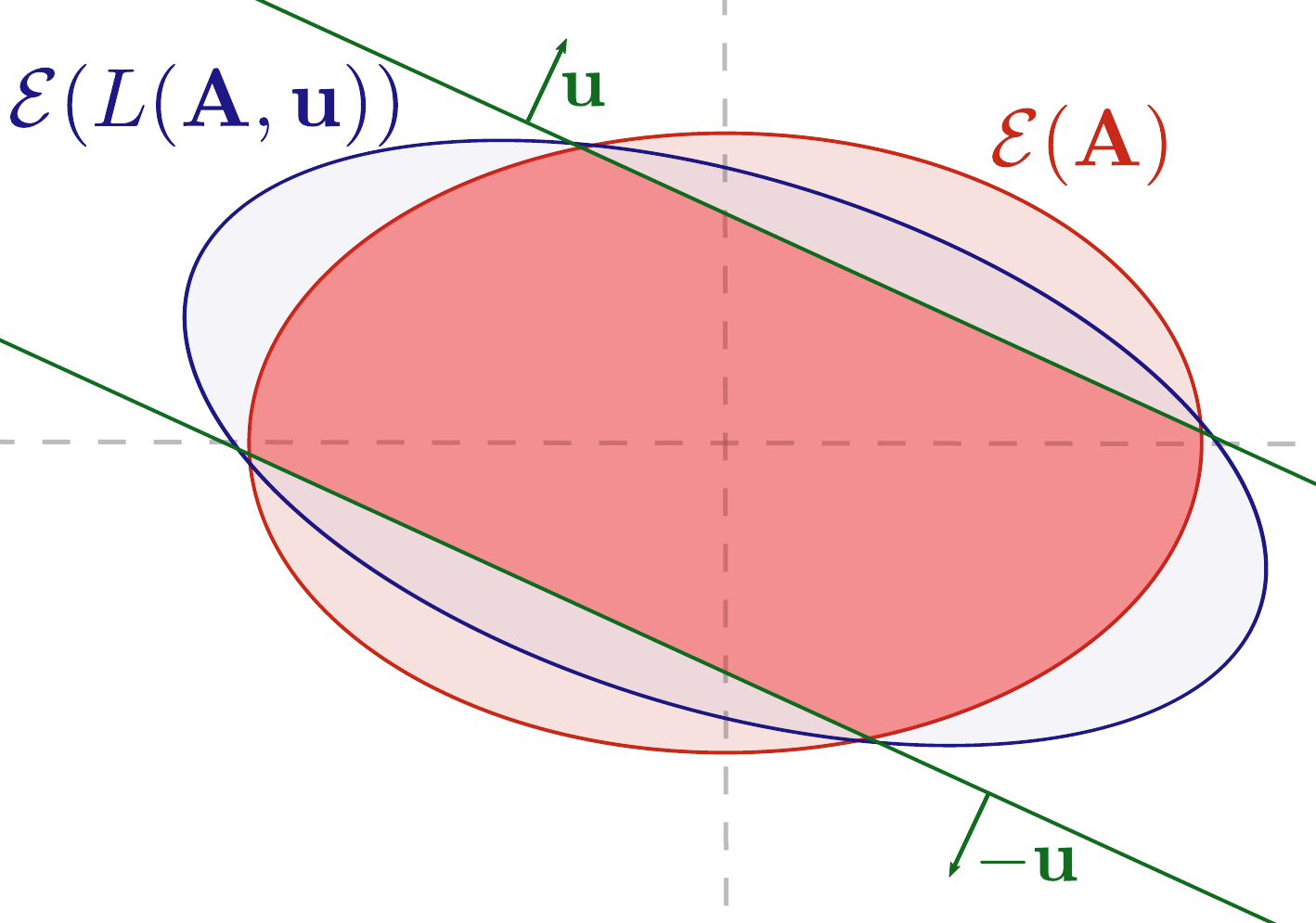}\hfill%
  \includegraphics[width=0.45\textwidth]{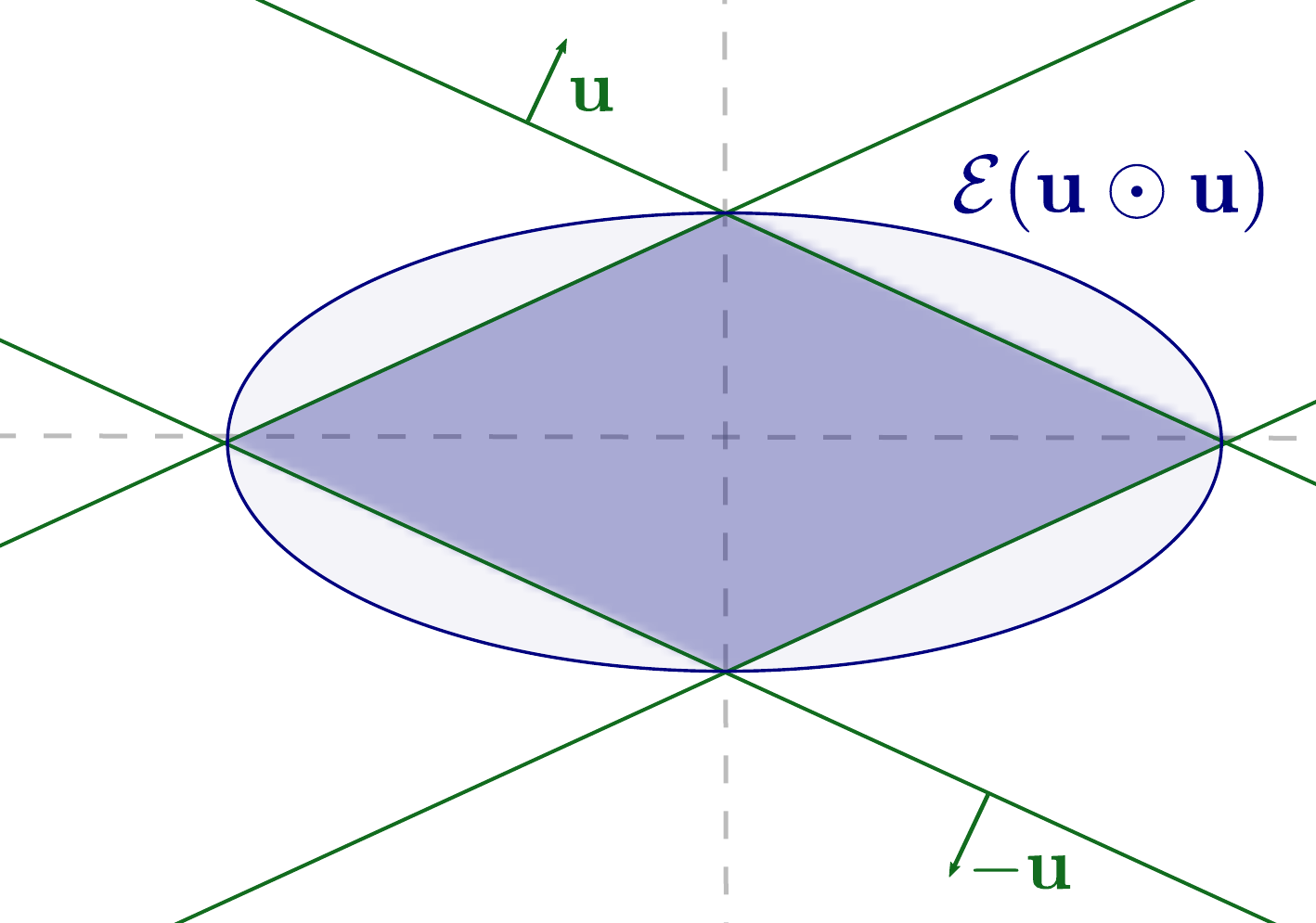}
  \caption{Illustration of the ellipsoids in \cref{thm:non-sym-ellipsoid} in the left. In the right an illustration of the symmetrized intersection of halfspaces \(\setS(\vu)\) used in the proof of~\cref{lem:lemlem} together with the ellipsoid \(\setE(\vu \odot \vu)\) used in the convex combination in the lemma.}
  \end{figure}

{\bf General ellipsoids.} 
Although we are ultimately interested in working solely with ellipsoids defined by diagonal matrices, 
we will start by looking at more general ellipsoids, 
and then exploit symmetry in our case to derive the result in~\cref{lem:lemlem}.
We start with an ellipsoid \(\setE(\mA)\) where \(\mA\) is a positive definite matrix. 
Then, given a vector \(\smash{\vu \in \R^d}\), 
we are interested in finding an ellipsoid the intersection of \(\setE(\mA)\) 
with the half-spaces defined by \(\vu\) and \(-\vu\) that contain the origin, that is, the set
\begin{equation*}
  \setE(\mA) \cap \setst{\vx \in \R^d}{\lin{\vx, \vu} < 1} \cap \setst{\vx \in \R^d}{- \lin{\vx, \vu} < 1} 
  =\setE(\mA) \cap \setst{\vx \in \R^d}{|\lin{\vx, \vu}| < 1}.
\end{equation*}
The following theorem shows how to find an ellipsoid that contains the above intersection, 
and how to guarantee its volume is smaller than \(\setE(\mA)\) if \(\vu\) is large enough. 
Interestingly, note that
\begin{equation*}
  \setst{\vx \in \R^d}{|\lin{\vx, \vu}| < 1} = \setst{\vx \in \R^d}{(\lin{\vx, \vu})^2 < 1} =  \setE(\vu \vu^{\transp}).
\end{equation*}
The set \(\smash{\setE(\vu \vu^{\transp})}\) is a degenerate ellipsoid,
in the sense that it is not a compact set, 
and any $\vp$ orthogonal to $\vu$ is contained in $\smash{\setE(\vu\vu^\top)}$.
Still, the next theorem shows how to find a convex combination
of \(\setE(\mA)\) and \(\setE(\vu \vu^{\transp})\)---which always
contains \(\setE(\mA) \cap \setE(\vu \vu^{\transp})\)---that is
guaranteed to have volume smaller than \(\setE(\mA)\) if \(\vu\) is large enough. 
The following result can be seen
as the polar result of \citet[Lemma~2]{GoemansHIM09a}.
\begin{theorem}
  \label{thm:non-sym-ellipsoid}
  Let \(\mA \in \R^{d \times d}\) be positive definite and let \(\vu \in
  \R^d\). 
  Let \(\lambda \in (0,1)\) and define
  \begin{equation*}
    L(\mA, \vu) \coloneqq \lambda \mA + (1 - \lambda) \vu \vu^\transp.
  \end{equation*}
  Then \(\setE(\mA)\cap \setE(\vu \vu^{\transp})  \subseteq \setE(L(\mA, \vu)) \) and 
  \begin{equation*}
    \Vol(\setE(L(\mA, \vu))) = \sqrt{\frac{\lambda}{\lambda + (1 - \lambda)\cdot \ell} \cdot \frac{1}{\lambda^d}} \cdot \Vol(\setE(\mA))
  \end{equation*}
  In particular, if \(\ell \coloneqq \norm{\vu}_{\mA^{-1}}^2 > d\) and
  \begin{equation}
    \label{eq:opt_lambda_ellipsoid}
    \lambda = \frac{\ell}{d} \cdot \frac{d-1}{\ell-1},
  \end{equation}
  then \(\lambda \in (0,1)\) and 
  \(\Vol(\setE(L(\mA, \vu))) = \nu_d(\vu) \Vol(\setE(\mA))\) where 
  \begin{equation}
    \label{eq:def_nu_d}
    \nu_d(\vu) = \sqrt{\frac{1}{\lambda^d} \cdot \frac{d-1}{\ell-1}} = \parens{\frac{d}{\ell}}^{d/2} \parens{\frac{\ell-1}{d-1}}^{(d-1)/2} \in (0,1).
  \end{equation}
\end{theorem}
\begin{proof}
  First, note that for any \(\vp \in \setE(\mA)\cap \setE(\vu \vu^{\transp})\) and any \(\lambda \in (0,1)\) we have 
  \begin{equation*}
    \lin{\vp, L(\mA, \vu) \vp}
    =  \lambda \lin{\vp, \mA \vp}
    + (1 - \lambda) \lin{\vp, \vu} \leq \lambda + (1 - \lambda) = 1.
  \end{equation*}
  Thus, \(\setE(L(\mA, \vu)) \subseteq \setE(\mA) \cap \setE(\vu \vu^{\transp})\). For
  the volume decrease, recall that for ellipsoids \(\setE(\mA)\) we have
  \(\smash{\Vol(\setE(\mA)) = \nicefrac{V_d}{\sqrt{\det(\mA)}}}\) where \(V_d\) is the
  volume of the unit sphere in \(\R^d\). By the matrix-determinant
  lemma, we have
  \begin{equation*}
    \det(L(\mA, \vu)) = \parens{1 + \frac{1 - \lambda}{\lambda}\cdot \lin{\vu, \mA^{-1} \vu}} \det(\lambda \mA) = \parens{1 + \frac{1 - \lambda}{\lambda}\cdot \ell}\lambda^d \det(\mA).
  \end{equation*}
  Therefore, 
  \begin{equation*}
    \Vol(\setE(L(\mA, \vu))) = \sqrt{\frac{1}{\parens{1 + \frac{1 - \lambda}{\lambda}\cdot \ell}} \cdot \frac{1}{\lambda^d}} \cdot \Vol(\setE(\mA))
    = \sqrt{\frac{\lambda}{\lambda + (1 - \lambda)\cdot \ell}} \cdot \frac{1}{\lambda^d} \cdot \Vol(\setE(\mA)).
  \end{equation*}
  Finally, for \(\lambda\) defined as in~\eqref{eq:opt_lambda_ellipsoid}
  we have
  \begin{align*}
    1 + \frac{1 - \lambda}{\lambda}\cdot \ell
    &= 1 + \parens{1 - \frac{\ell (d-1)}{d (\ell -1)}} \frac{d (\ell -1)}{\ell (d-1)} \cdot \ell
    &&=  1 + \parens{\frac{d (\ell -1)}{\ell (d-1)} - 1} \cdot \ell,
    \\
    &=  1 + \parens{\frac{d (\ell -1) - \ell(d-1)}{\ell (d-1)}} \cdot \ell
    &&= 1 + \frac{\ell - d}{d - 1} = \frac{\ell -1}{d -1},
  \end{align*}
  which yields the desired formula for \(\nu_d(\vu)\).
\end{proof}

{\bf On the norm of \(\vu\).} The above theorem has a requirement
on the norm of the vector \(\vu\) that defines the half-space
\(\setH_{\leq}(\vu)\). However, in our cutting plane method we obtain
\(\vu\) from \cref{lem:starting-hyperplane-main} and
\cref{lem:stronger-hp}, which do not have any guarantees on the norm of
\(\vu\) explicitly. Crucially, at any given iteration \(t\) of
multidimensional backtracking with ellipsoids, we select a candidate
preconditioner \(\mPt = \Diag(\vp_t)\) such that \(\norm{\vp_t}_{\mA} =
\nicefrac{1}{\sqrt{2d}}\). Then, if it fails the Armijo condition
in~\eqref{eq:sufficient-progress} and \(\vu_t\) is as given by
\cref{lem:starting-hyperplane-main}, then we have \(\vp_t \in
\setH_{>}(\vu_t)\), that is, the separating hyperplane excludes~\(\vp_t\). As we will show, this implies that \(\norm{\vu}_{\mA^{-1}}\)
is large.

\begin{lemma}
  \label{lem:bound_norm_u}
  Let \(\mA \in \R^{d \times d}\) be positive definite and \(\vp \in \Rnn^d\) be such that \(\norm{\vp}_{\mA} \leq \gamma\) for some \(\gamma > 0\). Let \(\vu \in \Rnn^d\) be such that \(\vp \in \setH_{>}(\vu)\). Then \(\norm{\vu}_{\mA^{-1}} > \nicefrac{1}{\gamma}\).
\end{lemma}
\begin{proof}
  For the sake of contradiction,  assume \(\norm{\vu}_{\mA^{-1}} 
  \leq \nicefrac{1}{\gamma}\). Then \(\norm{\vu}_{\mA^{-1}} \cdot \norm{\vp}_{\mA} \leq 1\). Thus, by the Cauchy-Schwartz inequality,
  \begin{equation*}
    \lin{\vu, \vp} = \lin{\mA^{-1/2} \vu, \mA^{1/2} \vp}
    \leq \norm{\mA^{-1/2} \vu} \cdot \norm{\mA^{1/2} \vp} = \norm{\vu}_{\mA^{-1}} \cdot \norm{\vp}_{\mA} \leq 1. 
  \end{equation*}
  This is a contradiction since \(\vp \in \setH_{>}(\vu)\) and, therefore, \(\lin{\vu, \vp} > 1\). 
\end{proof}

{\bf On the volume decrease.}
Although the formula \(\nu_d(\vu)\) in \cref{eq:def_nu_d} can
be hard to interpret, we show a simple bound when
\(\norm{\vu}_{\mA^{-1}}^2 \geq 2d\).

\begin{lemma}
  \label{lem:bound_nu_d}
  Let \(\mA \in \R^{d \times d}\) be a positive definite matrix and \(\vu \in \R^d\) be such that \(\norm{\vu}_{A^{-1}}^2 > d\). For \(c \coloneqq d/\ell \in (0,1)\) we have \(\nu_d(\vu) \leq \sqrt{c \cdot e^{1 - c}}\), where \(\nu_d\) is defined as in~\eqref{eq:def_nu_d}. In particular, if $\norm{\vu}_{A^{-1}}^2 > d$, then \(\nu_d(\vu) \leq \nicefrac{\sqrt[4]{e}}{\sqrt{2}}\).
\end{lemma}
\begin{proof}
  Define \(\ell \coloneqq \norm{\vu}_{\mA^{-1}}^2 > d\) and \(c \coloneqq \nicefrac{d}{\ell} \in (0,1).\) Then,
  \begin{equation*}
    \begin{alignedat}{4}
      \nu_d(\vu)^2 
      &= \parens{\frac{d}{\ell}}^{d} \parens{\frac{\ell-1}{d-1}}^{(d-1)}
      &&= \frac{d}{\ell} \cdot \paren{\frac{d}{\ell} \cdot \frac{\ell-1}{d -1}}^{(d-1)}
      &&= c \cdot \paren{c \cdot \frac{\nicefrac{d}{c}-1}{d -1}}^{(d-1)}
      \\
      &= c \cdot \paren{\frac{ d - c}{d -1}}^{(d-1)}
      &&= c \cdot \paren{1 + \frac{ 1 - c}{d -1}}^{(d-1)}
      &&\leq c \cdot e^{1 - c},
    \end{alignedat} 
  \end{equation*}
  where the last inequality follows since \(1 + x \leq e^x\) for all \(x \in \R\). 
  In particular, note that \(c \in (0,1) \mapsto c  \cdot e^{1 - c}\) is increasing 
  since the derivative of the mapping is positive on \((0,1)\). 
  Thus, if \(\norm{\vu}_{\mA^{-1}} \geq 2d\), 
  then $c \leq \frac{1}{2}$ 
  and \( \smash{ c \cdot e^{1-c} \leq (\nicefrac{1}{2}) \cdot e^{1/2}}\).
\end{proof}

\paragraph{Exploiting symmetry.} 
Let us now exploit symmetry to avoid using non-diagonal matrices in our
ellipsoids. We use the notion of \emph{axis-aligned} sets in the next
few results. 
A set \(\setX \subseteq \R^d\) is \emph{axis-aligned} if
for any point $\vp \in \setX$, the reflections of $\vp$ along the axes are also contained in $\setX$. Formally, 
for any \(\vs \in \{\pm 1\}^d\),  we have that if $\vp \in \setX$, then $\Diag(\vs) \vp \in \setX$.
Furthermore, with a slight abuse of
notation define \(\Diag(\mA) \coloneqq \Diag(\diag(\mA))\). That is,
\(\Diag(\mA)\) is the diagonal matrix whose diagonal entries match those
of \(\mA\). The idea is that the set \(\setst{\vp \in
\Rnn^d}{\Diag(\vp)~\text{is valid}}\) of diagonals of valid
preconditioners is contained in the non-negative orthant. 
Yet, we can extend it by reflecting it over each of the axes. 
Although this may seem counter-intuitive, 
this translates the structure of our problem into symmetry among all orthant, 
and this can be exploited elegantly.
Formally, 
the set of diagonals of valid preconditioners reflected over each axis is given by
set
\begin{equation*}
  \setP \coloneqq \setst{\vp \in \R^d}{ \Diag(\abs{\vp})~\text{is valid}},
\end{equation*}
where \(\abs{\vp}\) is the entry-wise absolute value of \(\vp \in
\R^d\). The following lemma shows that when looking for low volume
ellipsoids that contain an axis-aligned set, we can restrict out
attention to axis-aligned ellipsoids, defined by a diagonal matrix. The
following lemma can be seen as the polar statement of
\citet[Proposition~3.1]{GoemansHIM09a}, with the benefit of not
requriring any matrix inversions.
\begin{lemma}
  \label{lem:symm_ellipsoids}
  Let \(\setX \subset \R^d\) be an axis-aligned convex set and let \(\mA
  \in \R^{d \times d}\) be positive definite matrix such that \(\setX \subseteq \setE(\mA)\). Then \(\setX \subseteq \setE(\Diag(\mA))\) and \(\Vol(\setE(\Diag(\mA))) \leq \Vol(\setE(\mA))\).
\end{lemma}
\begin{proof}
Let us start by showing that  \(\setX \subseteq \setE(\Diag(\mA))\).
We use the notation \(\Diag(\vv) \cdot \setX\)
to denote the set \(\Diag(\vv) \cdot \setX \coloneqq \setst{\Diag(\vv) \cdot \vx}{\vx \in \setX} \). 
Since \(\setX\) is axis-aligned, we have 
\begin{equation*}
  \setX = \Diag(\vs) \cdot \setX \subseteq \Diag(\vs) \cdot \setE(\mA)
  = \setE(\Diag(\vs) \mA \Diag(\vs)), \qquad \forall \vs \in \{\pm 1\}^d.
\end{equation*}
Therefore, \(\setX\) is contained in each of the \(2^d\) ellipsoids of the form \(\setE(\Diag(\vs) \mA \Diag(\vs))\). Thus,
\begin{equation*}
  \setX \subseteq \bigcap_{\vs \in \{\pm 1\}^d} \setE(\Diag(\vs) \mA \Diag(\vs)) \subseteq \setE\bigg(\frac{1}{ 2^d} 
  \sum_{\vs \in \{\pm 1\}^d} \Diag(\vs) \mA \Diag(\vs) \bigg),
\end{equation*}
where the last inclusion follows since, for any set of positive definite matrices \(\setM\), one may verify that \(\cap_{\mM \in \setM} \setE(\mM) \subseteq \setE((\nicefrac{1}{|\setM|}) \sum_{\mM \in \setM}\mM)\).
 Finally, note that
\begin{equation*}
  \sum_{\vs \in \{\pm 1\}^d} \Diag(\vs) \mA \Diag(\vs) = \Diag(\mA).
\end{equation*}
Indeed, let \(i,j \in \{1, \cdots, d\}\). If \(i = j\), then
\((\Diag(\vs) \mA \Diag(\vs))_{i,j} = \mA_{i,j}\) for any \(\vs \in \{\pm 1\}^d\). If \(i \neq j\), then
\begin{align*}
  &\sum_{\vs \in \{\pm 1\}^d} (\Diag(\vs) \mA \Diag(\vs))_{i,j}
  \\
  &= \sum_{\vs \in \{\pm 1\}^d \colon \vs[i] \neq \vs[j]} (\Diag(\vs) \mA \Diag(\vs))_{i,j} + 
  \sum_{\vs \in \{\pm 1\}^d \colon \vs[i] = \vs[j]} (\Diag(\vs) \mA \Diag(\vs))_{i,j}
  \\
  &= 2^{d-1} \cdot (-\mA_{i,j}) + 2^{d-1} \cdot \mA_{i,j} = 0.
\end{align*}
Let us now show that \(\Vol(\setE(\Diag(\mA))) \leq \Vol(\setE(\mA))\). Note that
\(\log(\Vol(\setE(\mA))) = \log(\Vol(\setE(\mI))) - \tfrac{1}{2}\log
\det(\mA)\). Since \(\log \det(\cdot)\) is concave over positive definite
matrices, we have
\begin{align*}
  \log \det (\Diag(\mA)) 
  &=  \log \det\Bigl(\frac{1}{ 2^d} 
  \sum_{\vs \in \{\pm 1\}^d} \Diag(\vs) \mA \Diag(\vs)  \Bigr)
  \\
  &\geq 
  \frac{1}{ 2^d} \sum_{\vs \in \{\pm 1\}^d} 
  \log \det\Bigl( 
   \Diag(\vs) \mA \Diag(\vs)  \Bigr)
  = 
  \frac{1}{ 2^d} \cdot 2^d
  \log \det(\mA)
  = \log \det(\mA).
\end{align*}
Therefore,
\begin{align*}
  \log(\Vol(\setE(\Diag(\mA))))
  &= \log(\Vol(\setE(\mI))) - \tfrac{1}{2}\log
  \det(\Diag(\mA))
  \\
  &\leq \log(\Vol(\setE(\mI))) - \tfrac{1}{2}\log
  \det(\mA)
  \\
  &=\log(\Vol(\setE(\mA))),
\end{align*}
which implies that \(\Vol(\setE(\Diag(\mA))) \leq \Vol(\setE(\mA))\).
\end{proof}

We are now in position to prove \cref{lem:lemlem}, which follows directly from the previous two results.

\begin{proof}[Proof of~\cref{lem:lemlem}]
  By the assumptions in \cref{lem:starting-hyperplane-main} we have that
  \(\mP \coloneqq \Diag(\vp)\) fails the Armijo condition \ref{eq:sufficient-progress} condition
  and, thus, \(\smash{\vp \in \setH_{>}(\vu)}\). 
  This together with the assumption that 
  \(\smash{\norm{\vp}_{\mA} \leq \nicefrac{1}{\sqrt{2d}}}\)
  imply via \cref{lem:bound_norm_u} that 
  \(\smash{ \norm{\vu}_{\mA^{-1}} \geq \sqrt{2d}} \). 
  This allows us to use \cref{thm:non-sym-ellipsoid} 
  to find a new ellipsoid containing \(\setE(\va) \cap \setH_{\leq}(\vu)\) 
  with the required volume decrease
  by \cref{lem:bound_nu_d}. 
  Yet, this ellipsoid may not be axis-aligned. 
  We shall exploit the symmetry described in \cref{lem:symm_ellipsoids}
  to show that the axis-aligned ellipsoid \(\setE(\va^+(\va, \vu))\) enjoys the same guarantees.

  Formally, we need \(\setE(\va^+(\va, \vu))\) to contain
  \(\smash{\setE(\va)\cap \setH_{\leq}(\vu)}\). 
  Since \(\vu \geq 0\), we have
  \begin{equation*}
    \setH_{\leq}(\vu) \subseteq \setS(\vu) \coloneqq \setst{\vp \in \R}{\Diag(\vs) \cdot \vp \in \setH_{\leq}(\vu)~\text{for all}~\vs \in \{\pm 1\}^d}.
  \end{equation*}
  Thus, it suffices for \(\setE(\va^+(\va, \vu))\) to
  contain \(\setE(\va) \cap \setS(\vu)\).
  From \cref{thm:non-sym-ellipsoid} we know that \(\setE(\va) \cap \setS(\vu)\) is contained in the ellipsoid given by the matrix \(\lambda \Diag(\va) +  (1 - \lambda) \vu \vu^\transp\) for any \(\lambda\), in particular for \(\lambda\) as in~\eqref{eq:opt_lambda_ellipsoid} since \(\smash{\norm{\vu}_{\mA^{-1}} > \sqrt{d}}\). 
  Since \(\setS(\vu)\) is axis-aligned, we can exploit symmetry using
  \cref{lem:symm_ellipsoids}, which tells that
  \(\setE(\va) \cap \setS(\vu)\) is contained in the ellipsoid given by the matrix
  \begin{equation*}
    \Diag\Big(\lambda \Diag(\va) +  (1 - \lambda) \vu \vu^\transp\Big) = \Diag(\va^+(\va, \vu)),
  \end{equation*}
  as desired. Finally, the bound on the volume follows by \cref{thm:non-sym-ellipsoid} and the bound on \(\nu_d(\vu)\) given by \cref{lem:bound_nu_d} since \(\smash{\norm{\vu}_{\mA^{-1}} \geq \sqrt{2d}}\).
\end{proof}

Finally, we are in position to prove \cref{thm:main_thm_ellipsoid}, 
which follows almost directly from~\cref{lem:lemlem}.

\begin{proof}[Proof of \cref{thm:main_thm_ellipsoid}]
  Note that (a) holds by induction and since, by
  \cref{lem:starting-hyperplane-main}, we have \(\diag(\mP^*) \in
  \setH_{\leq}(\vu_t)\) for any \(\vu_t\) used in a call to \(\CUT\).
  For (b), fix \(t \in \{1, \dotsc, T\}\) and recall that in this case
  we have \(\setS_t = \setst{\Diag(\vp)}{ \vp \in \setE(\va_t)}\). As
  described in~\eqref{eq:max_over_ellipsoid}, one may verify that
  \(\Diag(\vq_t^*)\) for \(\vq_t^*\) given by
  \begin{equation*}
    \vq_t^* \coloneqq \frac{1}{\norm{\nabla f(\vx_t)^2}_{\mA_t^{-1}}} \cdot \mA_t^{-1} \nabla f(\vx_t)^2
  \end{equation*}
  maximizes \(\norm{\nabla f(\vx_t)}_{\mP}\) for \(\mP \in \setS_t\). Since
  \begin{equation*}
    \mPt = \CANDIDATE(\setS_t, \nicefrac{1}{\sqrt{2d}}, \vx_t) = \frac{1}{\sqrt{2d}} \Diag(\vq_t^*),
  \end{equation*}
  we conclude that \(\mPt\) is \(\nicefrac{1}{\sqrt{2d}}\)-competitive.
  For (c), first note that we may assume \((\nicefrac{1}{L})\ones \in
  \setE(\alpha_0 I)\). 
  To see that, assume \((\nicefrac{1}{L}) \ones \not \in \setE(\alpha_0 I)\), 
  implying \(\alpha_0 d > L^2\). 
  In this case, any candidate preconditioner computed by
  \(\CANDIDATE\) is always valid and, thus, we never call $\CUT$.
  To see this, let \(\mA_0 \coloneqq \alpha_0 \mI\) be the matrix defining the initial ellipsoid. Then, by the definition of
  \(\CANDIDATE\) for ellipsoids we have that \(\mP_0 = \Diag(\vp_0)\) is
  such that
  \begin{align*}
    \norm{\vp_0}_{\mA_0}^2 = \alpha_0\norm{\vp_0}^2 = \frac{1}{2d}
    < \frac{1}{2} \frac{\alpha_0}{L^2}.
  \end{align*}
  Therefore, \(\vp_0[i] \leq \nicefrac{1}{L}\) for all \(i \in \{1,
  \dotsc, d\}\), which implies that \(\mP_0\) is valid since \(\mP_0 \preceq \tfrac{1}{L} \mI\).

  Let us look now at the case \((\nicefrac{1}{L})\ones \in
  \setE(\alpha_0 I)\). Therefore, \(\setB(\nicefrac{1}{L} \ones)
  \subseteq \setE(\va_t)\) for all iterations \(t\). Since the minimum
  volume ellipsoid containing the box \(\setB((\nicefrac{1}{L}) \ones)\)
  is the unit sphere of radius \(\nicefrac{1}{L}\), that is,
  \(\setE((\nicefrac{L^2}{d}) \ones)\). Therefore, \(\Vol(\setE(\va_t))
  \geq \Vol(\setE((\nicefrac{L^2}{d}) \ones))\). Moreover, every time we call
  cut the volume of the ellipsoid goes down by \(1/c \coloneqq
  \sqrt[4]{e}/\sqrt{2}\). Therefore, the total number of calls to \(\CUT\) is no more than 
  \begin{equation*}
    \log_c\paren{\frac{\Vol(\setE(\alpha_0 \ones))}{\Vol(\setE((\nicefrac{L^2}{d})\ones ))}} 
    = \log_c\paren{\frac{L^d}{d \alpha_0^{d/2}}}
    \leq \frac{d}{\log(c)}
    \log\paren{\frac{L}{d \alpha_0}}
    \leq 12 d \log\paren{\frac{L}{\alpha_0}}
  \end{equation*}
  since \(\log(c) \geq 1/12\). 
\end{proof}

{\bf Refining the choice of \(\lambda\).} 
Although we have shown in~\cref{lem:lemlem} a choice a \(\lambda\) 
that guarantees volume decrease, it may be sub-optimal. 
The choice of \(\lambda\) in \cref{eq:opt_lambda_ellipsoid} 
is inherited from the non-symmetric case in \cref{thm:non-sym-ellipsoid}. 
Although \cref{lem:lemlem} and \cref{thm:non-sym-ellipsoid} match
when \(\vu\) has only one non-zero entry,
we should expect better choices of \(\lambda\),
leading to more volume shrinkage,
to be possible in~\cref{lem:lemlem}. 
Although we have not found a choice of \(\lambda\)
that is dependent on \(\vu\) that generically improves upon \eqref{eq:opt_lambda_ellipsoid}, 
in practice we can solve for a better \(\lambda\) numerically, 
by directly minimizing the volume of the resulting ellipsoid,
\aligns{
  \min_{0 < \lambda < 1}
  \Vol(\setE(\lambda \va + (1-\lambda) \Diag(\vu\vu^\top)))
  =
  \min_{0 < \lambda < 1}
  -\sum_i \log(\lambda \va[i] + (1-\lambda)\vu[i]^2).
}
As the problem is one-dimensional, 
numerical solvers can often find near-optimal solutions. 
By warm-starting a numerical solver with the \(\lambda\) 
defined in \eqref{eq:opt_lambda_ellipsoid}, 
we can guarantee that the resulting ellipsoid 
leads to a smaller volume
and we do not lose our worst-case theoretical guarantees.

\textbf{Connection to the polar problem and related work.} Our results
have an interesting connection to some of the results from
\citet{GoemansHIM09a}, via the use of polarity theory. Here we give a
quick overview of their work and the connection to our cutting plane
methods. \citet{GoemansHIM09a} shows techniques to approximate some polyhedron
\(\setP \subseteq \R^d\) (a polymatroid being one of the main examples)
\emph{from inside} by some ellipsoid \(\setE(\mA)\). Their algorithm
maintains an ellipsoid \(\setE(\mA) \subseteq \setP\) and tries to
iteratively enlarge it. They assume access to an oracle such that, at
each iteration, either finds a point \(\vu \in \mP\) that is
sufficiently far from \(\setE(\mA)\), meaning \(\smash{\norm{\vu}_{\mA}
> \sqrt{d} + \epsilon}\) for some \(\epsilon > 0\), or guarantees that
\(\setE(\mA)\) ``approximates well'' \(\setP\) from inside in the sense
that \(\smash{\norm{\vu}_{\mA} \leq (\sqrt{n} + \epsilon)/\alpha}\) for
all \(\vu \in \setP\), where \(\alpha > 0\) is some approximation
factor. In their algorithm, when the oracle finds a point \(\vu \in
\setP\) such that \(\smash{\norm{\vu}_{\mA} > \sqrt{d} + \epsilon}\) the
algorithm needs to find an ellipsoid \(\setE(\mA^+)\) such that 
\begin{equation}
  \label{eq:polar_problem}
  \setE(\mA^+) \subseteq \mathrm{conv}(\setE(\mA) \cup \{\vu, -\vu\}),
\end{equation}
where \(\mathrm{conv}(\setD)\) is the convex hull of \(\setD\).
Interestingly, the polar problem is exactly what we need for out cutting
plane method. More precisely, the polar set \(\setX^*\) of a set \(\setX\) is given by~\(\setX^* \coloneqq \setst{z \in \R^d}{\lin{z, x} 
\leq 1}\). Then, by taking polars and using that \(\setE(\mA)^* = \setE(\mA^{-1})\), we have that \(\setP^* \subseteq \setE(\mA^{-1})\). Moreover, taking polar on both sides of~\eqref{eq:polar_problem} yields that an equivalent problem is finding \((\mA^+)^{-1}\) such that
\begin{equation*}
  \setE((\mA^+)^{-1}) \supseteq \setE(\mA^{-1}) \cap \{-\vu, \vu\}^*
  = \setE(\mA^{-1}) \cap \setst{\vz}{ |\lin{\vu, \vz}| \leq 1}.
\end{equation*}
That is, the problem is the one of finding a smaller ellipsoid
\(\setE((\mA^+)^{-1})\) that contains \(\setE(\mA^{-1}) \cap
\setst{\vz}{ |\lin{\vu, \vz}| \leq 1}\), which is broadly the goal of the
subroutine \(\CUT\).

\clearpage
\section{Experiments}
\label{apx:experimental-details}
\label{apx:exp-additional}

\begin{table}[t]
\centering
\caption{
  {\bf Datasets used in our experiments,} including number of samples $n$ and dimension $d$, 
  and order of magnitude of the 
  condition number of the regularized system ($\kappa(\mX^\top \mX + \nicefrac{1}{n} \mI)$)
  and condition number of the system 
  when using the optimal diagonal preconditioner, $\kappa_*$.
}

\small
\vspace{1em}

\begin{tabular}{lrrrllllllllllll}
\toprule
{\bf Dataset} & Repository/Source & $n$ & $d$ & $\kappa$ & $\kappa_*$ \\
\midrule
{\bf cpusmall} 
& LIBSVM, Delve ({\tt comp-activ})
& $8\,192$ & $12$ & $10^{13}$ & $10^2$ 
\\
{\bf california-housing} 
& Scikit/StatLib, \citet{kelley1997sparse}
& $20\,640$ & $8$ & $10^{10}$ & $10^4$ 
\\
{\bf power-plant} 
& UCI, \citet{tufekci2014prediction}
& $9\,568$ & $4$ & $10^{9}$ & $10^4$
\\
{\bf concrete} 
& UCI, \citet{yeh1998concrete}
& $1\,030$ & $8$ & $10^{9}$ & $10^3$ 
\\
{\bf mg} 
& LIBSVM, \citet{flake2002efficient}
& $1\,385$ & $6$ & $10^{3}$ & $10^3$ 
\\
{\bf breast-cancer} 
& UCI
& $569$ & $32$ & $10^{13}$ & $10^2$ 
\\
{\bf australian} & 
LIBSVM, Statlog
& $690$ & $14$ & $10^{9}$ & $10^2$ 
\\
{\bf heart} 
& LIBSVM, Statlog
& $270$ & $13$ & $10^{7}$ & $10^2$ 
\\
{\bf diabetes} 
& UCI
& $768$ & $8$ & $10^{6}$ & $10^2$ 
\\
{\bf ionosphere} 
& UCI
& $351$ & $34$ & $10^{3}$ & $10^2$ 
\\
{\bf news20} 
& LIBSVM, \citet{keerthi2005modified}
& $19\,996$ & $1\,355\,191$ & $10^{13}$ & NA 
\\
{\bf rcv1} 
& LIBSVM, \citet{lewis2004rcv1}
& $20\,242$ & $47\,236$ & $10^{13}$ & NA 
\\
\bottomrule
\end{tabular}
\label{tbl:datasets}
\end{table}

{\bf Objective functions}

We use $L_2$-regularized linear regression
$\Loss_{\text{\textsc{linear}}}$
and $L_2$-regularized logistic regression
$\Loss_{\text{\textsc{logistic}}}(\vw)$,
with a regularization coefficient of $1$. 
Given a data matrix $\mX \in \R^{n\times d}$, 
target $y \in \R^n$ for regression tasks 
and $y \in \{0,1\}^n$ for classification tasks,
and parameters $\vw \in \R^d$, 
\aligns{
  \Loss_{\text{\textsc{linear}}}(\vw)
  &= \frac{1}{n} \paren{
    \frac{1}{2} \norm{\mX \vw - \vy}^2 + \frac{1}{2}\norm{\vw}^2
  }.
  \\
  \Loss_{\text{\textsc{logistic}}}(\vw)
  &= \frac{1}{n}\sum_{i=1}^n 
  -\vy[i] \log(\sigma(\lin{\vx_i, \vw}))
  - (1-\vy[i]) \log(1-\sigma(\lin{\vx_i, \vw}))
  + \frac{1}{n}\frac{1}{2}\norm{\vw}^2.
}
where $\vx_i$ is the $i$th row of $\mX$
and $\sigma$ is the sigmoid function, 
$\sigma(z) = \nicefrac{1}{1+\exp(-z)}$.
For all datasets, we add a bias term 
by prepending a feature column of ones to $\mX$.

{\bf Datasets}

We use the datasets listed in \cref{tbl:datasets},
made available by LIBSVM \citep{chang2011libsvm},
Scikit-Learn \citep{pedregosa2011scikitlearn}
and the UCI repository \citep{dua2019uci}.

{\bf Data rescaling}

We do not rescale, standardize or otherwise change any of the datasets
beyond adding a bias term,
as our goal is to check whether preconditioned methods
can handle badly scaled data.

{\bf Initializations}

We consider two types of initializations. 
The first approximates a ``best-case'' scenario 
where we start from an estimate with a reasonable loss value 
despite the bad scaling of the data.
We set $\vw[i] = 0$ except for the bias term $\vw[0]$
which is set at the MLE of the non-regularized problem, 
\aligns{
  \vw[0] &= 
  \textstyle
  \bar y 
  &&
  \textstyle
  \text{ where } \bar y = \frac{1}{n} \sum_{i=1}^n \vy[i] 
  &&
  \text{ for linear regression},\\
  \vw[0] &= 
  \textstyle
  \log\paren{\frac{\bar y}{1-\bar y}}
  &&
  \textstyle
  \text{ where } \bar y = \frac{1}{n} \sum_{i=1}^n \vy[i] 
  &&
  \text{ for logistic regression.}
}
The results in the main text use this initialization. 
The second initialization takes $\vw \sim \Normal{0, \mI}$, 
giving a starting point with potentially large loss.
We give results using both initializations in the appendix. 

\clearpage
{\bf Optimizers used}
\begin{itemize}[leftmargin=1em]
  \item For the small linear regression problems, 
  we use preconditioned gradient descent 
  with the optimal preconditioner, 
  pre-computed using the semidefinite formulation 
  of \citet{qu2022optimal},
  solved using CVXPY \citep{diamond2016cvxpy}
  based on  the Matlab implementation of \citeauthor{qu2022optimal}.

  \url{https://github.com/Gwzwpxz/opt_dpcond}

  \item Gradient descent with a backtracking line-search 
  with backtracking parameter $\gamma = \nicefrac{1}{2}$.
  
  \item 
  RPROP \citep{riedmiller1993rprop}  
  following the implementation 
  and default hyperparameters 
  in PyTorch
  \citep{paszke2019pytorch}
  (starting step-size of $10^{-1}$, 
  increase step-size factor $\eta^+ = 1.2$, 
  decreasing step-size factor $\eta^- = 0.5$, 
  minimum step-size of $10^{-6}$ 
  and maximum step-size of $50$).

  \url{https://github.com/pytorch/pytorch/blob/v2.0.1/torch/optim/rprop.py}
  
  \item Hypergradient descent to set the step-size, 
  using (S)GD-HD 
  \citep[the multiplicative variant, ][]{baydin2018online}.
  The hypergradient step-size is set to the default
  $\beta = 0.02$ \citep[][footnote 3]{baydin2018online}.
  The initial step-size 
  is set to $\alpha_0 = 10^{-10}$,
  as otherwise most runs diverged immediately.

  \item 
  The diagonal Barzilai-Borwein method 
  of \citet{park2020variable},
  using their non-monotonic line-search.
  We use the default parameters suggested;
  a starting step-size of $10^{-6}$,
  regularization factor on the previous diagonal approximation 
  $\mu = 10^{-6}$, 
  a backtracking factor of $\nicefrac{1}{2}$
  for the backtracking line-search and
  a window of $15$ steps for the non-monotone line-search.
  This line-search does not use a forward step 
  as the update can increase the preconditioner.

  \item Preconditioned gradient descent 
  using the diagonal Hessian,
  with a backtracking line-search.

  \item AdaGrad \citep{duchi2011adagrad}
  but augmented with a backtracking line-search 
  as suggested by \citet{vaswani2020adaptive}
  to make it competitive in the deterministic setting,
  following the PyTorch \citep{paszke2019pytorch} implementation.

  \url{https://github.com/pytorch/pytorch/blob/v2.0.1/torch/optim/adagrad.py}

\end{itemize}

{\bf Line-search and forward steps}

For all methods, 
the backtracking line-search is augmented by a forward step.
When a step-size is accepted, 
it is increased by a factor of $1.1$ for the next step.
For multidimensional backtracking, 
we increase the set uniformly,
taking $\vb' = 1.1 \cdot \vb$ for the box
and $\smash{\va' = \va / \sqrt{1.1}}$ for the ellipsoid.
The ellipsoid uses a slightly smaller increase factor.\footnote{
To increase by a factor of $1.1$ in the one-dimensional case, 
the update to the ellipsoid should be $\smash{\va' = \va / 1.1^2}$.
}

{\bf Hyperparameters for the line-search and multidimensional backtracking}

For the backtracking line-searches used in gradient descent, preconditioned gradient descent 
and used to augment the other algorithms,
we start the search at an initial step-size of $10^{10}$
and backtrack by a factor of $\nicefrac{1}{2}$
when failing the Armijo condition, 
implemented generically as 
\aligns{
  \textstyle
  f(\vx - \vd) \leq f(\vx) - \frac{1}{2}\lin{\nabla f(\vx), \vd}
}
For multidimensional backtracking, 
we initialize the sets 
such that the first preconditioner is on the order of $10^{10} \mI$.
Using the notation of \cref{apx:pseudocode}, 
we use the scaling factor $c_0 = d \cdot 10^{10}$ for the box variant 
and $c_0 = \smash{\sqrt{d} \cdot 10^{10}}$ for the ellipsoid variant.
The first preconditioner tried by the box variant with backtracking factor $\gamma = \nicefrac{1}{2d}$
is then $\nicefrac{1}{2} \cdot  10^{10} \mI$, 
and the first preconditioner tried by the ellipsoid variant
(assuming the gradient is uniform, $\nabla f(\vx_0) \propto \vone$)
is $\nicefrac{1}{\sqrt{2}} \cdot 10^{10} \mI$.

\subsection{Additional results}
Figures \ref{fig:additional-plots-first}--\ref{fig:additional-plots-last}
give additional results 
on small linear and logistic regression problems 
and large logistic regression problems. 
Multidimensional backtracking has a consistent performance 
across problems and 
does not suffer from the extremely bad conditioning 
of cpusmall or california-housing (linear regression)
or australian, breast-cancer, diabetes and heart (logistic regression).

\clearpage
\begin{figure}[t]
\caption{Runs on small linear regression datasets with Bias initialization}
\label{fig:additional-plots-first}
\includegraphics{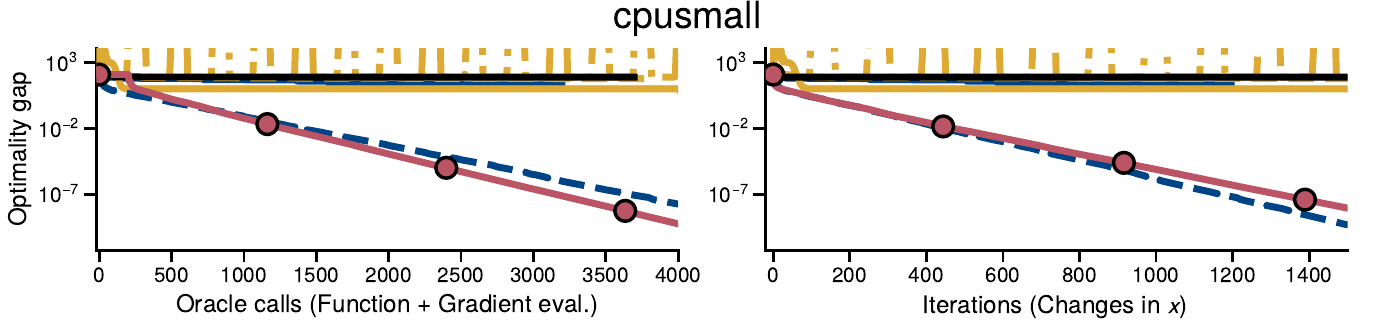}
\includegraphics{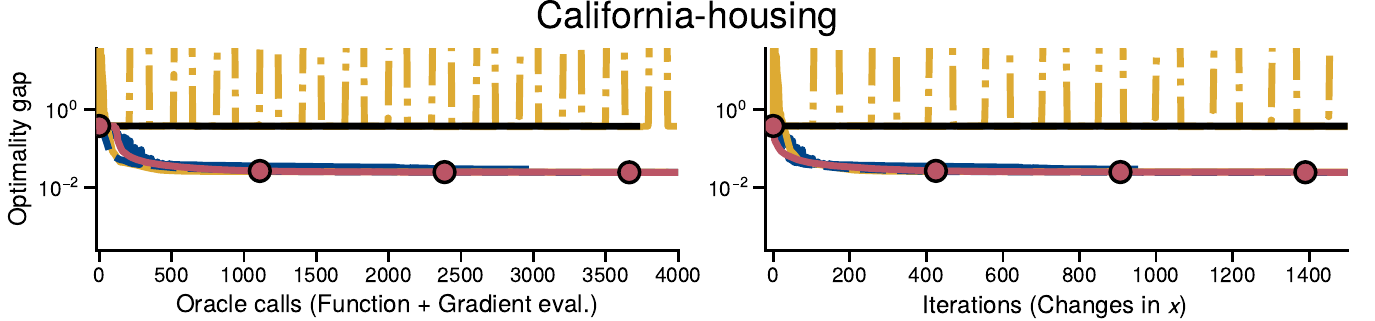}
\includegraphics{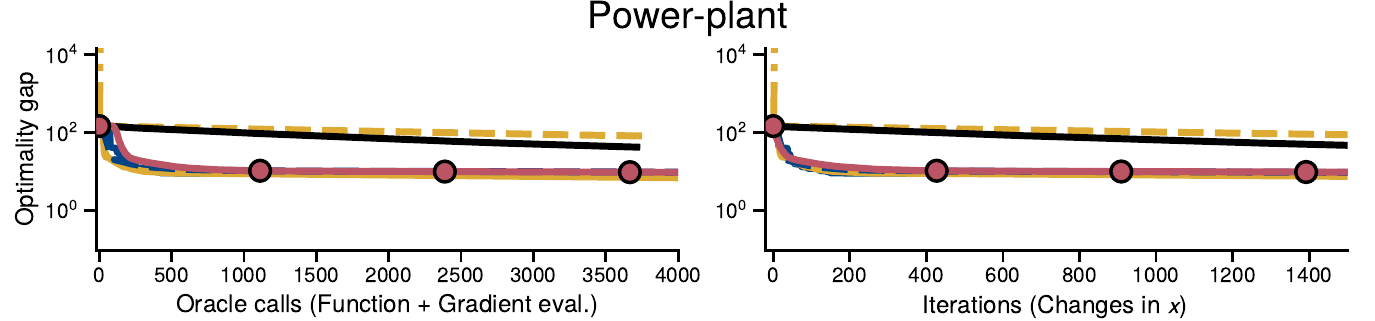}
\includegraphics{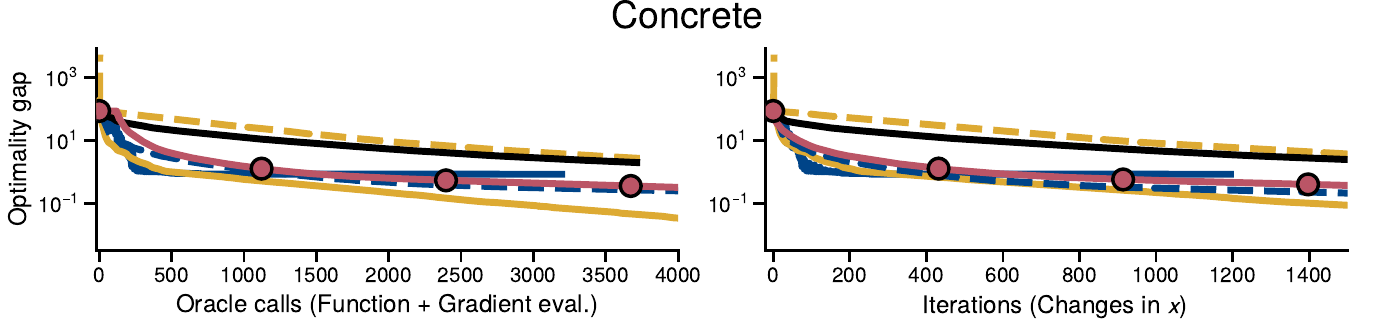}
\includegraphics{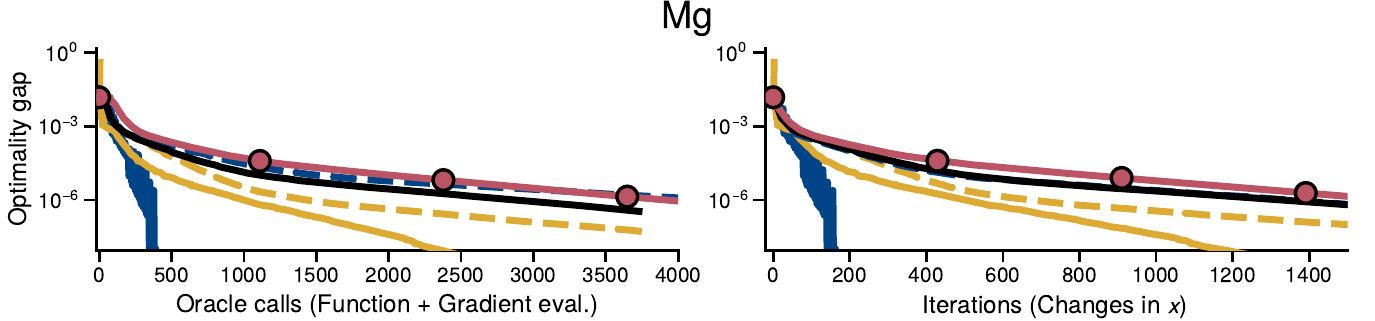}
\includegraphics{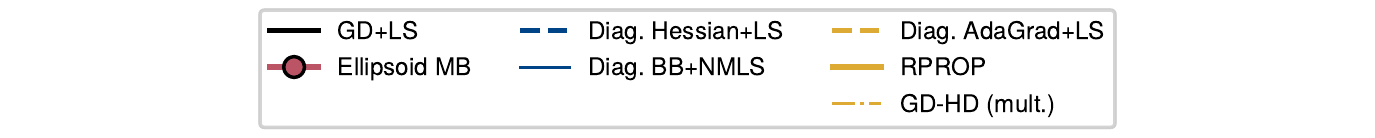}
\end{figure}

\begin{figure}[t]
\caption{Runs on small linear regression datasets with Gaussian initialization}
\includegraphics{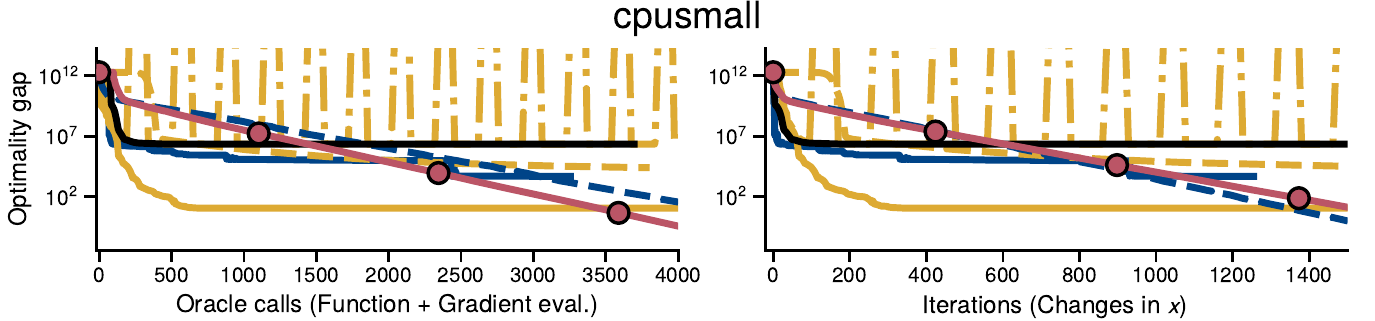}
\includegraphics{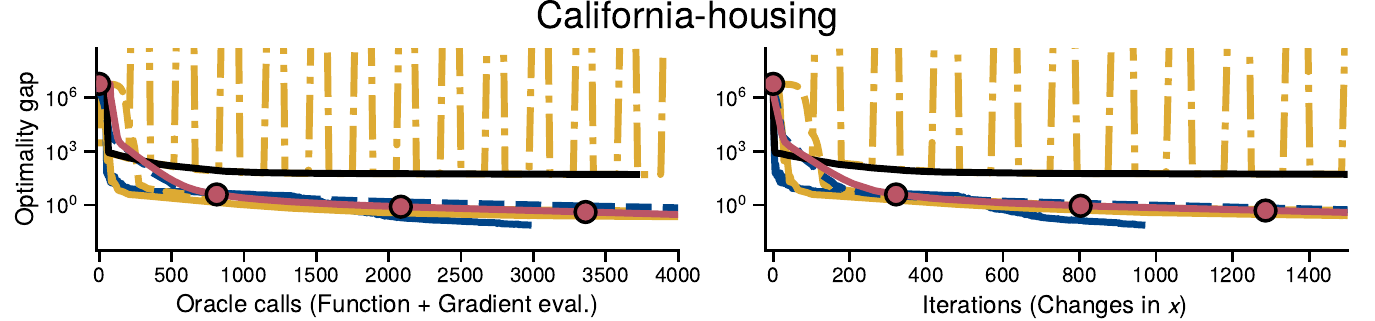}
\includegraphics{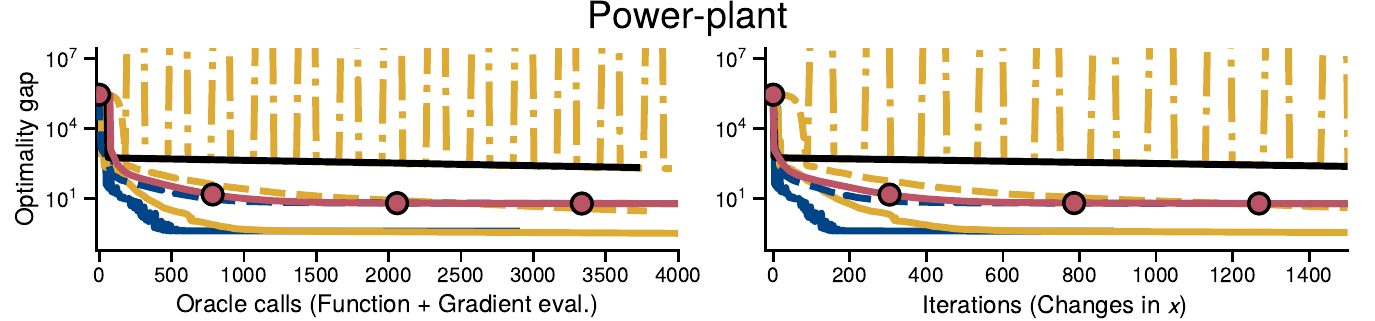}
\includegraphics{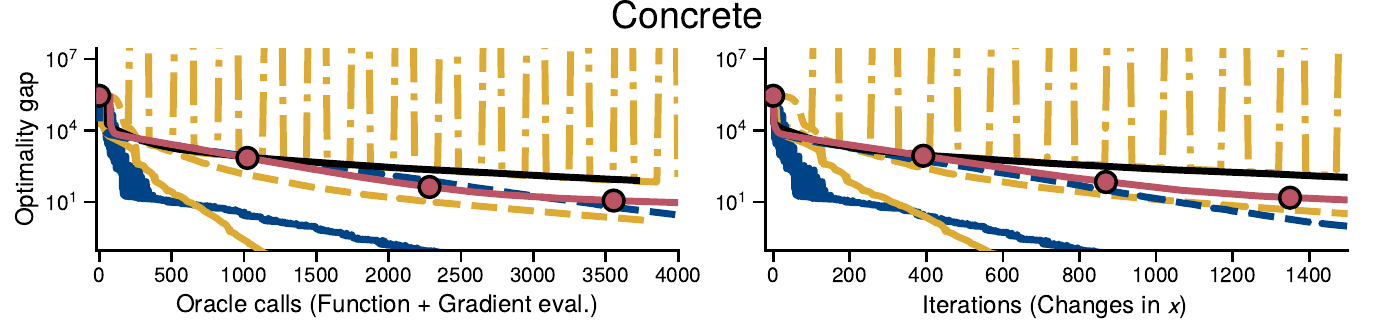}
\includegraphics{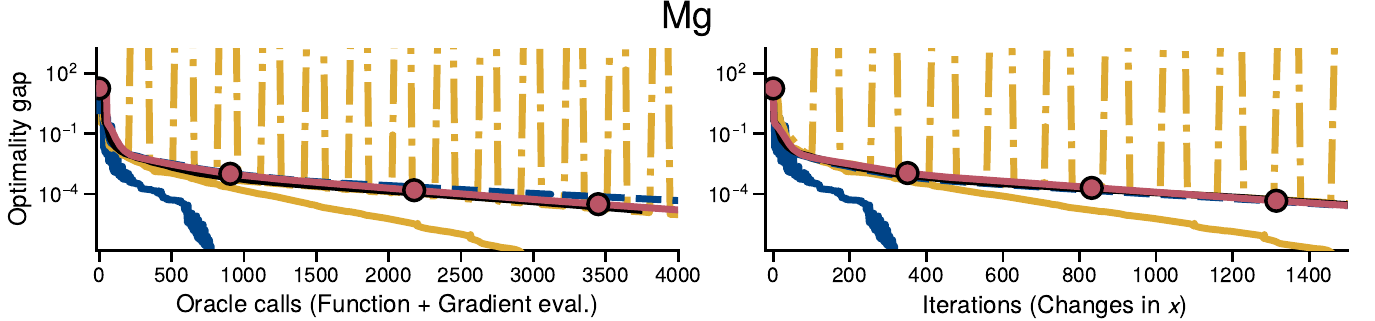}
\includegraphics{fig/legend.pdf}
\end{figure}

\begin{figure}[t]
\caption{Runs on small logistic regression datasets with Bias initialization}
\includegraphics{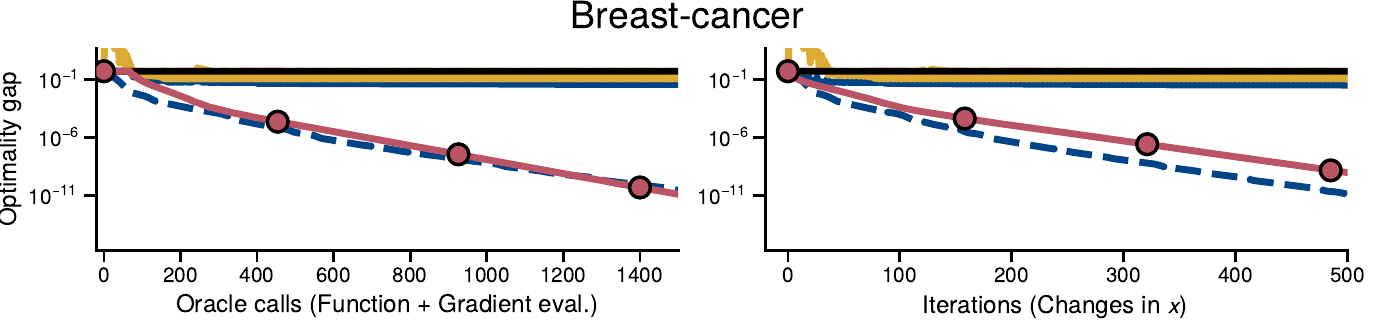}
\includegraphics{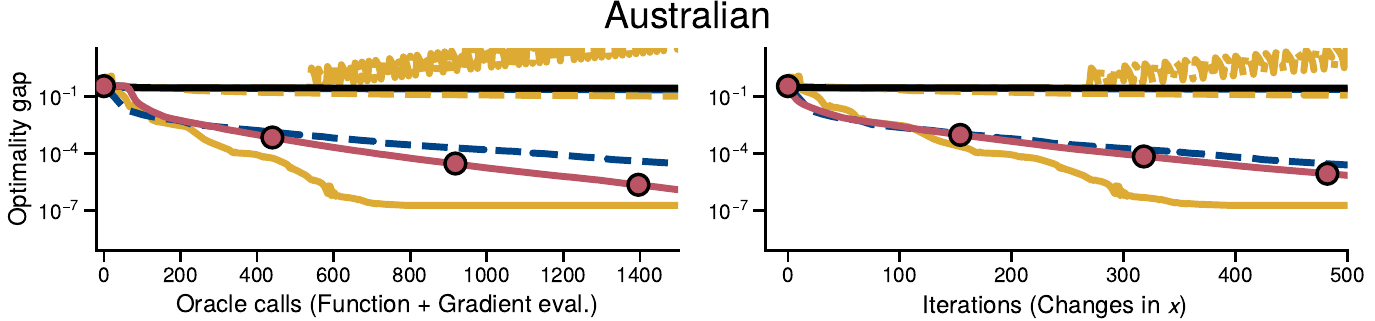}
\includegraphics{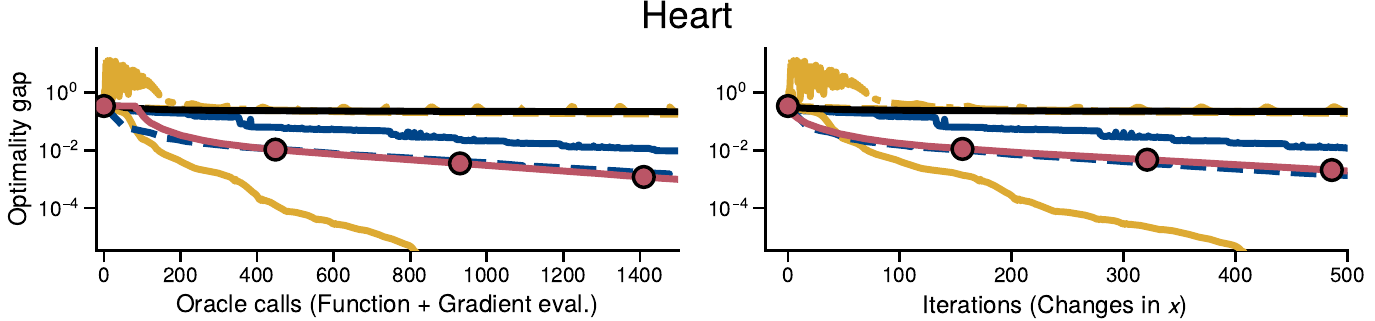}
\includegraphics{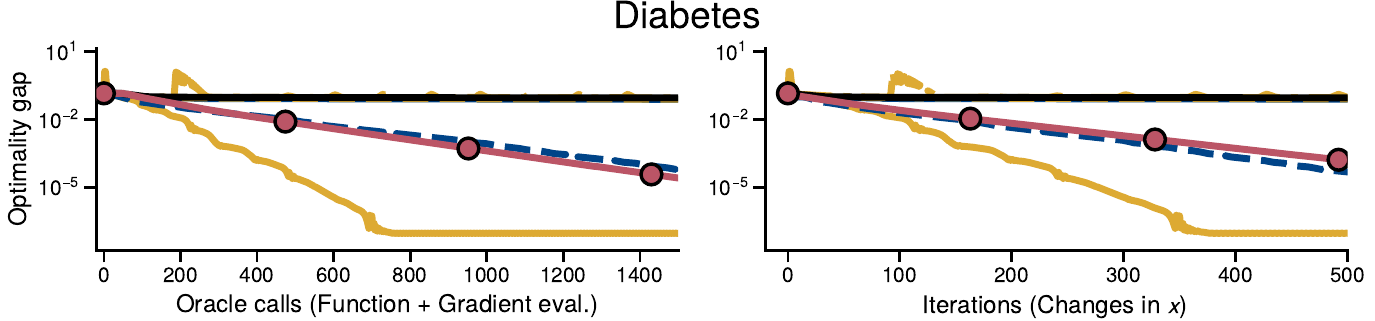}
\includegraphics{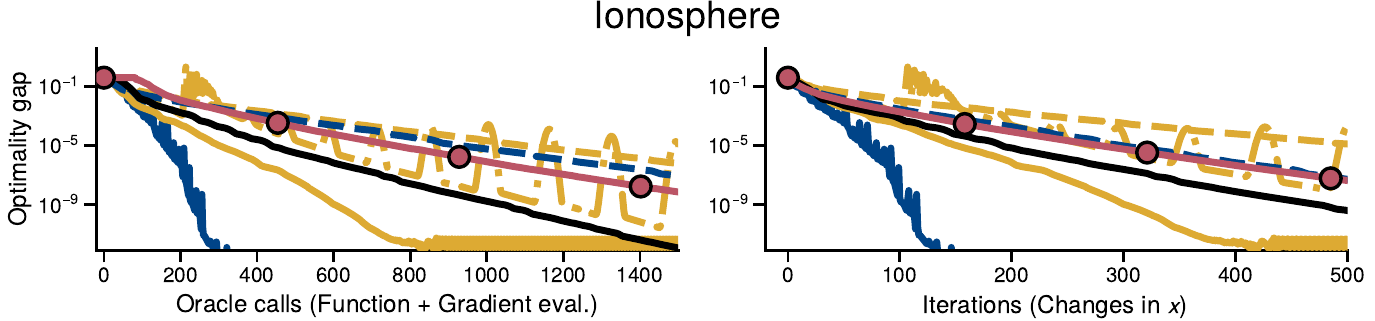}
\includegraphics{fig/legend.pdf}
\end{figure}

\begin{figure}[t]
\caption{Runs on small logistic regression datasets with Gaussian initialization}
\includegraphics{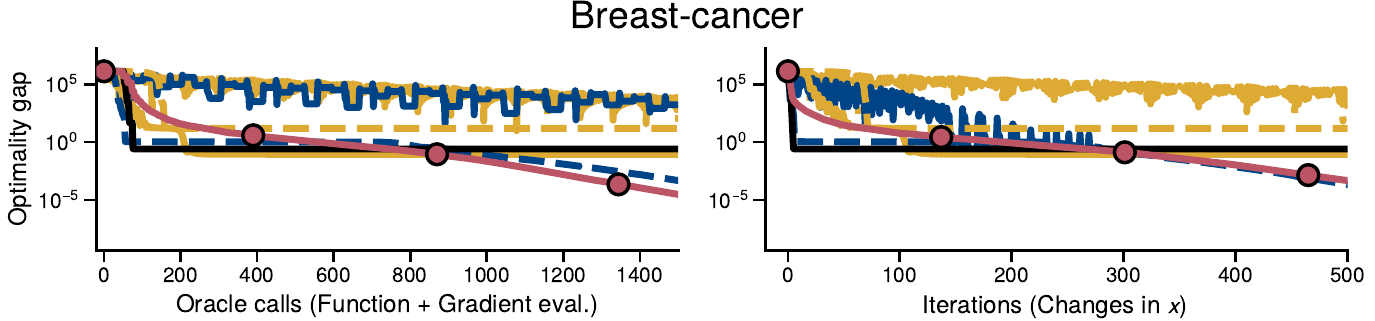}
\includegraphics{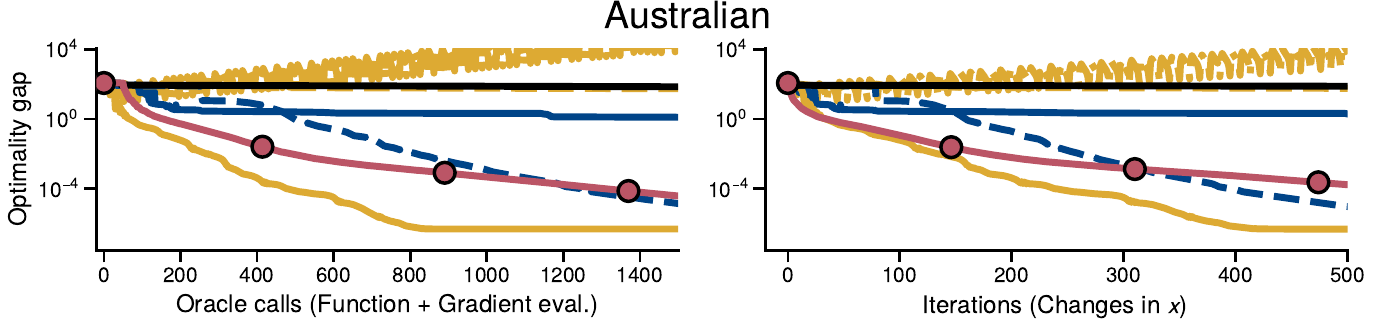}
\includegraphics{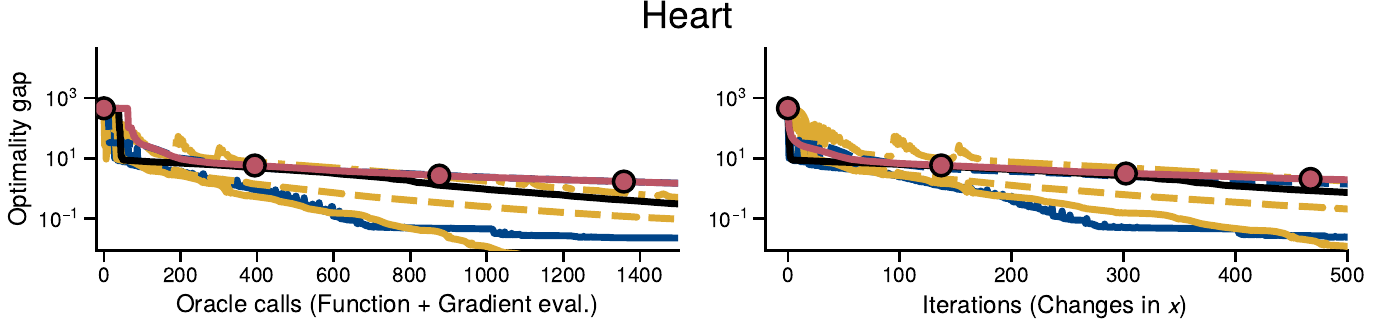}
\includegraphics{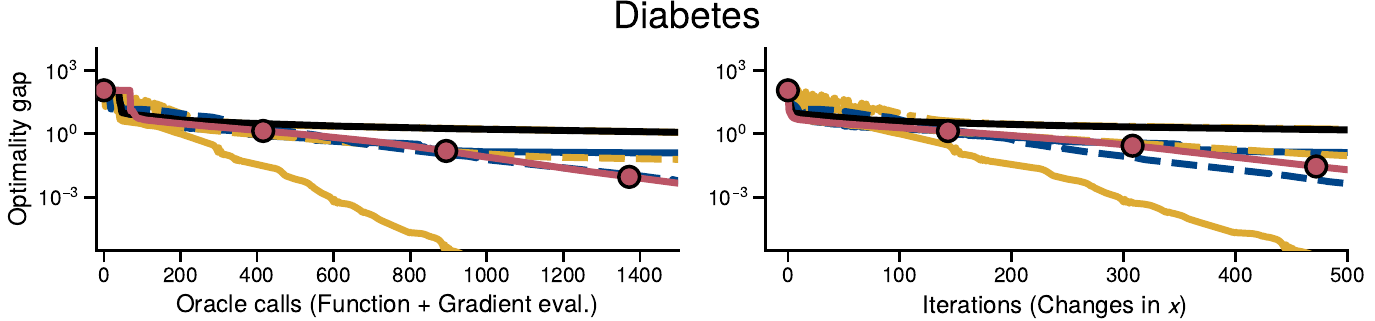}
\includegraphics{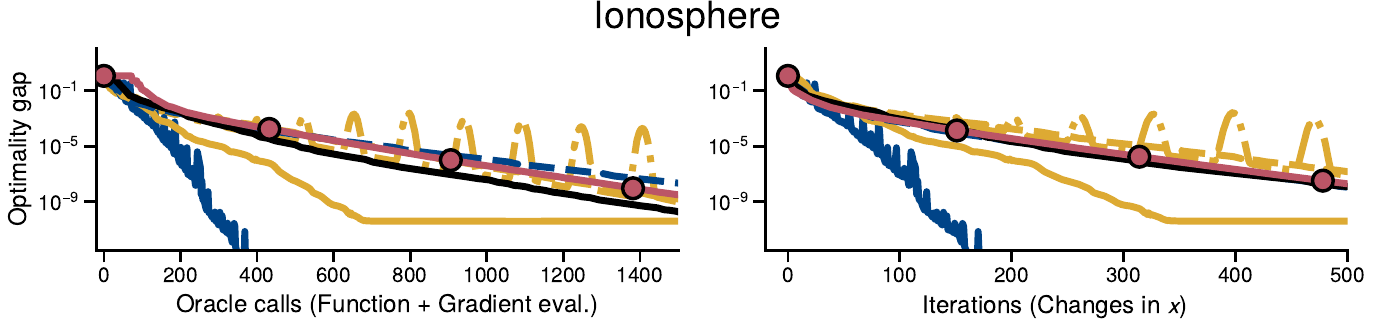}
\includegraphics{fig/legend.pdf}
\end{figure}

\begin{figure}[t]
\caption{Runs on large logistic regression datasets with Bias initialization}
\includegraphics{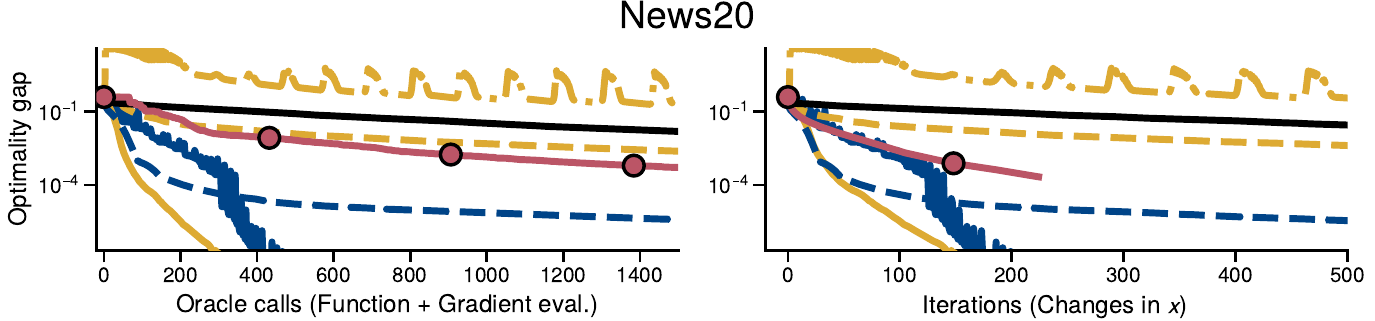}
\includegraphics{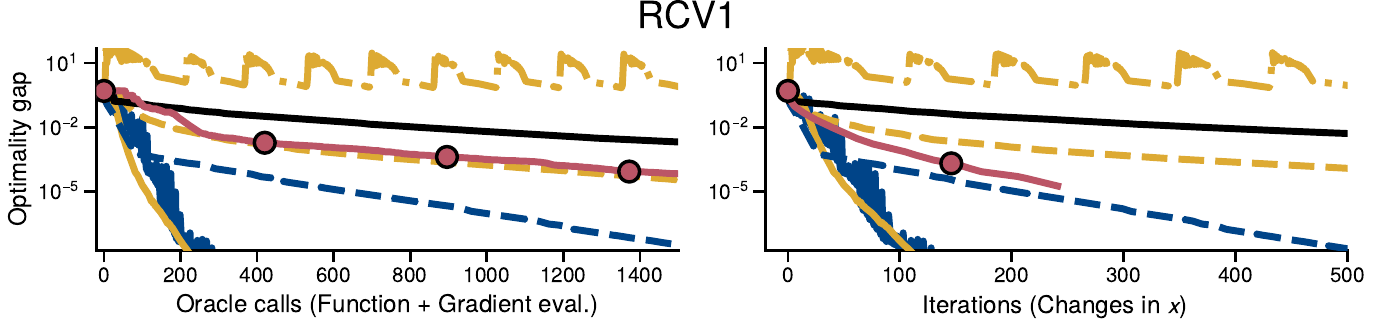}
\includegraphics{fig/legend.pdf}
\end{figure}

\begin{figure}[t]
\caption{Runs on large logistic regression datasets with Gaussian initialization}
\label{fig:additional-plots-last}
\includegraphics{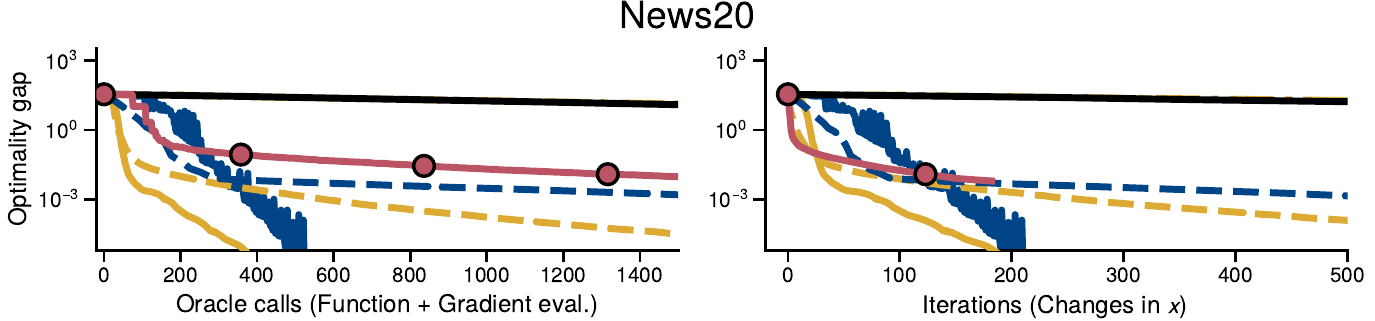}
\includegraphics{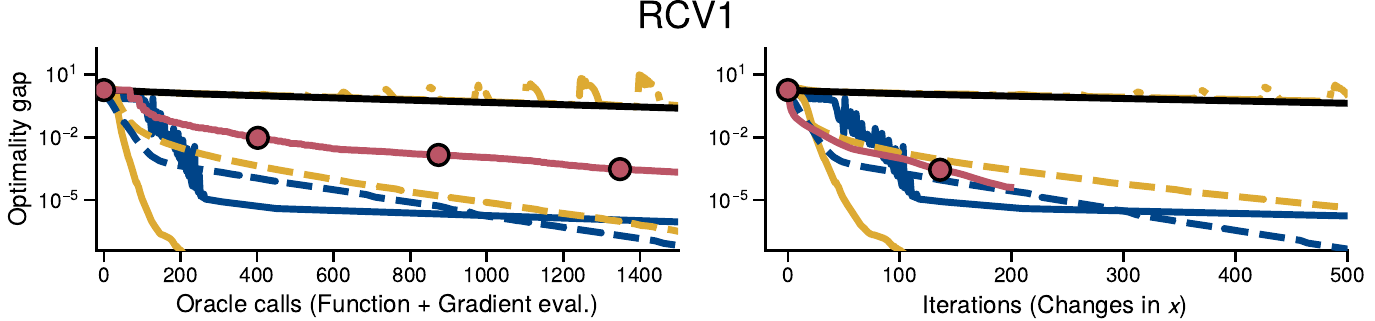}
\includegraphics{fig/legend.pdf}
\end{figure}

\end{document}